\documentclass[10pt]{amsart}

\usepackage{amsfonts,amssymb}
\usepackage{hyperref}
\usepackage{xcolor}
\usepackage[title]{appendix}

\input{xy}
\xyoption{all}

\newtheorem{proposition}{Proposition}[section]
\newtheorem{lemma}[proposition]{Lemma}
\newtheorem{corollary}[proposition]{Corollary}
\newtheorem{theorem}[proposition]{Theorem}
\theoremstyle{definition}
\newtheorem{definition}[proposition]{Definition}

\newtheorem{example}[proposition]{Example}
\newtheorem{examples}[proposition]{Examples}
\theoremstyle{remark}
\newtheorem{remark}[proposition]{Remark}
\newtheorem{remarks}[proposition]{Remarks}\newcommand{\thlabel}[1]{\label{th:#1}}
\newcommand{\thref}[1]{Theorem~\ref{th:#1}}
\newcommand{\selabel}[1]{\label{se:#1}}
\newcommand{\seref}[1]{Section~\ref{se:#1}}
\newcommand{\sselabel}[1]{\label{sse:#1}}
\newcommand{\sseref}[1]{Subsection~\ref{sse:#1}}
\newcommand{\lelabel}[1]{\label{le:#1}}
\newcommand{\leref}[1]{Lemma~\ref{le:#1}}
\newcommand{\prlabel}[1]{\label{pr:#1}}
\newcommand{\prref}[1]{Proposition~\ref{pr:#1}}
\newcommand{\colabel}[1]{\label{co:#1}}
\newcommand{\coref}[1]{Corollary~\ref{co:#1}}

\newcommand{\exlabel}[1]{\label{ex:#1}}
\newcommand{\exref}[1]{Example~\ref{ex:#1}}
\newcommand{\delabel}[1]{\label{de:#1}}
\newcommand{\deref}[1]{Definition~\ref{de:#1}}

\newcommand{\eqlabel}[1]{\label{eq:#1}}
\newcommand{\equref}[1]{(\ref{eq:#1})}

\def\ra{\rightarrow}
\def\cd{\cdot}
\def\Id{{\rm Id}}

\def\mfg{\mf {g}}

\newcommand{\alp}{\sigma}
\newcommand{\be}{\psi}
\def\ot{\otimes}
\def\va{\varepsilon}

\def\un{\underline}

\def\mf{\mathfrak}

\def\mfq{\mf {q}}
\def\le{\langle}
\def\ri{\rangle}
\def\l{\lambda}

\def\r{\rho}

\def\tie{\mathrel>\joinrel\mathrel\triangleleft}

\def\va{\varepsilon}
\def\v{\varphi}

\def\ra{\rightarrow}
\def\a{\alpha}
\def\b{\beta}

\def\d{\delta}

\def\ov{\overline}
\def\cal{\mathcal}
\def\un{\underline}

\def\equal#1{\smash{\mathop{=}\limits^{#1}}}

\allowdisplaybreaks[4]

 \newcount\grcalca
 \newcount\grcalcb
 \newcount\grcalcc
 \newcount\grcalcd
 \newcount\grcalce
 \newcount\grcalcf
 \newcount\grcalcg
 \newcount\grcalch
 \newcount\grcalci
 \newcount\grrow
 \newcount\grcolumn
 \newcount\grwidth
 \newcount\factor
 \newcount\hfactor
 \newcount\qfactor
 \newcount\tfactor
 \newcount\sfactor
 \newcount\dfactor
 \hfactor = \factor
 \divide    \hfactor by 2
 \qfactor = \hfactor
 \divide    \qfactor by 2
 \tfactor = \factor
 \divide    \tfactor by 3
 \sfactor = \factor
 \divide    \sfactor by 6
 \dfactor = \factor
 \divide    \dfactor by 12

 \newcommand{\gbeg}[2]{
   \unitlength=1pt
   \grrow = #2
   \grcolumn = 0
   \grcalca = #1
   \grcalcb = #2
   \multiply \grcalca by \factor
   \grwidth = \grcalca
   \multiply \grcalcb by \factor
   \begin{minipage}{\grcalca pt}
   \begin{picture}(\grcalca,\grcalcb)
   \advance \grcalcb by -\factor
   \put(0, \grcalcb){\line(1,0){\grwidth}} }

 \newcommand{\gend}{
   \put(0, \factor){\line(1,0){\grwidth}}
   \end{picture}
   {\vskip2.5ex}
   \end{minipage} }

 \newcommand{\gnl}{
   \advance \grrow by -1
   \grcolumn = 0}

 \newcommand{\gvac}[1]{       
   \advance \grcolumn by #1} 
 
 \newcommand{\gcl}[1]{
   \grcalca = \grcolumn
   \multiply \grcalca by \factor
   \advance \grcalca by \hfactor
   \grcalcb = \grrow
   \multiply \grcalcb by \factor
   \grcalcc = #1
   \multiply \grcalcc by \factor
   \put(\grcalca,\grcalcb) {\line(0,-1){\grcalcc}} 
   \advance \grcolumn by 1}

 \newcommand{\gcn}[4]{
   \grcalca = \grcolumn
   \multiply \grcalca by \factor
   \grcalci = #3
   \multiply \grcalci by \hfactor
   \advance \grcalca by \grcalci
   \grcalcb = \grcolumn
   \multiply \grcalcb by \factor 
   \grcalci = #3
   \advance \grcalci by #4
   \multiply \grcalci by \qfactor
   \advance \grcalcb by \grcalci
   \grcalcc = \grcolumn
   \multiply \grcalcc by \factor
   \grcalci = #4
   \multiply \grcalci by \hfactor
   \advance \grcalcc by \grcalci
   \grcalcd = \grrow
   \multiply \grcalcd by \factor 
   \grcalce = \grrow
   \multiply \grcalce by \factor 
   \grcalci = #2
   \multiply \grcalci by \tfactor
   \advance \grcalce by -\grcalci
   \grcalcf = \grrow
   \multiply \grcalcf by \factor 
   \grcalci = #2
   \multiply \grcalci by \hfactor
   \advance \grcalcf by -\grcalci
   \grcalcg = \grrow
   \multiply \grcalcg by \factor 
   \grcalci = #2
   \multiply \grcalci by \tfactor
   \multiply \grcalci by 2
   \advance \grcalcg by -\grcalci
   \grcalch = \grrow
   \advance \grcalch by -#2
   \multiply \grcalch by \factor 
   \qbezier(\grcalca,\grcalcd)(\grcalca,\grcalce)(\grcalcb,\grcalcf) 
   \qbezier(\grcalcb,\grcalcf)(\grcalcc,\grcalcg)(\grcalcc,\grcalch) 
   \advance \grcolumn by #1}

 \newcommand{\gnot}[1]{
   \grcalca = \grcolumn
   \multiply \grcalca by \factor
   \advance \grcalca by \hfactor
   \grcalcb = \grrow
   \multiply \grcalcb by \factor
   \advance \grcalcb by -\hfactor
   \put(\grcalca,\grcalcb) {\makebox(0,0){$\scriptstyle #1$}} }

 \newcommand{\got}[2]{
   \grcalca = \grcolumn
   \multiply \grcalca by \factor
   \grcalcc = #1
   \multiply \grcalcc by \hfactor
   \advance \grcalca by \grcalcc
   \grcalcb = \grrow
   \multiply \grcalcb by \factor
   \advance \grcalcb by -\tfactor
   \advance \grcalcb by -\tfactor
   \put(\grcalca,\grcalcb){\makebox(0,0)[b]{$#2$}}
   \advance \grcolumn by #1}

 \newcommand{\gob}[2]{
   \grcalca = \grcolumn
   \multiply \grcalca by \factor
   \grcalcc = #1
   \multiply \grcalcc by \hfactor
   \advance \grcalca by \grcalcc
   \put(\grcalca,0){\makebox(0,0)[b]{$#2$}}
   \advance \grcolumn by #1}

 \newcommand{\gmu}{  
   \grcalca = \grcolumn
   \advance \grcalca by 1
   \multiply \grcalca by \factor
   \grcalcb = \grrow
   \multiply \grcalcb by \factor
   \grcalcc = \factor
   \advance \grcalcc by \hfactor
   \put(\grcalca,\grcalcb){\oval(\factor,\grcalcc)[b]}
   \advance \grcalcb by -\hfactor
   \advance \grcalcb by -\qfactor
   \put(\grcalca,\grcalcb) {\line(0,-1){\qfactor}} 
   \advance \grcolumn by 2}

 \newcommand{\gcmu}{   
   \grcalca = \grcolumn
   \advance \grcalca by 1
   \multiply \grcalca by \factor
   \grcalcb = \grrow
   \advance \grcalcb by -1
   \multiply \grcalcb by \factor
   \grcalcc = \factor
   \advance \grcalcc by \hfactor
   \put(\grcalca,\grcalcb){\oval(\factor,\grcalcc)[t]}
   \advance \grcalcb by \factor
   \put(\grcalca,\grcalcb) {\line(0,-1){\qfactor}} 
   \advance \grcolumn by 2}

 \newcommand{\glm}{
   \grcalca = \grcolumn
   \multiply \grcalca by \factor
   \advance \grcalca by \hfactor
   \grcalcb = \grcalca
   \advance \grcalcb by \factor
   \grcalcc = \grrow
   \multiply \grcalcc by \factor
   \grcalcd = \grcalcc
   \advance \grcalcd by -\tfactor
   \grcalce = \grcalcd
   \advance \grcalce by -\tfactor
   \put(\grcalca, \grcalcc){\line(0,-1){\tfactor}}
   \put(\grcalca, \grcalcd){\line(1,0){\factor}}
   \put(\grcalca, \grcalcd){\line(3,-1){\factor}}
   \put(\grcalcb, \grcalcc){\line(0,-1){\factor}}
   \advance \grcolumn by 2}

 \newcommand{\grm}{
   \grcalcb = \grcolumn
   \multiply \grcalcb by \factor
   \advance \grcalcb by \hfactor
   \grcalca = \grcalcb
   \advance \grcalca by \factor
   \grcalcc = \grrow
   \multiply \grcalcc by \factor
   \grcalcd = \grcalcc
   \advance \grcalcd by -\tfactor
   \grcalce = \grcalcd
   \advance \grcalce by -\tfactor
   \put(\grcalca, \grcalcc){\line(0,-1){\tfactor}}
   \put(\grcalca, \grcalcd){\line(-1,0){\factor}}
   \put(\grcalca, \grcalcd){\line(-3,-1){\factor}}
   \put(\grcalcb, \grcalcc){\line(0,-1){\factor}}
   \advance \grcolumn by 2}

 \newcommand{\glcm}{
   \grcalca = \grcolumn
   \multiply \grcalca by \factor
   \advance \grcalca by \hfactor
   \grcalcb = \grcalca
   \advance \grcalcb by \factor
   \grcalcc = \grrow
   \advance \grcalcc by -1
   \multiply \grcalcc by \factor
   \grcalcd = \grcalcc
   \advance \grcalcd by \tfactor
   \grcalce = \grcalcd
   \advance \grcalce by \tfactor
   \put(\grcalca, \grcalcc){\line(0,1){\tfactor}}
   \put(\grcalca, \grcalcd){\line(1,0){\factor}}
   \put(\grcalca, \grcalcd){\line(3,1){\factor}}
   \put(\grcalcb, \grcalcc){\line(0,1){\factor}}
   \advance \grcolumn by 2}

 \newcommand{\grcm}{
   \grcalcb = \grcolumn
   \multiply \grcalcb by \factor
   \advance \grcalcb by \hfactor
   \grcalca = \grcalcb
   \advance \grcalca by \factor
   \grcalcc = \grrow
   \advance \grcalcc by -1
   \multiply \grcalcc by \factor
   \grcalcd = \grcalcc
   \advance \grcalcd by \tfactor
   \grcalce = \grcalcd
   \advance \grcalce by \tfactor
   \put(\grcalca, \grcalcc){\line(0,1){\tfactor}}
   \put(\grcalca, \grcalcd){\line(-1,0){\factor}}
   \put(\grcalca, \grcalcd){\line(-3,1){\factor}}
   \put(\grcalcb, \grcalcc){\line(0,1){\factor}}
   \advance \grcolumn by 2}

 \newcommand{\gwmu}[1]{    
   \grcalca = \grcolumn
   \multiply \grcalca by \factor
   \grcalcd = \hfactor
   \multiply \grcalcd by #1
   \advance \grcalca by \grcalcd
   \grcalcb = \grrow
   \multiply \grcalcb by \factor
   \grcalcc = \factor
   \advance \grcalcc by \hfactor
   \grcalcd = #1
   \advance \grcalcd by -1
   \multiply \grcalcd by \factor
   \put(\grcalca,\grcalcb){\oval(\grcalcd,\grcalcc)[b]}
   \advance \grcalcb by -\hfactor
   \advance \grcalcb by -\qfactor
   \put(\grcalca,\grcalcb) {\line(0,-1){\qfactor}} 
   \advance \grcolumn by #1}

 \newcommand{\gwcm}[1]{   
   \grcalca = \grcolumn
   \multiply \grcalca by \factor
   \grcalcd = \hfactor
   \multiply \grcalcd by #1
   \advance \grcalca by \grcalcd
   \grcalcb = \grrow
   \advance \grcalcb by -1
   \multiply \grcalcb by \factor
   \grcalcc = \factor
   \advance \grcalcc by \hfactor
   \grcalcd = #1
   \advance \grcalcd by -1
   \multiply \grcalcd by \factor
   \put(\grcalca,\grcalcb){\oval(\grcalcd,\grcalcc)[t]}
   \advance \grcalcb by \factor
   \put(\grcalca,\grcalcb) {\line(0,-1){\qfactor}} 
   \advance \grcolumn by #1}

 \newcommand{\gwmuc}[1]{    
   \grcalca = \grcolumn
   \multiply \grcalca by \factor
   \advance \grcalca by \hfactor
   \grcalcb = \grrow
   \multiply \grcalcb by \factor
   \grcalcc = #1
   \advance \grcalcc by -1
   \multiply \grcalcc by \factor
   \put(\grcalca,\grcalcb){\line(1,0){\grcalcc}}
   \advance \grcalca by -\hfactor
   \grcalcd = \hfactor
   \multiply \grcalcd by #1
   \advance \grcalca by \grcalcd
   \grcalcc = \factor
   \advance \grcalcc by \hfactor
   \grcalcd = #1
   \advance \grcalcd by -1
   \multiply \grcalcd by \factor
   \put(\grcalca,\grcalcb){\oval(\grcalcd,\grcalcc)[b]}
   \advance \grcalcb by -\hfactor
   \advance \grcalcb by -\qfactor
   \put(\grcalca,\grcalcb) {\line(0,-1){\qfactor}} 
   \advance \grcolumn by #1}

 \newcommand{\gwcmc}[1]{   
   \grcalca = \grcolumn
   \multiply \grcalca by \factor
   \advance \grcalca by \hfactor
   \grcalcb = \grrow
   \multiply \grcalcb by \factor
   \advance \grcalcb by -\factor
   \grcalcc = #1
   \advance \grcalcc by -1
   \multiply \grcalcc by \factor
   \put(\grcalca,\grcalcb){\line(1,0){\grcalcc}}
   \grcalcd = #1
   \advance \grcalcd by -1
   \multiply \grcalcd by \hfactor
   \advance \grcalca by \grcalcd
   \grcalcc = \factor
   \advance \grcalcc by \hfactor
   \grcalcd = #1
   \advance \grcalcd by -1
   \multiply \grcalcd by \factor
   \put(\grcalca,\grcalcb){\oval(\grcalcd,\grcalcc)[t]}
   \advance \grcalcb by \factor
   \put(\grcalca,\grcalcb) {\line(0,-1){\qfactor}} 
   \advance \grcolumn by #1}

 \newcommand{\gev}{  
   \grcalca = \grcolumn
   \advance \grcalca by 1
   \multiply \grcalca by \factor
   \grcalcb = \grrow
   \multiply \grcalcb by \factor
   \grcalcc = \factor
   \advance \grcalcc by \hfactor
   \put(\grcalca,\grcalcb){\oval(\factor,\grcalcc)[b]}
   \advance \grcolumn by 2}

 \newcommand{\gdb}{   
   \grcalca = \grcolumn
   \advance \grcalca by 1
   \multiply \grcalca by \factor
   \grcalcb = \grrow
   \advance \grcalcb by -1
   \multiply \grcalcb by \factor
   \grcalcc = \factor
   \advance \grcalcc by \hfactor
   \put(\grcalca,\grcalcb){\oval(\factor,\grcalcc)[t]}
   \advance \grcolumn by 2}

 \newcommand{\gwev}[1]{    
   \grcalca = \grcolumn
   \multiply \grcalca by \factor
   \grcalcd = \hfactor
   \multiply \grcalcd by #1
   \advance \grcalca by \grcalcd
   \grcalcb = \grrow
   \multiply \grcalcb by \factor
   \grcalcc = \factor
   \advance \grcalcc by \hfactor
   \grcalcd = #1
   \advance \grcalcd by -1
   \multiply \grcalcd by \factor
   \put(\grcalca,\grcalcb){\oval(\grcalcd,\grcalcc)[b]}
   \advance \grcolumn by #1}

 \newcommand{\gwdb}[1]{   
   \grcalca = \grcolumn
   \multiply \grcalca by \factor
   \grcalcd = \hfactor
   \multiply \grcalcd by #1
   \advance \grcalca by \grcalcd
   \grcalcb = \grrow
   \advance \grcalcb by -1
   \multiply \grcalcb by \factor
   \grcalcc = \factor
   \advance \grcalcc by \hfactor
   \grcalcd = #1
   \advance \grcalcd by -1
   \multiply \grcalcd by \factor
   \put(\grcalca,\grcalcb){\oval(\grcalcd,\grcalcc)[t]}
   \advance \grcolumn by #1}

 \newcommand{\gbr}{
   \grcalca = \grcolumn
   \multiply \grcalca by \factor
   \advance \grcalca by \hfactor
   \grcalcb = \grcalca
   \advance \grcalcb by \hfactor
   \grcalcc = \grcalca
   \advance \grcalcc by \factor
   \grcalcd = \grrow
   \multiply \grcalcd by \factor
   \grcalce = \grcalcd
   \advance \grcalce by -\tfactor
   \grcalcf = \grcalcd
   \advance \grcalcf by -\hfactor
   \grcalcg = \grcalce
   \advance \grcalcg by -\tfactor
   \grcalch = \grcalcd
   \advance \grcalch by -\factor
   \qbezier(\grcalca,\grcalcd)(\grcalca,\grcalce)(\grcalcb,\grcalcf) 
   \qbezier(\grcalcb,\grcalcf)(\grcalcc,\grcalcg)(\grcalcc,\grcalch) 
   \advance \grcalcf by -\dfactor
   \advance \grcalcb by -\sfactor
   \qbezier(\grcalca,\grcalch)(\grcalca,\grcalcg)(\grcalcb,\grcalcf) 
   \advance \grcalcf by \sfactor
   \advance \grcalcb by \tfactor
   \qbezier(\grcalcc,\grcalcd)(\grcalcc,\grcalce)(\grcalcb,\grcalcf) 
   \advance \grcolumn by 2}

 \newcommand{\gibr}{
   \grcalca = \grcolumn
   \multiply \grcalca by \factor
   \advance \grcalca by \hfactor
   \grcalcb = \grcalca
   \advance \grcalcb by \hfactor
   \grcalcc = \grcalca
   \advance \grcalcc by \factor
   \grcalcd = \grrow
   \multiply \grcalcd by \factor
   \grcalce = \grcalcd
   \advance \grcalce by -\tfactor
   \grcalcf = \grcalcd
   \advance \grcalcf by -\hfactor
   \grcalcg = \grcalce
   \advance \grcalcg by -\tfactor
   \grcalch = \grcalcd
   \advance \grcalch by -\factor
   \qbezier(\grcalcc,\grcalcd)(\grcalcc,\grcalce)(\grcalcb,\grcalcf) 
   \qbezier(\grcalcb,\grcalcf)(\grcalca,\grcalcg)(\grcalca,\grcalch) 
   \advance \grcalcf by -\dfactor
   \advance \grcalcb by \sfactor
   \qbezier(\grcalcc,\grcalch)(\grcalcc,\grcalcg)(\grcalcb,\grcalcf) 
   \advance \grcalcf by \sfactor
   \advance \grcalcb by -\tfactor
   \qbezier(\grcalca,\grcalcd)(\grcalca,\grcalce)(\grcalcb,\grcalcf) 
   \advance \grcolumn by 2}

\newcommand{\gsy}{
   \grcalca = \grcolumn
   \multiply \grcalca by \factor
   \advance \grcalca by \hfactor
   \grcalcb = \grcalca
   \advance \grcalcb by \hfactor
   \grcalcc = \grcalca
   \advance \grcalcc by \factor
   \grcalcd = \grrow
   \multiply \grcalcd by \factor
   \grcalce = \grcalcd
   \advance \grcalce by -\tfactor
   \grcalcf = \grcalcd
   \advance \grcalcf by -\hfactor
   \grcalcg = \grcalce
   \advance \grcalcg by -\tfactor
   \grcalch = \grcalcd
   \advance \grcalch by -\factor
   \qbezier(\grcalcc,\grcalcd)(\grcalcc,\grcalce)(\grcalcb,\grcalcf) 
   \qbezier(\grcalcb,\grcalcf)(\grcalca,\grcalcg)(\grcalca,\grcalch) 
   \advance \grcalcf by -\dfactor
   \advance \grcalcb by \sfactor
   \qbezier(\grcalcc,\grcalch)(\grcalcc,\grcalcg)(\grcalcb,\grcalcf) 
   \qbezier(\grcalca,\grcalcd)(\grcalca,\grcalce)(\grcalcb,\grcalcf) 
   \advance \grcolumn by 2}

 \newcommand{\gbrc}{
   \grcalca = \grcolumn
   \multiply \grcalca by \factor
   \advance \grcalca by \hfactor
   \grcalcb = \grcalca
   \advance \grcalcb by \hfactor
   \grcalcc = \grcalca
   \advance \grcalcc by \factor
   \grcalcd = \grrow
   \multiply \grcalcd by \factor
   \grcalce = \grcalcd
   \advance \grcalce by -\tfactor
   \grcalcf = \grcalcd
   \advance \grcalcf by -\hfactor
   \grcalcg = \grcalce
   \advance \grcalcg by -\tfactor
   \grcalch = \grcalcd
   \advance \grcalch by -\factor
   \put(\grcalcb,\grcalcf){\circle{\hfactor}}
   \qbezier(\grcalca,\grcalcd)(\grcalca,\grcalce)(\grcalcb,\grcalcf) 
   \qbezier(\grcalcb,\grcalcf)(\grcalcc,\grcalcg)(\grcalcc,\grcalch) 
   \advance \grcalcf by -\dfactor
   \advance \grcalcb by -\sfactor
   \qbezier(\grcalca,\grcalch)(\grcalca,\grcalcg)(\grcalcb,\grcalcf) 
   \advance \grcalcf by \sfactor
   \advance \grcalcb by \tfactor
   \qbezier(\grcalcc,\grcalcd)(\grcalcc,\grcalce)(\grcalcb,\grcalcf) 
   \advance \grcolumn by 2}

 \newcommand{\gibrc}{
   \grcalca = \grcolumn
   \multiply \grcalca by \factor
   \advance \grcalca by \hfactor
   \grcalcb = \grcalca
   \advance \grcalcb by \hfactor
   \grcalcc = \grcalca
   \advance \grcalcc by \factor
   \grcalcd = \grrow
   \multiply \grcalcd by \factor
   \grcalce = \grcalcd
   \advance \grcalce by -\tfactor
   \grcalcf = \grcalcd
   \advance \grcalcf by -\hfactor
   \grcalcg = \grcalce
   \advance \grcalcg by -\tfactor
   \grcalch = \grcalcd
   \advance \grcalch by -\factor
   \put(\grcalcb,\grcalcf){\circle{\hfactor}}
   \qbezier(\grcalcc,\grcalcd)(\grcalcc,\grcalce)(\grcalcb,\grcalcf) 
   \qbezier(\grcalcb,\grcalcf)(\grcalca,\grcalcg)(\grcalca,\grcalch) 
   \advance \grcalcf by -\dfactor
   \advance \grcalcb by \sfactor
   \qbezier(\grcalcc,\grcalch)(\grcalcc,\grcalcg)(\grcalcb,\grcalcf) 
   \advance \grcalcf by \sfactor
   \advance \grcalcb by -\tfactor
   \qbezier(\grcalca,\grcalcd)(\grcalca,\grcalce)(\grcalcb,\grcalcf) 
   \advance \grcolumn by 2}

 \newcommand{\gu}[1]{
   \grcalca = \grcolumn
   \multiply \grcalca by \factor
   \grcalcd = \hfactor
   \multiply \grcalcd by #1
   \advance \grcalca by \grcalcd
   \grcalcb = \grrow
   \advance \grcalcb by -1
   \multiply \grcalcb by \factor
   \put(\grcalca,\grcalcb) {\line(0,1){\hfactor}} 
   \advance \grcalcb by \hfactor
   \put(\grcalca,\grcalcb) {\circle*{3}}
   \advance \grcolumn by #1}

 \newcommand{\gcu}[1]{
   \grcalca = \grcolumn
   \multiply \grcalca by \factor
   \grcalcd = \hfactor
   \multiply \grcalcd by #1
   \advance \grcalca by \grcalcd
   \grcalcb = \grrow
   \multiply \grcalcb by \factor
   \put(\grcalca,\grcalcb) {\line(0,-1){\hfactor}} 
   \advance \grcalcb by -\hfactor
   \put(\grcalca,\grcalcb) {\circle*{3}}
   \advance \grcolumn by #1}

 \newcommand{\gmp}[1]{
   \grcalca = \grcolumn
   \multiply \grcalca by \factor
   \advance \grcalca by \hfactor
   \grcalcb = \grrow
   \multiply \grcalcb by \factor
   \put(\grcalca,\grcalcb) {\line(0,-1){\dfactor}} 
   \advance \grcalcb by -\factor
   \put(\grcalca,\grcalcb) {\line(0,1){\dfactor}} 
   \advance \grcalcb by \hfactor
   \grcalcc = \factor
   \advance \grcalcc by -\qfactor
   \put(\grcalca,\grcalcb) {\circle{\grcalcc}}
   \put(\grcalca,\grcalcb) {\makebox(0,0){$\scriptstyle #1$}}
   \advance \grcolumn by 1}

 \newcommand{\gbmp}[1]{
   \grcalca = \grcolumn
   \multiply \grcalca by \factor
   \advance \grcalca by \hfactor
   \grcalcb = \grrow
   \multiply \grcalcb by \factor
   \put(\grcalca,\grcalcb) {\line(0,-1){\dfactor}} 
   \advance \grcalcb by -\factor
   \put(\grcalca,\grcalcb) {\line(0,1){\dfactor}} 
   \advance \grcalca by -\hfactor
   \advance \grcalca by \dfactor
   \advance \grcalcb by \dfactor
   \grcalcc = \factor
   \advance \grcalcc by -\sfactor
   \put(\grcalca,\grcalcb) {\framebox(\grcalcc,\grcalcc){$\scriptstyle #1$}}
   \advance \grcolumn by 1}

 \newcommand{\gbmpt}[1]{
   \grcalca = \grcolumn
   \multiply \grcalca by \factor
   \advance \grcalca by \hfactor
   \grcalcb = \grrow
   \multiply \grcalcb by \factor
   \put(\grcalca,\grcalcb) {\line(0,-1){\dfactor}} 
   \advance \grcalcb by -\factor
   \advance \grcalca by -\hfactor
   \advance \grcalca by \dfactor
   \advance \grcalcb by \dfactor
   \grcalcc = \factor
   \advance \grcalcc by -\sfactor
   \put(\grcalca,\grcalcb) {\framebox(\grcalcc,\grcalcc){$\scriptstyle #1$}}
   \advance \grcolumn by 1}

 \newcommand{\gbmpb}[1]{
   \grcalca = \grcolumn
   \multiply \grcalca by \factor
   \advance \grcalca by \hfactor
   \grcalcb = \grrow
   \multiply \grcalcb by \factor
   \advance \grcalcb by -\factor
   \put(\grcalca,\grcalcb) {\line(0,1){\dfactor}} 
   \advance \grcalca by -\hfactor
   \advance \grcalca by \dfactor
   \advance \grcalcb by \dfactor
   \grcalcc = \factor
   \advance \grcalcc by -\sfactor
   \put(\grcalca,\grcalcb) {\framebox(\grcalcc,\grcalcc){$\scriptstyle #1$}}
   \advance \grcolumn by 1}

 \newcommand{\gbmpn}[1]{
   \grcalca = \grcolumn
   \multiply \grcalca by \factor
   \advance \grcalca by \hfactor
   \grcalcb = \grrow
   \multiply \grcalcb by \factor
   \advance \grcalcb by -\factor
   \advance \grcalca by -\hfactor
   \advance \grcalca by \dfactor
   \advance \grcalcb by \dfactor
   \grcalcc = \factor
   \advance \grcalcc by -\sfactor
   \put(\grcalca,\grcalcb) {\framebox(\grcalcc,\grcalcc){$\scriptstyle #1$}}
   \advance \grcolumn by 1}

 \newcommand{\glmptb}{    
   \grcalca = \grcolumn
   \multiply \grcalca by \factor
   \advance \grcalca by \hfactor
   \grcalcb = \grrow
   \multiply \grcalcb by \factor
   \put(\grcalca,\grcalcb) {\line(0,-1){\dfactor}} 
   \advance \grcalcb by -\factor
   \put(\grcalca,\grcalcb) {\line(0,1){\dfactor}} 
   \advance \grcalca by -\hfactor
   \advance \grcalca by \dfactor
   \advance \grcalcb by \dfactor
   \put(\grcalca,\grcalcb) {\line(1,0){\factor}} 
   \advance \grcalcb by \factor
   \advance \grcalcb by -\sfactor
   \put(\grcalca,\grcalcb) {\line(1,0){\factor}} 
   \grcalcc = \factor
   \advance \grcalcc by -\sfactor
   \put(\grcalca,\grcalcb) {\line(0,-1){\grcalcc}} 
   \advance \grcolumn by 1}

 \newcommand{\glmpt}{    
   \grcalca = \grcolumn
   \multiply \grcalca by \factor
   \advance \grcalca by \hfactor
   \grcalcb = \grrow
   \multiply \grcalcb by \factor
   \put(\grcalca,\grcalcb) {\line(0,-1){\dfactor}} 
   \advance \grcalca by -\hfactor
   \advance \grcalca by \dfactor
   \advance \grcalcb by -\dfactor
   \put(\grcalca,\grcalcb) {\line(1,0){\factor}} 
   \advance \grcalcb by -\factor
   \advance \grcalcb by \sfactor
   \put(\grcalca,\grcalcb) {\line(1,0){\factor}} 
   \grcalcc = \factor
   \advance \grcalcc by -\sfactor
   \put(\grcalca,\grcalcb) {\line(0,1){\grcalcc}} 
   \advance \grcolumn by 1}

 \newcommand{\glmpb}{    
   \grcalca = \grcolumn
   \multiply \grcalca by \factor
   \advance \grcalca by \hfactor
   \grcalcb = \grrow
   \multiply \grcalcb by \factor
   \advance \grcalcb by -\factor
   \put(\grcalca,\grcalcb) {\line(0,1){\dfactor}} 
   \advance \grcalca by -\hfactor
   \advance \grcalca by \dfactor
   \advance \grcalcb by \dfactor
   \put(\grcalca,\grcalcb) {\line(1,0){\factor}} 
   \advance \grcalcb by \factor
   \advance \grcalcb by -\sfactor
   \put(\grcalca,\grcalcb) {\line(1,0){\factor}} 
   \grcalcc = \factor
   \advance \grcalcc by -\sfactor
   \put(\grcalca,\grcalcb) {\line(0,-1){\grcalcc}} 
   \advance \grcolumn by 1}

 \newcommand{\glmp}{    
   \grcalca = \grcolumn
   \multiply \grcalca by \factor
   \advance \grcalca by \dfactor
   \grcalcb = \grrow
   \multiply \grcalcb by \factor
   \advance \grcalcb by -\dfactor
   \put(\grcalca,\grcalcb) {\line(1,0){\factor}} 
   \advance \grcalcb by -\factor
   \advance \grcalcb by \sfactor
   \put(\grcalca,\grcalcb) {\line(1,0){\factor}} 
   \grcalcc = \factor
   \advance \grcalcc by -\sfactor
   \put(\grcalca,\grcalcb) {\line(0,1){\grcalcc}} 
   \advance \grcolumn by 1}

 \newcommand{\gcmptb}{    
   \grcalca = \grcolumn
   \multiply \grcalca by \factor
   \advance \grcalca by \hfactor
   \grcalcb = \grrow
   \multiply \grcalcb by \factor
   \put(\grcalca,\grcalcb) {\line(0,-1){\dfactor}} 
   \advance \grcalcb by -\factor
   \put(\grcalca,\grcalcb) {\line(0,1){\dfactor}} 
   \advance \grcalca by -\hfactor
   \advance \grcalcb by \dfactor
   \put(\grcalca,\grcalcb) {\line(1,0){\factor}} 
   \advance \grcalcb by \factor
   \advance \grcalcb by -\sfactor
   \put(\grcalca,\grcalcb) {\line(1,0){\factor}} 
   \advance \grcolumn by 1}

\newcommand{\gmpcu}[1]{
   \grcalca = \grcolumn
   \multiply \grcalca by \factor
   \advance \grcalca by \hfactor
   \grcalcb = \grrow
   \multiply \grcalcb by \factor
   \put(\grcalca,\grcalcb) {\line(0,-1){\dfactor}} 
   \advance \grcalcb by -\factor
   \advance \grcalcb by \hfactor
   \grcalcc = \factor
   \advance \grcalcc by -\qfactor
   \put(\grcalca,\grcalcb) {\circle{\grcalcc}}
   \put(\grcalca,\grcalcb) {\makebox(0,0){$\scriptstyle #1$}}
   \advance \grcolumn by 1}
   
\newcommand{\gmpu}[1]{
   \grcalca = \grcolumn
   \multiply \grcalca by \factor
   \advance \grcalca by \hfactor
   \grcalcb = \grrow
   \multiply \grcalcb by \factor
   \advance \grcalcb by -\factor
   \put(\grcalca,\grcalcb) {\line(0,1){\dfactor}} 
   \advance \grcalcb by \hfactor
   \grcalcc = \factor
   \advance \grcalcc by -\qfactor
   \put(\grcalca,\grcalcb) {\circle{\grcalcc}}
   \put(\grcalca,\grcalcb) {\makebox(0,0){$\scriptstyle #1$}}
   \advance \grcolumn by 1}      

 \newcommand{\gcmpt}{    
   \grcalca = \grcolumn
   \multiply \grcalca by \factor
   \advance \grcalca by \hfactor
   \grcalcb = \grrow
   \multiply \grcalcb by \factor
   \put(\grcalca,\grcalcb) {\line(0,-1){\dfactor}} 
   \advance \grcalcb by -\factor
   \advance \grcalca by -\hfactor
   \advance \grcalcb by \dfactor
   \put(\grcalca,\grcalcb) {\line(1,0){\factor}} 
   \advance \grcalcb by \factor
   \advance \grcalcb by -\sfactor
   \put(\grcalca,\grcalcb) {\line(1,0){\factor}} 
   \advance \grcolumn by 1}

 \newcommand{\gcmpb}{    
   \grcalca = \grcolumn
   \multiply \grcalca by \factor
   \advance \grcalca by \hfactor
   \grcalcb = \grrow
   \multiply \grcalcb by \factor
   \advance \grcalcb by -\factor
   \put(\grcalca,\grcalcb) {\line(0,1){\dfactor}} 
   \advance \grcalca by -\hfactor
   \advance \grcalcb by \dfactor
   \put(\grcalca,\grcalcb) {\line(1,0){\factor}} 
   \advance \grcalcb by \factor
   \advance \grcalcb by -\sfactor
   \put(\grcalca,\grcalcb) {\line(1,0){\factor}} 
   \advance \grcolumn by 1}

 \newcommand{\gcmp}{    
   \grcalca = \grcolumn
   \multiply \grcalca by \factor
   \grcalcb = \grrow
   \multiply \grcalcb by \factor
   \advance \grcalcb by -\factor
   \advance \grcalcb by \dfactor
   \put(\grcalca,\grcalcb) {\line(1,0){\factor}} 
   \advance \grcalcb by \factor
   \advance \grcalcb by -\sfactor
   \put(\grcalca,\grcalcb) {\line(1,0){\factor}} 
   \advance \grcolumn by 1}

 \newcommand{\grmptb}{    
   \grcalca = \grcolumn
   \multiply \grcalca by \factor
   \advance \grcalca by \hfactor
   \grcalcb = \grrow
   \multiply \grcalcb by \factor
   \put(\grcalca,\grcalcb) {\line(0,-1){\dfactor}} 
   \advance \grcalcb by -\factor
   \put(\grcalca,\grcalcb) {\line(0,1){\dfactor}} 
   \advance \grcalca by \hfactor
   \advance \grcalca by -\dfactor
   \advance \grcalcb by \dfactor
   \put(\grcalca,\grcalcb) {\line(-1,0){\factor}} 
   \advance \grcalcb by \factor
   \advance \grcalcb by -\sfactor
   \put(\grcalca,\grcalcb) {\line(-1,0){\factor}} 
   \grcalcc = \factor
   \advance \grcalcc by -\sfactor
   \put(\grcalca,\grcalcb) {\line(0,-1){\grcalcc}} 
   \advance \grcolumn by 1}

 \newcommand{\grmpt}{    
   \grcalca = \grcolumn
   \multiply \grcalca by \factor
   \advance \grcalca by \hfactor
   \grcalcb = \grrow
   \multiply \grcalcb by \factor
   \put(\grcalca,\grcalcb) {\line(0,-1){\dfactor}} 
   \advance \grcalca by \hfactor
   \advance \grcalca by -\dfactor
   \advance \grcalcb by -\dfactor
   \put(\grcalca,\grcalcb) {\line(-1,0){\factor}} 
   \advance \grcalcb by -\factor
   \advance \grcalcb by \sfactor
   \put(\grcalca,\grcalcb) {\line(-1,0){\factor}} 
   \grcalcc = \factor
   \advance \grcalcc by -\sfactor
   \put(\grcalca,\grcalcb) {\line(0,1){\grcalcc}} 
   \advance \grcolumn by 1}

 \newcommand{\grmpb}{    
   \grcalca = \grcolumn
   \multiply \grcalca by \factor
   \advance \grcalca by \hfactor
   \grcalcb = \grrow
   \multiply \grcalcb by \factor
   \advance \grcalcb by -\factor
   \put(\grcalca,\grcalcb) {\line(0,1){\dfactor}} 
   \advance \grcalca by \hfactor
   \advance \grcalca by -\dfactor
   \advance \grcalcb by \dfactor
   \put(\grcalca,\grcalcb) {\line(-1,0){\factor}} 
   \advance \grcalcb by \factor
   \advance \grcalcb by -\sfactor
   \put(\grcalca,\grcalcb) {\line(-1,0){\factor}} 
   \grcalcc = \factor
   \advance \grcalcc by -\sfactor
   \put(\grcalca,\grcalcb) {\line(0,-1){\grcalcc}} 
   \advance \grcolumn by 1}

 \newcommand{\grmp}{    
   \grcalca = \grcolumn
   \multiply \grcalca by \factor
   \advance \grcalca by \factor
   \advance \grcalca by -\dfactor
   \grcalcb = \grrow
   \multiply \grcalcb by \factor
   \advance \grcalcb by -\dfactor
   \put(\grcalca,\grcalcb) {\line(-1,0){\factor}} 
   \advance \grcalcb by -\factor
   \advance \grcalcb by \sfactor
   \put(\grcalca,\grcalcb) {\line(-1,0){\factor}} 
   \grcalcc = \factor
   \advance \grcalcc by -\sfactor
   \put(\grcalca,\grcalcb) {\line(0,1){\grcalcc}} 
   \advance \grcolumn by 1}

 \newcommand{\gwmuh}[3]{    
   \grcalca = \grcolumn
   \multiply \grcalca by \factor
   \grcalcb = #2
   \advance \grcalcb by #3
   \multiply \grcalcb by \qfactor
   \advance \grcalca by \grcalcb
   \grcalcb = \grrow
   \multiply \grcalcb by \factor
   \grcalcc = #3
   \advance \grcalcc by -#2
   \multiply \grcalcc by \hfactor
   \grcalcd = \factor
   \advance \grcalcd by \hfactor
   \put(\grcalca,\grcalcb){\oval(\grcalcc,\grcalcd)[b]}
   \grcalca = \grcolumn
   \multiply \grcalca by \factor
   \grcalcc = #1
   \multiply \grcalcc by \hfactor
   \advance \grcalca by \grcalcc
   \advance \grcalcb by -\hfactor
   \advance \grcalcb by -\qfactor
   \put(\grcalca,\grcalcb) {\line(0,-1){\qfactor}} 
   \advance \grcolumn by #1}

 \newcommand{\gwcmh}[3]{   
   \grcalca = \grcolumn
   \multiply \grcalca by \factor
   \grcalcb = #2
   \advance \grcalcb by #3
   \multiply \grcalcb by \qfactor
   \advance \grcalca by \grcalcb
   \grcalcb = \grrow
   \advance \grcalcb by -1
   \multiply \grcalcb by \factor
   \grcalcc = #3
   \advance \grcalcc by -#2
   \multiply \grcalcc by \hfactor
   \grcalcd = \factor
   \advance \grcalcd by \hfactor
   \put(\grcalca,\grcalcb){\oval(\grcalcc,\grcalcd)[t]}
   \grcalca = \grcolumn
   \multiply \grcalca by \factor
   \grcalcc = #1
   \multiply \grcalcc by \hfactor
   \advance \grcalca by \grcalcc
   \advance \grcalcb by \factor
   \put(\grcalca,\grcalcb) {\line(0,-1){\qfactor}} 
   \advance \grcolumn by #1}

 \newcommand{\gsbox}[1]{
   \grcalca = \grcolumn
   \multiply \grcalca by \factor
   \grcalcb = \grrow
   \multiply \grcalcb by \factor
   \advance \grcalcb by -\factor
   \grcalcc = #1
   \multiply \grcalcc by \factor
   \grcalcd = \factor
   \put(\grcalca,\grcalcb){\framebox(\grcalcc,\grcalcd){}}}

\allowdisplaybreaks[4]

\begin{document}
\title[Biproducts of rank 2]
{Biproduct quasi-Hopf algebras of rank $2$}
\author{D. Bulacu}
\address{Faculty of Mathematics and Informatics, University
of Bucharest, Str. Academiei 14, RO-010014 Bucharest 1, Romania}
\email{daniel.bulacu@fmi.unibuc.ro}
\author{M. Misurati}
\address{University of Ferrara, Department of Mathematics, Via Machiavelli
35, Ferrara, I-44121, Italy}
\email{matteo.misurati@unife.it}
\thanks{
The first author was supported by the grant PCE47-2022 of UEFISCDI,
project PN-III-P4-PCE-2021-0282.} 

\begin{abstract}
Inspired by the work of Radford, for $H$ an arbitrary quasi-Hopf algebra we describe all the Hopf algebras of dimension 
$2$ within the braided category of left Yetter-Drinfeld modules over $H$ and determine the biproduct quasi-Hopf algebras defined by them. 
Classes of such biproduct quasi-Hopf algebras are obtained by taking $H$ as the Hopf algebra of functions on a group $G$, 
endowed with the quasi-Hopf algebra structure provided by a non-trivial 
$3$-cocycle on $G$ (especially when $G$ is a finite cyclic group or the double dihedral group), or as being a quasi-Hopf algebra with 
radical of codimension two. In this way we uncover new classes  of basic quasi-Hopf algebras of even dimension, as well as new classes of tensor categories.         
\end{abstract}
\maketitle
\section*{Introduction}\selabel{intro}
For a pointed Hopf algebra $H$, the associated graded algebra ${\rm gr}(H)$ of $H$ is a Hopf algebra with projection, and so it is determined by a braided 
Hopf algebra $R$ within a category of Yetter-Drinfeld modules defined by a Hopf group algebra. Note that, the Yetter-Drinfeld categories were introduced 
by Yetter in \cite{y} under the name of crossings bimodules, and that their actual name is due to Radford and Towber \cite{rt}. The lifting method consists 
of studying $R$, then transfer the properties we found to ${\rm gr}(H)$ and finally lift them to $H$. This method has been proved to be an extremely powerful 
tool in classifying pointed Hopf algebras; see the survey papers \cite{a, ai}.

The lifting method was used for the first time in the setting of quasi-Hopf algebras by Etingof and Gelaki in \cite{eg3,eg5,eg4}. 
Quasi-Hopf algebras were introduced 
by Drinfeld \cite{dri87, dri}; for short, $H$ is an algebra for which its category of representations ${}_H{\cal M}$ is monoidal in such a way that the 
forgetful functor from ${}_H{\cal M}$ to the category of $k$-vector spaces is monoidal. This fact was exploited by Majid in \cite{m1} in order to introduce 
the Yetter-Drinfeld module category over a quasi-Hopf algebra $H$, ${}_H^H{\cal YD}$, as the categorical center of ${}_H{\cal M}$. Afterwards, in \cite{bn}, it is 
proved that to a Hopf algebra $R$ in ${}_H^H{\cal YD}$ we can associate a quasi-Hopf algebra $R\times H$, called the biproduct of $R$ and $H$. 
As in the quasi-Hopf case we cannot deal with the coradical filtration, we must consider the dual situation (this idea was exploited for the first time in \cite{eg3}); namely, to consider the filtration of an algebra by the powers of the Jacobson radical instead of the coradical filtration. 
In the case when $H$ is a basic quasi-Hopf algebra 
(the dual notion of pointed in the Hopf case) the Jacobson radical $J$ of $H$ is a quasi-Hopf ideal of $H$ and the graded algebra ${\rm gr}(H)$ is a biproduct 
quasi-Hopf algebra of $R$ and $H_0:=\frac{H}{J}$. We must mention that $H_0$ is a group Hopf algebra of a certain group $G$, seen as a quasi-Hopf algebra 
via a non-trivial $3$-cocycle on $G$, and that $R$ is a braided Hopf algebra in ${}_{H_0}^{H_0}{\cal YD}$. Therefore, we have a lifting method in the quasi-Hopf case, too. Apart from the papers mentioned above, it has been also used in \cite{an,bm}. 

This article started from the desire to have a quasi-Hopf analogue for $H_4$, the famous Sweedler $4$-dimensional Hopf algebra. 
It is well known that $H_4$ is a biproduct between a $2$-dimensional braided Hopf algebra $R$ and $k[C_2]$, the group Hopf algebra of $k$ and the cyclic group of order two $C_2$. Thus, in order to have a quasi-Hopf analogue for $H_4$, we must replace $k[C_2]$ by a quasi-Hopf algebra of dimension two, 
that is by $k_\phi[C_2]$: there is only one non-trivial 
$3$-cocycle $\phi$ on $C_2$, provided that $k$ has characteristic different from $2$, and the unique quasi-Hopf algebra of dimension 
$2$ is $k_\phi[C_2]$, the group Hopf algebra $k[C_2]$ considered as a quasi-Hopf algebra via the reassociator determined by $\phi$. 
Resuming, the quasi-Hopf analogue of $H_4$ can be obtained by looking at the $2$-dimensional Hopf algebras within 
${}_{k_\phi[C_2]}^{k_\phi[C_2]}{\cal YD}$. Inspired by the work of Radford \cite{Rbke}, we do something more general: we classify the $2$-dimensional 
Hopf algebras $R$ within the category of Yetter-Drinfeld modules ${}_H^H{\cal YD}$ over an arbitrary quasi-Hopf algebra $H$. More precisely, 
in Section \ref{2dH} we show that any such Hopf algebra $R$ is isomorphic to either the group Hopf algebra $k[C_2]$, with trivial action and coaction, or to a 
certain braided Hopf algebra $\tilde{B}_{\alpha, y}$, having the $H$-action and $H$-coaction determined by an algebra morphism $\alpha$, and 
respectively by an element $y$ of $H$; of course, $\alpha$ and $y$ must satisfy some relations. Coming back to our initial aim, we deduce that   
the quasi-Hopf analogue of $H_4$ is actually a twist deformation of $H_4$. Looking at the classification of certain tensor finite categories, we realized   
that our result follows as well from \cite[Corollary 3.6]{eg3}; furthermore, the quoted result contributes to the classification of quasi-Hopf algebras $H$ 
of dimension $4$: $H$ 
is either twist equivalent to $H_4$ or is a group Hopf algebra of $k$ and a group $G$ of order $4$, seen as quasi-Hopf algebra via a non-trivial $3$-cocycle on $G$. 
In \seref{class4dim}, we present a slightly different proof for the fact that ${\rm gr}(A)$ of a nonsemisimple quasi-Hopf algebra $A$ of dimension $4$ is always 
isomorphic to $H_4$, as a Hopf algebra; the technique we used is similar to the one in \cite{bm} (where we completed the classification of the quasi-Hopf algebras  of dimension $6$), and substitutes the categorical arguments used in \cite{eg3}. Then, we uncover when two twisted versions of 
$H_4$ are isomorphic as quasi-Hopf algebras; surprisingly, and different from the Hopf algebra case, our criterion gives rise to an infinite class of non-semisimple 
quasi-Hopf algebras of dimension $4$ which are not pairwise isomorphic (but, of course, all of them are twist equivalent to $H_4$). \seref{class4dim} ends with the 
description of the all 12 semisimple quasi-Hopf algebras of dimension $4$ which are not pairwise twist equivalent.  
 
Starting with \seref{funalgexp}, we move to applications. We first consider the group case, that is when $H$ is $k^G_\omega$, the algebra of functions on a finite 
group $G$ regarded as a quasi-Hopf algebra with the help of a normalized $3$-cocycle $\omega$ on $G$. By generalizing to the non-abelian case a result 
from \cite{hlyy}, we associate to $\omega$ and an element ${\mf g}$ in the center of $G$ a $2$-cocycle $\flat_{\mf g}\omega$ on $G$. Then we 
show that the non-trivial (this means, different from $k[C_2]$) $2$-dimensional braided Hopf algebras in 
${}_{k^G_\omega}^{k^G_\omega}{\cal YD}$ are determined by couples $(\mf{g}, \rho)$ 
consisting of an element $\mfg$ in the center of $G$ and a map $\rho: G\ra k^*$ obeying $\rho(e)=1$, $\rho(\mfg)=-1$ and $\flat_\mfg\omega=\partial\rho$ (the latest 
condition says that $\flat_{\mf g}\omega$ must by a coboundary $2$-cocycle); see \prref{fun chara}. We explicitly describe these pairs $(\mf{g}, \rho)$ 
in the case when $G$ is a finite abelian group (\prref{chara abel}). Towards this end, we use the concrete description of the representative 
$3$-cocycles on $G$ given in \cite{hlyy}, and see that the finding of $(\mf{g}, \rho)$ reduces to the solving of a system of linear congruences or, equivalently, 
to the solving of a homogeneous linear system $XA=0$, where $A$ is a certain matrix. The system $XA=0$ can be solved quickly if we know the 
Schmidt decomposition of $A$, for instance when $G$ is the direct product of a 3 cyclic groups. In general, we indicate a way for the solving of the system 
$AX=0$ and we exemplify it for various finite abelian groups. To argue the complexity of the problem we are dealing with, note that in the case when $G$ is a 
product of two cyclic groups we reduce our problem to the finding of the rational points of a conic section. We can do this by reducing the conic to its standard form, and then facing with combinations between Diophantine equations and Pell-type equations; note that, plenty of parameters are involved here. 
An "ideal" case that can occur is the one in which the conic section is defined by a quadratic form, as in this case the Hasse-Minkowski theorem applies (we refer to \cite{HCohen} for the quoted theorem). We illustrate all these aspects in \seref{Appendix}, where a particular case is considered. We also present an example 
in the non-abelian case; namely, when $G$ is the double dihedral group $\bar D_n$, a group that realizes the group of quaternions in the case when $n=2$. Last but not least, in \seref{noncommnoncocommexp} we specialize our results to the case when $H$ is one of the quasi-Hopf algebras with radical of codimension $2$, previously  classified in \cite{eg3}.
\section{Preliminaries}
\setcounter{equation}{0}
\subsection{Quasi-bialgebras and quasi-Hopf algebras.}
We work over a field $k$. All algebras, linear
spaces, etc. will be over $k$; unadorned $\ot $ means $\ot_k$.
Following Drinfeld \cite{dri}, a quasi-bialgebra is
a quadruple $(H, \Delta , \va , \Phi )$ where $H$ is
an associative algebra with unit,
$\Phi$ is an invertible element in $H\ot H\ot H$, and
$\Delta :\ H\ra H\ot H$ and $\va :\ H\ra k$ are algebra
homomorphisms obeying the identities
\begin{eqnarray*}
&&(\Id_H \ot \Delta )(\Delta (h))=
\Phi (\Delta \ot \Id_H)(\Delta (h))\Phi ^{-1},\\
&&(\Id_H \ot \va )(\Delta (h))=h~~,~~
(\va \ot \Id_H)(\Delta (h))=h,
\end{eqnarray*}
for all $h\in H$, where
$\Phi$ is a normalized $3$-cocycle, in the sense that $(\Id \ot \va \ot \Id_H)(\Phi)=1\ot 1$ and 
\begin{eqnarray}
&&(1\ot \Phi)(\Id_H\ot \Delta \ot \Id_H)
(\Phi)(\Phi \ot 1)\nonumber\\
&&\hspace*{1.5cm}
=(\Id_H\ot \Id_H \ot \Delta )(\Phi )
(\Delta \ot \Id_H \ot \Id_H)(\Phi).\eqlabel{q3}
\end{eqnarray}
The map $\Delta$ is called the coproduct or the
comultiplication, $\va $ is the counit, and $\Phi $ is the
reassociator. As for Hopf algebras, we denote $\Delta (h)=h_1\ot h_2$,
but since $\Delta $ is only quasi-coassociative we adopt the
further convention (summation understood):
\[
(\Delta \ot \Id_H)(\Delta (h))=h_{11}\ot h_{12}\ot h_2~~,~~
(\Id_H\ot \Delta)(\Delta (h))=h_1\ot h_{21}\ot h_{22},
\]
for all $h\in H$. We will denote the tensor components of $\Phi$
by capital letters, and the ones of $\Phi^{-1}$ by lower case letters, namely
\begin{eqnarray*}
&&\Phi=X^1\ot X^2\ot X^3=T^1\ot T^2\ot T^3=
V^1\ot V^2\ot V^3=\cdots\\
&&\Phi^{-1}=x^1\ot x^2\ot x^3=t^1\ot t^2\ot t^3=
v^1\ot v^2\ot v^3=\cdots
\end{eqnarray*}
$H$ is called a quasi-Hopf
algebra if, moreover, there exists an
anti-morphism $S$ of the algebra
$H$ and elements $\a , \b \in
H$ such that, for all $h\in H$, we
have:
\begin{eqnarray}
&&
S(h_1)\a h_2=\va(h)\a
~~{\rm and}~~
h_1\b S(h_2)=\va (h)\b,\eqlabel{q5}\\ 
&&X^1\b S(X^2)\a X^3=1
~~{\rm and}~~
S(x^1)\a x^2\b S(x^3)=1.\eqlabel{q6}
\end{eqnarray}

Our definition of a quasi-Hopf algebra is different from the
one given by Drinfeld \cite{dri}, in the sense that we do not
require the antipode to be bijective. In the case where $H$ is finite dimensional
or quasi-triangular, bijectivity of the antipode follows from the other axioms,
see \cite{bc1} and \cite{bn3}, so the two definitions are equivalent. Anyway, the bijectivity 
of the antipode $S$ will be implicitly understood in the case when $S^{-1}$, the inverse 
of $S$, appears is formulas or computations.

Let $H$ be a quasi-bialgebra and $F\in H\otimes H$ an invertible element such that $\va(F^1)F^2=\va(F^2)F^1=1$, where 
$F=F^1\ot F^2$ is the formal notation for the tensor components of $F$; a similar notation we adopt for $F^{-1}$, $F^{-1}=G^1\ot G^2$. Note that 
$F$ is called a twist or gauge transformation for $H$. Let 
\begin{eqnarray*}
&&\Delta_F: H\ra H\ot H,~~\Delta_F(h)=F\Delta(h)F^{-1},\\
&&\Phi_F=(1_H\ot F)(\Id_H\ot \Delta)(F)\Phi (\Delta \ot \Id_H)(F^{-1})(F^{-1}\ot 1_H).
\end{eqnarray*}
Then $H_F:=(H, \Delta_F, \va , \Phi_F)$ is a quasi-bialgebra as well. Furthermore, if $H$ is a quasi-Hopf algebra then so is 
$H_F$ with $S_F=S$, $\a_F=S(G^1)\a G^2$ and $\b_F=F^1\b S(F^2)$.  

Two quasi-bialgebras (resp. quasi-Hopf algebras) $H$, $H'$ are twist equivalent if there is a twist 
$F\in H'\ot H'$ such that $H$, $H'_F$ are isomorphic as quasi-bialgebras (resp. quasi-Hopf algebras); see \cite[Definition 3.6]{bcpvo}  
for the definition of a quasi-bialgebra (resp. quasi-Hopf algebra) isomorphism.

Following \cite{hnRMP}, we define
\begin{eqnarray}
p_R&=&p^1\ot p^2=x^1\ot x^2\b S(x^3),\eqlabel{pr}\\
q_R&=&q^1\ot q^2=X^1\ot S^{-1}(\a X^3)X^2.\eqlabel{qr}
\end{eqnarray}
For all $h\in H$, we then have
\begin{eqnarray}
\Delta (h_1)p_R(1\ot S(h_2))&=&p_R(h\ot 1),\eqlabel{qr1}\\
(1\ot S^{-1}(h_2))q_R\Delta (h_1)&=&(h\ot 1)q_R.\eqlabel{qr1a}
\end{eqnarray}
Furthermore, the following relations hold: 
\begin{eqnarray}
(1\ot S^{-1}(p^2))q_R\Delta (p^1)&=&1\ot 1,\eqlabel{pqra}\\
\Delta (q^1)p_R(1\ot S(q^2))&=&1\ot 1.\eqlabel{pqr}
\end{eqnarray}

\subsection{Yetter-Drinfeld module categories and braided Hopf algebras}
The category of left Yetter-Drinfeld modules over a quasi-bialgebra $H$, denoted by ${}_H^H{\cal YD}$, was introduced by Majid in \cite{m1}. 
The objects of ${}_H^H{\cal YD}$ are left $H$-modules $M$ on which $H$ coacts from the left (we denote 
by $\lambda_M:\ M\to H\ot M,~~\lambda_M(m)=m_{[-1]}\ot m_{[0]}$ the left $H$-coaction on $M$) such that $\va(m_{[-1]})m_{[0]}=m$ and, 
for all $m\in M$,  
\begin{eqnarray}
&&\hspace*{-1.5cm}
X^1m_{[-1]}\ot (X^2\cd m_{[0]})_{[-1]}X^3
\ot (X^2\cd m_{[0]})_{[0]}\nonumber\\
&&\hspace*{1cm}=X^1(Y^1\cd m)_{[-1]_1}Y^2\ot X^2(Y^1\cd m)_{[-1]_2}Y^3
\ot X^3\cd (Y^1\cd m)_{[0]}.\label{y1}
\end{eqnarray}
It is compatible with the left $H$-module structure on $M$, in the sense that   
\begin{equation}\label{y3}
h_1m_{[-1]}\ot h_2\cd m_{[0]}=(h_1\cd m)_{[-1]}h_2\ot (h_1\cd m)_{[0]},~\forall~h\in H,~m\in M.
\end{equation}
When $H$ is a quasi-Hopf algebra with bijective antipode, ${}_H^H{\cal YD}$ is braided. 
The monoidal structure on ${}_H^H{\cal YD}$ is such that the forgetful functor ${}_H^H{\cal YD}\ra {}_H{\cal M}$ is strong monoidal. 
The coaction on the tensor product $M\ot N$ of two Yetter-Drinfeld modules $M$, $N$ is given, for all $m\in M$ and $n\in N$, by 
\begin{equation}\eqlabel{y4}
m\ot n\mapsto X^1(x^1Y^1\cdot m)_{[-1]} x^2(Y^2\cd n)_{[-1]}Y^3
\ot X^2\cd (x^1Y^1\cdot m)_{[0]}\ot X^3x^3\cdot (Y^2\cdot n)_{[0]}.
\end{equation}

In what follows, we call an algebra, coalgebra etc. in ${}_H^H{\cal YD}$ a Yetter-Drinfeld algebra, coalgebra etc. ($YD$-algebra, coalgebra etc. for short). 
Record that a YD-algebra is a Yetter-Drinfeld module $B$ equipped with a unital multiplication $\un{m}_B: B\ot B\ni b\ot b'\mapsto bb'\in B$, $1_B\in B$ is our   
notation for the unit of $B$, such that 
\begin{eqnarray}
&&h\cdot (\mf{b}\mf{b}')=(h_1\cdot \mf{b})(h_2\cdot \mf{b}')~,~h\cdot 1_B=\va(h)1_B,\label{modalg1}\\
&&(\mf{b}\mf{b}')\mf{b}{''}=(X^1\cdot \mf{b})[(X^2\cdot \mf{b}')(X^3\cdot \mf{b}{''})],\label{modalg2}\\
&&(\mf{b}\mf{b}')_{[-1]}\ot (\mf{b}\mf{b}')_{[0]}\nonumber\\
&&\hspace{2cm}
=X^1(x^1Y^1\cdot \mf{b})_{[-1]} x^2(Y^2\cd \mf{b}')_{[-1]}Y^3
\ot (X^2\cd (x^1Y^1\cdot \mf{b})_{[0]})(X^3x^3\cdot (Y^2\cdot \mf{b}')_{[0]}),\label{modalg3}
\end{eqnarray}
for all $h\in H$ and $\mf{b}, \mf{b}', \mf{b}{''}\in B$. Note that (\ref{modalg1}) expresses the $H$-linearity of the multiplication and unit morphisms of $B$, 
(\ref{modalg2}) is just the associativity of $\un{m}_B$ in ${}_H{\cal M}$ and ${}_H^H{\cal YD}$ and (\ref{modalg3}) is the $H$-colinearity of $\un{m}_B$ 
in ${}_H^H{\cal YD}$; see \equref{y4}. Also, (\ref{modalg1}, \ref{modalg2}) are the required condition on $B$ to be an algebra in ${}_H{\cal M}$ or, equivalently, 
an $H$-module algebra.  

Likewise, a YD-coalgebra is a Yetter-Drinfeld module $B$ endowed with a comultiplication 
$\un{\Delta}_B: B\ni b\mapsto b_{\un{1}}\ot b_{\un{2}}\in B\ot B$ and counit $\un{\va}_B: B\ra k$ such that, for all $b\in B$,
\begin{eqnarray}
&&\un{\Delta}_B(h\cdot \mf{b})=h_1\cdot \mf{b}_{\un{1}}\ot h_2\cdot \mf{b}_{\un{2}},~~\un{\va}_B(h\cdot \mf{b})=\va(h)\un{\va}_B(\mf{b}),~~
\un{\va}_B(\mf{b}_{\un{1}})\mf{b}_{\un{2}}=\mf{b}=\un{\va}_B(\mf{b}_{\un{2}})\mf{b}_{\un{1}},\eqlabel{ydc1}\\
&&X^1\cdot \mf{b}_{\un{1}\un{1}}\ot X^2\cdot \mf{b}_{\un{1}\un{2}}\ot X^3\cdot \mf{b}_{\un{2}}=
\mf{b}_{\un{1}}\ot \mf{b}_{\un{2}\un{1}}\ot \mf{b}_{\un{2}\un{2}},\eqlabel{ydc2}\\
&&\mf{b}_{[-1]}\ot \mf{b}_{[0]_{\un{1}}}\ot \mf{b}_{[0]_{\un{2}}}=X^1(x^1Y^1\cdot \mf{b}_{\un{1}})_{[-1]}x^2(Y^2\cdot \mf{b}_{\un{2}})_{[-1]}Y^3\nonumber\\
&&\hspace{2cm}\ot 
X^2\cdot (x^1Y^1\cdot \mf{b}_{\un{1}})_{[0]}\ot X^3x^3\cdot (Y^2\cdot \mf{b}_{\un{2}})_{[0]},~~
\un{\va}_B(\mf{b}_{[0]})\mf{b}_{[-1]}=\un{\va}_B(\mf{b})1.\eqlabel{ydc3}
\end{eqnarray} 
Observe that \equref{ydc1} expresses the fact that $\un{\Delta}_B$, $\un{\va}_B$ are left $H$-linear and $\un{\va}_B$ is counit for $\un{\Delta}_B$; otherwise stated, 
$B$ is a left $H$-module coalgebra, this means a coalgebra in ${}_H{\cal M}$. The relation \equref{ydc2} is the coassociativity of $\un{\Delta}_B$ in 
${}_H{\cal M}$ or, equivalently, in ${}_H^H{\cal YD}$ (since the forgetful functor from ${}_H^H{\cal YD}$ to ${}_H{\cal M}$ is strong monoidal), 
while \equref{ydc3} expresses the left $H$-colinearity of $\un{\Delta}_B$, $\un{\va}_B$ in ${}_H^H{\cal YD}$; see \equref{y4}. Also, we denoted 
$(\un{\Delta}_B\ot \Id_B)(\un{\Delta}_B(\mf{b})):=\mf{b}_{\un{1}\un{1}}\ot \mf{b}_{\un{1}\un{2}}\ot \mf{b}_{\un{2}}$ and 
$(\Id_B\ot \un{\Delta}_B)(\un{\Delta}_B(\mf{b})):=\mf{b}_{\un{1}}\ot \mf{b}_{\un{2}\un{1}}\ot \mf{b}_{\un{2}\un{2}}$. 

A bialgebra in ${}_H^H{\cal YD}$ is a YD-algebra $B$ that is at the same time a YD-coalgebra such that the comultiplication $\un{\Delta}_B$ and the counit 
$\un{\va}_B$ are algebra morphisms in ${}_H^H{\cal YD}$; $B\ot B$ is considered as an algebra in ${}_H^H{\cal YD}$ via the tensor product algebra structure. 
Explicitly, for all $b, b'\in B$, 
\begin{eqnarray}
&&~\un{\Delta}_B(1_{B})=1_B\ot 1_B~,~\un{\va}_B(1_B)=1_k~,~\un{\va}_B(\mf{b}\mf{b}')=\un{\va}_B(\mf{b})\un{\va}_B(\mf{b}'),\label{moltcon1}\\
&&~\un{\Delta}_B(\mf{b}\mf{b}')=(y^1X^1\hspace{-1mm}\cdot \mf{b}_{\un{1}})(y^2Y^1(x^1X^2\hspace{-1mm}\cdot \mf{b}_{\un{2}})_{[-1]}x^2X_1^3\hspace{-1mm}\cdot 
\mf{b}'_{\un{1}})\ot (y_1^3Y^2\cdot (x^1X^2\hspace{-1mm}\cdot \mf{b}_{\un{2}})_{[0]})(y_2^3Y^3x^3X_2^3\hspace{-1mm}\cdot \mf{b}'_{\un{2}}).\label{moltcon2}
\end{eqnarray}

A Hopf algebra in ${}_H^H{\cal YD}$ is a YD-bialgebra $B$ for which there exists a morphism $\un{S}_B: B\ra B$ in ${}_H^H{\cal YD}$, called antipode, such that 
$\un{S}_B(b_{\un{1}})b_{\un{2}}=\un{\va}_B(b)1_B=b_{\un{1}}\un{S}(b_{\un{2}})$, for all $b\in B$.
\subsection{Biproduct quasi-Hopf algebras}
Owing to \cite{db, panfred}, to a YD-algebra $B$ one can associate a $k$-algebra $B\# H$, called the smash product of $B$ and $H$ (it is suffices for  
$B$ to be an $H$-module algebra only). As a vector space $B\# H$ equals $B\ot H$, and its multiplication is defined by 
\[
(\mf{b}\# h)(\mf{b}'\# h')=(x^1\cdot \mf{b})(x^2h_1\cdot \mf{b}')\# x^3h_2h',
\]
for all $\mf{b}, \mf{b}'\in B$ and $h, h'\in H$; the above multiplication is unital with unit $1_B\times 1$.

Similarly, owing to \cite{db}, to a $YD$-coalgebra $B$ one can associate an $H$-bimodule coalgebra; this means a coalgebra within the category of $H$-bimodules 
${}_H{\cal M}_H$, endowed with the monoidal structure coming from its identification to the category of left representations over 
the tensor product quasi-Hopf algebra $H\ot H^{\rm op}$ ($H^{\rm op}$ is the opposite quasi-Hopf algebra associated to $H$). This coalgebra 
structure, denoted by $B\tie H$ and called the smash product coalgebra of $B$ and $H$, is built on the $k$-vector space $B\ot H$ as follows 
(we write $\mf{b}\tie h$ in place of $\mf{b}\ot h$ to distinguish this coalgebra structure of $B\ot H$):
\begin{equation}\eqlabel{comultsmashprodcoalg}
\un{\Delta}(\mf{b}\tie h)=y^1X^1\cdot \mf{b}_{\un{1}}\tie y^2Y^1(x^1X^2\cdot \mf{b}_{\un{2}})_{[-1]}x^2X^3_1h_1\ot 
y^3_1Y^2\cdot (x^1X^2\cdot \mf{b}_{\un{2}})_{[0]}\tie y^3_2Y^3x^3X^3_2h_2
\end{equation}
and $\un{\va}(\mf{b}\tie h)=\un{\va}_B(\mf{b})\va(h)$, for all $\mf{b}\in B$, $h\in H$.

When $B$ is a YD-bialgebra, $B\# H$ and $B\tie H$ determine a quasi-bialgebra structure on $B\ot H$ with reassociator $1_B\ot X^1\ot 1_B\ot X^2\ot 1_B\ot X^3$,  
denoted in what follows by $B\times H$ and called the biproduct quasi-bialgebra of $B$ and $H$. Furthermore, 
$B\times H$ is a quasi-Hopf algebra with antipode 
\begin{equation}\eqlabel{antipbipr}
{\cal S}(\mf{b}\times h)=(1_B\times S(X^1x^1_1\mf{b}_{[-1]}h)\a)(X^2x^1_2\cdot \un{S}_B(\mf{b}_{[0]})\times X^3x^2\b S(x^3))
\end{equation}
and distinguished elements $1_B\times \a$ and $1_B\times \b$, provided that $B$ is a YD-Hopf algebra with antipode $\un{S}_B$; we refer to \cite{bn} for more details.   

According to \cite{bn}, biproduct quasi-Hopf algebras characterize the quasi-Hopf algebras with a projection: if there exist quasi-Hopf algebra morphisms 
$
\xymatrix{
H \ar[r]<2pt>^i &\ar[l]<2pt>^{\pi} A
}
$
such that $\pi i=\Id_H$ then there exists a Hopf algebra $B$ in ${}_H^H{\cal YD}$ such that $A$ is isomorphic to $B\times H$ as a quasi-Hopf algebra. 
 \section{Two-dimensional Hopf algebras within Yetter-Drinfeld module categories}\label{2dH}
\setcounter{equation}{0}
Throughout this section $k$ is a field of characteristic different from $2$ and $H$ is a quasi-Hopf algebra with bijective antipode. 
Inspired by the Hopf case \cite{Rbke}, we characterize the $2$-dimensional Hopf algebras $B$ in ${}_H^H{\cal YD}$. To this end, 
we start by describing the YD-algebras of dimension $2$ in ${}_H^H{\cal YD}$. 

Let $B$ be a $2$-dimensional vector space with basis $\{1, n\}$ that admits a left $H$-module algebra structure. Since $n^2+rn+s1=0$ for some 
$r, s\in k$, as in \cite{Rbke} one can assume that $n^2=\omega1$ for a certain $\omega\in k$ (we replace $n$ with $m:=n+\frac{r}{2}1$; then 
$\{1, m\}$ is a basis of $B$ and $m^2=\omega 1$, where $\omega=\frac{r^2}{4}-s$). If $\cdot: H\ot B\ra B$ provides the left $H$-module structure on $B$, 
it follows that $\cdot$ is determined by two $k$-linear maps $\alp, \be: H\ra k$, this means elements in the dual space $H^*$ of $H$, in the sense that, 
for all $h\in H$, $h\cdot 1=\va(h)1$ and  
\begin{equation}\label{qexpro}
h\cdot n = \be(h)1 + \alp(h)n. 
\end{equation}   

We endow $H^*$ with (possible non-associative) multiplication given by the convolution product: $(\alp\be)(h)=\alp(h_1)\be(h_2)$, for all 
$\alp, \be\in H^*$; it is unital with unit $\va$, the counit of $H$. 

\begin{lemma}\lelabel{malgstrB}
Giving an $H$-module algebra structure on a $2$-dimensional vector space $B$ with basis $\{1, n\}$ as above ($n^2=\omega 1$ for a certain $\omega\in k$) 
is equivalent to giving $\alp, \be\in H^*$ such that $\alp$ is a $k$-algebra map and the following relations are satisfied (for all $h, h'\in H$):
\begin{eqnarray}
&&\be (hh')=\va(h)\be(h')+\be(h)\alp(h')~,~\be \alp =- \alp\be~,~\be^2=\omega(\va-\alp^2),\eqlabel{ma1B}\\
&&\be(X^1)\be(X^2)\be(X^3)+\omega[\be(X^1)\alp(X^2X^3) + \alp(X^1X^2)\be(X^3) + \alp(X^1X^3)\be(X^2)]=0,\label{ma2B}\\
&&\omega[1-\alp(X^1X^2X^3)]=\be(X^1)\be(X^2)\alp(X^3) +  \alp(X^1)\be(X^2)\be(X^3) + \be(X^1)\alp(X^2)\be(X^3).\label{ma3B}
\end{eqnarray}
\end{lemma} 
\begin{proof}
Since $h\cdot (h'\cdot n)=hh'\cdot n$, for all $h, h'\in H$, and $\{1, n\}$ is a basis of $B$ it follows that $\alp$ is an algebra map and 
$\be$ is an $(\va, \alp)$-derivation (this means, the first equality in \equref{ma1B} holds). These two conditions on $(\alp, \be)$ characterize the 
$H$-module structures on $B$. 

We see now when $B$ is, moreover, an $H$-module algebra. Looking at the associativity of the multiplication $\un{m}_B$ of $B$, 
one can show that (\ref{modalg2}) holds if and only if it holds for $\mf{b}=\mf{b}'=\mf{b}{''}=n$; the latter is equivalent to 
\[
\omega n=[\be(X^1)1+\alp(X^1)n][\be(X^2)1+\alp(X^2)n][\be(X^3)1 + \alp(X^3)n],
\] 
and a simple computation ensures us that this happens just when (\ref{ma2B}, \ref{ma3B}) are satisfied. 

Finally, similar to the computation performed in \cite[Proposition 1]{Rbke}, the condition (\ref{modalg1}) is satisfied if and only if 
it is verified, for all $h\in H$, by $\mf{b}=\mf{b}'=n$, if and only if $\be\alp+\alp\be=0$ and $\be^2+\omega\alp^2=\omega\va$.
\end{proof}

We move now to the coalgebra case. From now on, assume that $B$ is a $2$-dimensional $H$-module algebra as in \leref{malgstrB}.

\begin{lemma}\lelabel{algcoalgstrB}
Let $B$ be as in the above and assume that $B$ admits a left $H$-module coalgebra structure given by $(\un{\Delta}_B, \un{\va}_B)$ 
such that $\un{\Delta}_B(1)=1\ot 1$ and $\un{\va}_B(1)=1_k$. Then $\be=0$, $\omega(\alp^2-\va)=0$, $\omega[1-\alp(X^1X^2X^3)]=0$, 
and the coalgebra structure of $B$ is completely determined by two 
elements $b, b'\in B$ obeying the following conditions:
\begin{eqnarray}
&&\un{\Delta}_B(b)=b\ot 1 + b'\ot b,~\un{\Delta}_B(b')=b'\ot b',~\un{\va}_B(b)=0,~\un{\va}_B(b')=1_k,\\
&&h\cdot b=\alp(h)b,~h\cdot b'=\va(h)b',~\forall~h\in H.
\end{eqnarray}   
In addition, $n=b+\un{\va}_B(n)b'$ and $\un{\Delta}_B(n)=b\ot 1+ b'\ot n$.
\end{lemma}
\begin{proof}
It is similar to the one given in \cite[Corollary 1]{Rbke}, only that this time it is more complicate to show that $\be=0$. 

We write $\un{\Delta}_B(n)=b\ot 1 + b'\ot n$, for some $b, b'\in B$. From the counit properties, we get that $\un{\va}_B(b)=0$, 
$\un{\va}_B(b')=1_k$ and $n=b+\un{\va}_B(n)b'$. Now, the coassociativity of $\un{\Delta}_B$ is fulfilled if and only if 
\equref{ydc2} is verified for $\mf{b}=n$, if and only if 
\begin{eqnarray*}
&&\un{\Delta}_B(b)=b\ot 1 + b'\ot b + \be(x^3)x^1\cdot b'\ot x^2\cdot b'~\mbox{and}\\
&&\un{\Delta}_B(b')=\alp(x^3)x^1\cdot b'\ot x^2\cdot b'.
\end{eqnarray*}  

Likewise, \equref{ydc1} holds if and only if it holds for $\mf{b}=n$ or, equivalently, if 
\begin{equation}\eqlabel{intermform1}
\be(h)1+\alp(h)b=h\cdot b+\be(h_2)h_1\cdot b'~\mbox{and}~\alp(h)b'=\alp(h_2)h_1\cdot b',
\end{equation}
for all $h\in H$. We claim that the two conditions above are equivalent to the fact that $h\cdot b'=\va(h)b'$, $h\cdot b=\alp(h)b$, 
for all $h\in H$, and $\be=0$. Indeed, for $q_R$ as in \equref{qr}, we compute that, for all $h\in H$, 
\begin{eqnarray*}
\alp(S^{-1}(h_2))\alp(q^2h_{12})q^1h_{11}\cdot b'&=&
\alp(S^{-1}(h_2)q^2h_{12})q^1h_{11}\cdot b'\\
&\equal{\equref{qr1a}}&\alp(q^2)hq^1\cdot b'.
\end{eqnarray*}
On the other hand, by using the fact that $\alp$ is an algebra morphism and the second relation in \equref{intermform1} we get that 
\begin{eqnarray*}
\alp(S^{-1}(h_2))\alp(q^2h_{12})q^1h_{11}\cdot b'&=&\alp(S^{-1}(h_2))\alp(q^2)\alp(h_1)q^1\cdot b'\\
&=&\alp(S^{-1}(S(h_1)\a h_2))\alp(S^{-1}(X^3)X^2)X^1\cdot b'\\
&=&\va(h)\alp(q^2)q^1\cdot b',
\end{eqnarray*}
for all $h\in H$. Thus, we have shown that $\alp(q^2)hq^1\cdot b'=\va(h)\alp(q^2)q^1\cdot b'$, for all $h\in H$. 
Consequently, 
\begin{eqnarray*}
\alp(S^{-1}(\a))h\cdot b'&=&\alp(S^{-1}(x^3)x^2)\alp(q^2)hx^1q^1\cdot b'\\
&=&\va(hx^1)\alp(S^{-1}(x^3)x^2)\alp(q^2)q^1\cdot b'=\va(h)\alp(q^2)q^1\cdot b',
\end{eqnarray*}
for all $h\in H$. Together with \equref{pqra} and the second equality in \equref{intermform1} this implies that, for all $h\in H$,  
\begin{eqnarray*}
\va(h)b'&=&\va(h)\alp(S^{-1}(p^2)q^2p^1_2)q^1p^1_1\cdot b'\\
&=&\va(h)\alp(S^{-1}(p^2)p^1)\alp(q^2)q^1\cdot b'\\
&=&\alp(S^{-1}(\a p^2)p^1)h\cdot b'=h\cdot b',
\end{eqnarray*}
as stated. Now, by considering $b=n-\un{\va}_B(n)b'$ in \equref{intermform1} we get that 
$\be(h_2)h_1\cdot b'-\un{\va}_B(n)h\cdot b'=-\un{\va}_B(n)\alp(h)b'$, for all $h\in H$, and so 
$(\be(h)-\va(h)\un{\va}_B(n))b'=-\un{\va}_B(n)\alp(h)b'$, for all $h\in H$. As $\un{\va}_B(b')=1_k$, the last equality 
is equivalent to $\be=\un{\va}_B(n)(\va-\alp)$; in particular, $\alp$ and $\be$ commute and anti-commute and the same time. 
Since we work over a field of characteristic different from $2$, it follows that $\be\alp=\alp\be=0$ in $H^*$. Thus, with the help of 
\equref{pqr} we see that 
\begin{eqnarray*}
\be(h)&=&\be(q^1_1p^1h)\alp(q^1_2p^2S(q^2))\\
&\equal{\equref{qr1}}&
\be((q^1h_1)_1p^1)\alp((q^1h_1)_2p^2S(q^2h_2))\\
&=&\be(p^1)\alp(q^1h_1p^2S(q^2h_2))+\alp(p^1)\be((q^1h_1)_1)\alp((q^1h_1)_2)\alp(p^2S(q^2h_2))\\
&=&\va(h)\be(x^1)\alp(q^1\b S(q^2))\alp(x^2S(x^3))=\va(h)\kappa,
\end{eqnarray*}   
where in the third equality we used that $\be$ is an $(\va, \alp)$-derivation; here $\kappa:=\be(x^1)\alp(x^2S(x^3))$. Then, 
the $(\va, \alp)$-derivation condition on $\be$ reads as $\alp(h')\kappa=0$, for all $h'\in H$, and therefore $\be=0$ and 
$h\cdot b=\alp(h)b$, for all $h\in H$. The remaining details are left to the reader.  
\end{proof}

Thus, a $2$-dimensional $H$-module algebra and coalgebra $B$ is determined by an algebra morphism $\alp\in H^*$, a scalar 
$\omega$ and two elements $b, b'\in B$. Next we see when $B$ is a (co)algebra in ${}_H^H{\cal YD}$; we denote by $\r: B\ni \mf{b}\mapsto 
\mf{b}_{[-1]}\ot \mf{b}_{[0]}\in H\ot B$ a left coaction of 
$H$ on $B$ and assume that $\r(1)=1\ot 1$. 

\begin{lemma}\lelabel{coalgYDstrB}
For $B$ as above, the left Yetter-Drinfeld module structures on $B$ are determined by two elements $u, v\in H$ satisfying $\va(u)=0$, 
$\va(v)=1$ and 
\begin{eqnarray}
&&\Delta(u)=\alp(x^1)ux^2\ot x^3 + \alp(x^1X^2)X^1vx^2\ot uX^3x^3,~\Delta(v)=\alp(x^3X^2y^1)x^1X^1vy^2\ot x^2vX^3y^3,\eqlabel{coasscoactonB}\\
&&hu=\alp(h_1)uh_2~\mbox{and}~\alp(h_2)h_1v=\alp(h_1)vh_2,~\forall~h\in H.\eqlabel{YDmodcondB}
\end{eqnarray} 

Furthermore, 

$\bullet$ $B$ is a coalgebra in ${}_H^H{\cal YD}$ if and only if $u=\un{\va}_B(n)(1-v)$ and 
\begin{eqnarray}
&&u\ot 1 + v\ot b=b_{[-1]}\ot b_{[0]} + b'_{[-1]}u\ot b'_{[0]},\label{YDcoalgB1}\\
&&v\ot b'=\alp(X^3)X^1b'_{[-1]}v\ot X^2\cdot b'_{[0]};\label{YDcoalgB2}
\end{eqnarray}

$\bullet$ $B$ is an algebra in ${}_H^H{\cal YD}$ if and only if 
\begin{eqnarray}
&&u[1+\alp(x^1x^3)x^2]=0,\label{YDalgB1}\\
&&u^2=\omega[\alp(x^1x^3X^2X^3)X^1vx^2v-\alp(y^1y^2)y^3].\label{YDalgB2}
\end{eqnarray}
\end{lemma}  
\begin{proof}
It is clear that $\r$ is defined by $\r(n)$, which can be written as $\r(n)=u\ot 1 + v\ot n$, for some $u, v\in H$. As  
$\r$ is counital, we have that $\va(u)=0$ and $\va(v)=1$. Also, the coassociativity of $\r$ in (\ref{y1}), which reduces to the case when 
$\mf{b}=n$, is equivalent to 
\begin{eqnarray*}
&&u\ot 1\ot 1 + \alp(X^2)X^1v\ot uX^3\ot 1 + \alp(X^2)X^1v\ot vX^3\ot n\\
&&\hspace{1.5cm}
=\alp(Y^1)\Delta(u)(Y^2\ot Y^3)\ot 1 +\alp(X^3Y^1)(X^1\ot X^2)\Delta(v)(Y^2\ot Y^3)\ot n.
\end{eqnarray*} 
Since $\{1, n\}$ is a basis of $B$, the above relation is equivalent to the two conditions in \equref{coasscoactonB}. 
In a similar manner one can prove that (\ref{y3}) reduces in our case to the two conditions in \equref{YDmodcondB}. In fact, it is suffices 
for (\ref{y3}) to be verified for $b=n$, meaning that $hu\ot 1 + \alp(h_2)h_1v\ot n=\alp(h_1)uh_2\ot 1 + \alp(h_2)vh_2\ot n$, for all $h\in H$. 
Clearly, the last condition holds if and only if the two conditions in \equref{YDmodcondB} are satisfied. We check now when $B$ is a coalgebra in ${}_H^H{\cal YD}$.

From the third equality in \equref{ydc1} we get that $\un{\va}_B(1)u+\un{\va}_B(n)v=\un{\va}_B(n)1$, which is equivalent to $u=\un{\va}_B(n)(1-v)$; in particular 
$u$ and $v$ commute. Similarly, \equref{ydc3} holds for any $\mf{b}\in B$ if and only if it holds for $n$; by using the fact that $\un{\Delta}_B(n)=b\ot 1+ b'\ot n$, 
one can see easily that \equref{ydc3} holds for $\mf{b}=n$ if and only if 
\[
u\ot 1\ot 1 + v\ot b\ot 1 + v\ot b'\ot n=b_{[-1]}\ot b_{[0]}\ot 1 + b'_{[-1]}u\ot b'_{[0]}\ot 1 + \alp(X^3)X^1b'_{[-1]}v\ot X^2\cdot b'_{[0]}\ot n.
\]
By using again the fact $\{1, n\}$ is a basis for $B$, it is immediate that the above equality is equivalent to the equalities (\ref{YDcoalgB1}, \ref{YDcoalgB2}).  

The condition that $B$ is an algebra in ${}_H^H{\cal YD}$ reduces to the validity of (\ref{modalg3}) for our context. As in the coalgebra case, it can be easily verified that in our case (\ref{modalg3}) holds if and only if 
\begin{eqnarray}
&&\alp(x^1x^3Y^1Y^2)ux^2vY^3 +\alp(Y^1Y^2)vuY^3=0,\label{YDalgB1a}\\
&&\alp(x^1x^3Y^1Y^2)ux^2uY^3 +\omega \alp(x^1x^3Y^1Y^2X^2X^3)X^1vx^2vY^3=\omega 1.\label{YDalgB2a}
\end{eqnarray}
As in the Hopf case, see the prof of \cite[Corollary 1]{Rbke}, 
we show that $v$ has a right inverse, implying that (\ref{YDalgB1a}, \ref{YDalgB2a}) and (\ref{YDalgB1}, \ref{YDalgB2}) are equivalent. 

Recall that $u$ and $v$ commute. For $h\in H$, we have that 
\begin{eqnarray*}
hv&\equal{\equref{pqra}}&\alp(S^{-1}(p^2)q^2p^1_2)hg^1p^1_1v\\
&\equal{\equref{qr1a}}&\alp(S^{-1}(h_2p^2)q^2(h_1p^1)_2)q^1(h_1p^1)_1v\\
&\equal{\equref{YDmodcondB}}&\alp(S^{-1}(h_2p^2)q^2)\alp(h_{11}p^1_1)q^1vh_{12}p^1_2=\alp(q^2)q^1v\un{h},
\end{eqnarray*}
where, for $h\in H$, we denote $\un{h}:=\alp(S^{-1}(h_2p^2))\alp(h_{11}p^1_1)h_{12}p^1_2$. Now, from the formula of 
$\Delta(v)$ in \equref{coasscoactonB}, since $\va(v)=1$,  we deduce that 
\[
1=q^1\b S(q^2)=q^1v_1\b S(q^2v_2)=\alp(y^1X^2x^3)q^1x^1X^1vy^2\b S(q^2x^2vX^3y^3)=\alp(Q^2)Q^1v\hbar,
\]
where $Q^1\ot Q^2$ is a second copy of $q_R$ and $\hbar:=\alp(y^1X^2x^3)\un{q^1x^1X^1}y^2\b S(q^2x^2vX^3y^3)$. So 
$\alp(Q^2)Q^1v$ has a right inverse. Since $\alp$ is an algebra morphism, \equref{q6} implies that $\alp(Q^2)Q^1$ is invertible in 
$H$ with inverse given by $\alp(S^{-1}(x^3)x^2)\alp(S^{-1}(\a))^{-1}x^1$, where 
$\alp(S^{-1}(\a))^{-1}=\alp(S^{-1}(S(x^1)x^2\b S(x^3))=\alp(S^{-1}(X^1\b S(X^2)X^3))$. Hence $v$ has a right inverse, as stated.  
\end{proof}

A last preliminary result that we need is the following.

\begin{proposition}\prlabel{casebprimeone}
Let $B$ be a $2$-dimensional bialgebra in ${}_H^H{\cal YD}$ with basis $\{1, n\}$, determined by $\alp\in H^*$, $\omega\in k$, $u, v\in H$ and 
$b, b'\in B$ as in the above. Then 
\begin{equation}\eqlabel{braidedbialgB}
b^2+\alp(u)b'b+ \omega b'^2=\omega 1~\mbox{and}~\alp(v)b'b+bb'=0.
\end{equation} 
Consequently, if $b'=1$ then $u=0$.
\end{proposition}
\begin{proof}
We have studied all the conditions required for $B$ to be an algebra and a coalgebra in ${}_H^H{\cal YD}$. The only one left in order to have $B$ a bialgebra in 
${}_H^H{\cal YD}$ is (\ref{moltcon2}), which holds in our case if and only if it holds for $\mf{b}=\mf{b}'=n$. By using the relations obtained so far, 
one can show that $\un{\Delta}_B$ is multiplicative if and only if 
\[
b^2\ot 1 + bb'\ot n+ \alp(u)b'b\ot 1+ \alp(v)b'b\ot n + \omega b'^2\ot 1=\omega 1\ot 1.
\]
It is clear that the above relation is equivalent to the relations in \equref{braidedbialgB}, and this completes the description of $B$ as a braided bialgebra.   

Assume now $b'=1$, so $n=b+\un{\va}_B(n)1$. As $\{1, n\}$ is a basis in $B$, we cannot have $b=0$, and therefore $\{1, b\}$ is another basis for $B$. 
In addition, the first equality in \equref{braidedbialgB} yields $b^2=-\alp(u)b$; thus $n^2=\omega 1$ leads us to 
$\alp(u)=2\un{\va}_B(n)$ and $\omega=\un{\va}_B(n)^2$. 

By way of contradiction, assume that $u\not=0$. Then $u=\un{\va}_B(n)(1-v)$ implies $\un{\va}_B(n)\not=0$, and so $\omega\not=0$ and 
$\alp(u)\not=0$. Hence \equref{ma1B} entails to $\alp^2=\va$. It follows that $v=1-2\alp(u)^{-1}u$, and 
the second equality in \equref{YDmodcondB} reads as 
\[
\alp(h_2)h_1(1-2\alp(u)^{-1}u)=\alp(h_1)(1-2\alp(u)^{-1}u)h_2,
\]
for all $h\in H$, which, owing to the first equality in \equref{YDmodcondB}, it is equivalent to 
\[
\alp(h_2)h_1 - \alp(h_1)h_2=2\alp(u)^{-1}(\alp(h_2)h_1 - h)u,
\]
for all $h\in H$. By applying $\alp$ to the both sides of the above equality, since $\alp^2=\va$, we get that $2\alp(u)^{-1}(\va-\alp)=0$, hence 
$\alp=\va$. But this yields $0=\va(u)=\alp(u)\not=0$, a contradiction. 
\end{proof}

We are now able to describe the $2$-dimensional braided Hopf algebras $B$ within ${}_H^H{\cal YD}$, provided that $H$ is a quasi-Hopf algebra 
with bijective antipode. To simplify the presentation, we introduce the following family of braided Hopf algebras.

\begin{definition}\delabel{Balpv}
Let $H$ be a quasi-Hopf algebra and $(\alp, v)$ a couple consisting of an algebra map $\alp: H\ra k$ and an element $v\in H$ verifying 
$\alp(v)=-1$, $\va(v)=1$, and the second formulas in \equref{coasscoactonB} and \equref{YDmodcondB}. Then $B_{\alp, v}$ is the  
the $2$-dimensional vector space with basis $\{1, n\}$ endowed with the following braided Hopf algebra structure:

$\bullet$ $B_{\alp, v}$ belongs to ${}_H^H{\cal YD}$ via the $H$-action determined by $h\cdot 1=\va(h)1$, $h\cdot n=\alp(h)n$, for all 
$h\in H$, and the $H$-coaction given by $1\mapsto 1\ot 1$, $n\mapsto v\ot n$;

$\bullet$ as an algebra, $B_{\alp, v}$ is generated by $n$ with relation $n^2=0$, and $B_{\alp, v}$ is unital with unit $1$;

$\bullet$ the coalgebra structure of $B_{\alp, v}$ is defined by $\un{\Delta}(1)=1\ot 1$, $\un{\Delta}(n)=n\ot 1+ 1\ot n$ and $\un{\va}(1)=1$, 
$\un{\va}(n)=0$;

$\bullet$ the braided antipode $\un{S}$ is determined by $\un{S}(1)=1$, $\un{S}(n)=-n$.  
\end{definition}

The fact that $B_{\alp, v}$ is indeed a braided Hopf algebra will follow from the proof of the next theorem. Also, denote by $C_2$ the cyclic group of 
order $2$ and let $k[C_2]$ be the group Hopf algebra of $C_2$ over $k$. It is easy to see that the trivial $H$-(co)module 
structure (given by the counit, respectively by the unit, of $H$) turns $k[C_2]$ into an object of ${}_H^H{\cal YD}$ and, moreover, into a braided 
Hopf algebra in ${}_H^H{\cal YD}$ (the braided Hopf structure equals the one considered in the Hopf case); in what follows, we will refer to this braided Hopf 
algebra structure of $k[C_2]$ as being the trivial one.   

The next result can be viewed as the quasi-Hopf version of \cite[Theorem 1]{Rbke}.  

\begin{theorem}\thlabel{2dimbraidedHA}
Let $H$ be a quasi-Hopf algebra with bijective antipode over a field $k$ of characteristic different from $2$. If $B$ is a braided Hopf algebra 
in ${}_H^H{\cal YD}$ then $B$ is isomorphic to a braided Hopf algebra of type $B_{\alp, v}$ or to the group Hopf algebra $k[C_2]$ considered 
in ${}_H^H{\cal YD}$ via the trivial structure.  
\end{theorem}
\begin{proof}
As before, we consider $B$ determined by $\alp\in H^*$, $\omega\in k$, $u, v\in H$ and $b, b'\in B$ satisfying the above mentioned relations. Parallel to 
the Hopf case, we distinguish two cases.

{\bf Case 1}. If $b'=1$, by \prref{casebprimeone} we get $u=0$; clearly, $b\not=0$. Thus the second equality in \equref{braidedbialgB} is equivalent to $\alp(v)=-1$; consequently, $v\not=1$ and so   
$u=\un{\va}_B(n)(1-v)$ is equivalent to $\un{\va}_B(n)=0$. Therefore, $n=b$ and $\omega 1=n^2=b^2=-\alp(u)b=0$; see the last part of the proof of 
\prref{casebprimeone}. We conclude that $\omega=0$ and $n^2=0$. Now, a simple inspection tells us that the braided Hopf algebra $B$ is nothing but 
$B_{\alp, v}$ presented in \deref{Balpv}; the details are left to the reader. In particular, this shows that $B_{\alp, v}$ is indeed a Hopf algebra in 
${}_H^H{\cal YD}$, as claimed.   

{\bf Case 2}. Assume that $b'\not=1$. Since $\un{\va}_B(b')=1$, it follows that $\{1, b'\}$ is a basis of $B$ for which $h\cdot 1=\va(h)1$ and 
$h\cdot b'=\va(h)b'$, for all $h\in H$. We deduce from here that $h\cdot \mf{b}=\va(h)\mf{b}$, for all $\mf{b}\in B$. But $h\cdot n=\alp(h)n$, for all 
$h\in H$, and therefore $\alp=\va$ in this case. We obtain from (\ref{YDalgB1}) that $u=0$, and from (\ref{YDcoalgB2}) that 
$b'_{[-1]}v\ot b'_{[0]}=v\ot b'$. We know from the proof of \leref{coalgYDstrB} that, in general, $v$ is right invertible, hence $b'_{[-1]}\ot b'_{[0]}=1\ot b'$. 
Keeping in mind that $\{1, b'\}$ is a basis of $B$ and the fact that $\r(1)=1\ot 1$, it follows that $\r(\mf{b})=1\ot \mf{b}$, for all $\mf{b}\in B$. 
By taking $\mf{b}=n$, with the help of $\r(n)=u\ot 1 + v\ot n=v\ot n$, we conclude that $v=1$. Summing up, $u=0$, $v=1$ and $\alp=\va$; again, a simple inspection guarantees us that $B$ is in this case the group Hopf algebra $k[C_2]$, equipped with the trivial braided Hopf algebra structure. For this, 
note that 
\[
(\un{\Delta}_B(n))^2=(b\ot 1 + b'\ot n)^2=b^2\ot 1 + b'^2\ot n^2=(b^2\ot \omega b'^2)\ot 1\equal{\equref{braidedbialgB}}\omega 1\ot 1=\un{\Delta}_B(n^2),
\]
since the second relation in \equref{braidedbialgB} reduces to $b'b=-bb'$; but $\{1, b'\}$ is a basis for $B$, so $B$ is a commutative algebra, and therefore 
$bb'=b'b=0$. The above equality tells us that $\un{\Delta}_B$ is multiplicative in the usual way, hence 
$\un{\Delta}_B(b'^2)=(\un{\Delta}_B(b'))^2=b'^2\ot b'^2$. As $\un{\va}_B(b')=1$, it follows that $b'^2$ is a grouplike element of the (ordinary) Hopf algebra $B$, 
hence $1, b', b'^2$ are linearly independent elements of $B$, provided that they are pairwise distinct. Since $B$ is $2$-dimensional it follows that 
$b'^2\in \{1, b'\}$. But we required for $B$ to have an antipode, and this occurs only in the case when $b'^2=1$; for more details we refer to 
\cite[Exercise 4.3.7]{DNRha}.        
\end{proof}

For further use, we present the quasi-Hopf algebra structures of $B\times H$ coming from $2$-dimensional braided Hopf algebras $B$. As before, we work over a 
field of characteristic different from $2$. Also, remark that, in general, $B\times H$ is a free right $H$-module of rank equals ${\rm dim}_k(B)$. 

\begin{proposition}\prlabel{qHa with a proj of rank 2}
Let $A$ be a quasi-Hopf algebra with a projection $
\xymatrix{
H \ar[r]<2pt>^i &\ar[l]<2pt>^{\pi} A
}
$, 
$\pi i=\Id_H$, such that $A$ is a free right $H$-module of rank $2$. The $A$ identifies either to $H_g$ or $H(\theta)_{\zeta, v}$, where: 

$\bullet$ $H_g$ is the unital $k$-algebra generated by $g$ and $H$ with relations $g^2=1$ and $gh=hg$, for all $h\in H$. $H$ is a quasi-Hopf 
subalgebra of $H_g$, and $\Delta(g)=g\ot g$, $\va(g)=1$ and $S(g)=g$ complete the quasi-Hopf algebra structure of $H_g$. In other words, $H_g$ is the tensor product 
of the quasi-Hopf algebras $k[C_2]$ and $H$, the former being seen as the group Hopf algebra with trivial reassociator. 

$\bullet$ Let $(\alp, v)$ be a couple  consisting of an algebra morphism $\alp: H\ra k$ and an element $v\in H$ subject to relations
\begin{equation}\eqlabel{2dim bialgebra cond}
\alp(v)=-1,~\va(v)=1,~\Delta(v)=\alp(y^1x^3X^2)x^1X^1vy^2\ot x^2vX^3y^3~\mbox{and}
\end{equation}
$\alp(h_2)h_1v=\alp(h_1)vh_2$, for all $h\in H$. Then $H(\theta)_{\alp, v}$ is the unital $k$-algebra generated by $\theta$ and $H$ with relations 
$\theta^2=0$ and $h\theta=\alp(h_1)\theta h_2$, for all $h\in H$. The quasi-Hopf algebra structure of $H(\theta)_{\alp, v}$ is defined in such a way that 
$H$ is a quasi-Hopf subalgebra of it and 
\begin{eqnarray}
&&\Delta(\theta)=\alp(X^2x^1)vX^1x^2\ot \theta X^3x^3 + \alp(x^1)\theta x^2\ot x^3,\\
&&S(\theta)=-\alp(X^2x^1_2)S(X^1x^1_1v)\a \theta X^3x^2\b S(x^3).
\end{eqnarray}
\end{proposition}
\begin{proof}
By \cite{bn}, $A$ identifies to $B\times H$ as a quasi-Hopf algebra, for some $2$-dimensional braided Hopf algebra $B$. If $B=k[C_2]$, $A$ identifies to 
$H_g$. When $B=B_{\alp, v}$, $A$ identifies to $H(\theta)_{\alp, v}$. Note that $\theta\equiv n\times 1$, where $\{1, n\}$ is the basis of $B$ as in 
the Case 1 of the proof of \thref{2dimbraidedHA}; the formulas for $\Delta(\theta)$ and $S(\theta)$ follow from \equref{comultsmashprodcoalg} 
and \equref{antipbipr}, respectively. 
\end{proof}

For a quasi-Hopf algebra $H$, a left integral in $H$ is an element $t\in H$ satisfying $ht=\va(h)t$ for all $h\in H$. Owing to \cite[Theorem 2.5]{bc1}, 
the space of left integrals in $H$ is non-zero if and only if $H$ is finite dimensional, provided that the antipode of $H$ is bijective. If this is the case then 
the space of left integrals in $H$ has dimension one. As far as we are concerned, the importance of integrals in $H$ resides on the fact that they characterize the 
semisimplicity of $H$: $H$ is semisimple as an algebra if and only if there exists a left integral $t$ in $H$ such that $\va(t)=1$; see \cite{pan}.  

\begin{corollary}\colabel{2classifgen}
Let $A, H$ be quasi-Hopf algebras over a field $k$ of characteristic different from $2$ and suppose that 
there are quasi-Hopf algebra morphisms 
$
\xymatrix{
H \ar[r]<2pt>^i &\ar[l]<2pt>^{\pi} A
}
$
such that $\pi i=\Id_H$. If $A$ is a free right $H$-module of rank $2$ and $H$ is finite dimensional then the following assertions hold. 

$(i)$ If $A$ is semisimple then so is $H$ and $A\simeq H_g$ as a quasi-Hopf algebra. 

$(ii)$ If $A$ is not semisimple and $H$ is semisimple then  $A\simeq H(\theta)_{\alp, v}$ as a quasi-Hopf algebra, for some $\alp, v, \theta$ as in the above;

$(iii)$ If both $A, H$ are not semisimple then $A\simeq H_g$ or $A\simeq H(\theta)_{\alp, v}$ as a quasi-Hopf algebra. 
\end{corollary}
\begin{proof}
We have $A\simeq B\times H$, for some $2$-dimensional braided Hopf algebra $B$; so \thref{2dimbraidedHA} applies. 
 
As in \cite[\S 2 Proposition 3]{rad}, one can show that a left integral $t$ in $B\times H$ is of the form $t_B\times t_H$ with $t_B\in B$ and $t_H\in H$ such that 
$\mf{b}t_B=\un{\va}_B(\mf{b})t_B$, for all $\mf{b}\in B$, and $h\cdot t_B=\va(h)t_B$, $ht_H=\va(h)t_H$, for all $h\in H$. Consequently, $B\times H$ is semisimple 
if and only if $H$ is semisimple and $t_B$ satisfies, moreover, $\un{\va}_B(t_B)=1_k$. Note that, for $B=B_{\alp, v}$ such an element $t_B$ does not exist. On the other hand, for $B=k[C_2]$ one can take $t_B=\frac{1}{2}(1+g)$, where $g$ is the generator of $C_2$.

(i). Since $A$ is finite dimensional and semisimple and $H$ can be identified to a quasi-Hopf subalgebra of $A$, by \cite[Corollary 7.82]{bcpvo} we obtain that 
$H$ is semisimple. By the above comments it follows that $A\simeq H_g\equiv k[C_2]\ot H$ as a quasi-Hopf algebra.

(ii) and (iii). If $A$ is not semisimple then either $H$ is not semisimple and $B$ is arbitrary or $H$ is semisimple and $B=B_{\alp, v}$.   
\end{proof}

\section{Quasi-Hopf algebras of dimension $4$}\selabel{class4dim}
\setcounter{equation}{0}
Throughout this section $k$ is an algebraic closed field of characteristic $0$. We need this assumption on $k$ in order to make use of the classification of 
Hopf algebras $H$ in dimension $4$: if $H$ is not semisimple then $H$ is isomorphic to $H_4$; if $H$ is semisimple, $H$ is isomorphic to a group Hopf algebra. 
Our first aim was to provide a similar result for quasi-Hopf algebras, and especially to produce a quasi-Hopf analogue for $H_4$; as we will see, a result of 
Etingof and Gelaki \cite{eg3} says that the latter is not possible. Nevertheless, we will use the mentioned result to show that, in dimension 4, there are infinitely many non-isomorphic quasi-Hopf algebras. 
\subsection{The nonsemisimple case}
The first example of a non-commutative and non-cocommutative Hopf algebra is the Sweedler's $4$-dimensional Hopf algebra $H_4$. As an algebra, it is unital, 
generated by $g, x$ with relations $g^2=1$, $x^2=0$, $xg=-gx$. Then the Hopf algebra structure of $H_4$ is given by 
\[
\Delta(g)=g\ot g,~\Delta(x)=g\ot x + x\ot 1,~\va(g)=1,~\va(x)=0,~S(g)=g~\mbox{and}~S(x)=-gx.
\]

It was shown in \cite[Corollary 3.6]{eg3} that a $4$-dimensional nonsemisimple quasi-Hopf algebra $A$ is twist equivalent to $H_4$. The proof in \cite{eg3} 
was done in two steps. Namely, the authors show that (1) ${\rm gr}(A)$ is a Hopf algebra isomorphic to $H_4$, and therefore a selfdual Hopf algebra, and that $(2)$ 
the third Hochschild cohomology of $H_4$ with coefficients in $k$, $H^3(H_4, k)$, is trivial; then, the classification result follows from \cite[Proposition 2.3]{eg4} which states that a finite dimensional quasi-Hopf algebra $A$ with ${\rm gr}(A)$ a Hopf algebra is twist equivalent to a Hopf algebra, 
provided that $H^3({\rm gr}(A)^*, k)=0$ (${\rm gr}(A)$ is the graded algebra associated to $A$, more details will be given below).   

Next, we present an alternative proof for (1). Towards this end, for a $4$-dimensional non-semisimple quasi-Hopf algebra $A$ we show that ${\rm gr}(A)\simeq A\simeq H_4$ 
as an algebra, and then we use the representative $3$-cocycles of $k[C_2]$ to see that ${\rm gr}(A)\simeq H_4$ as a Hopf algebra.  

Let $A$ be a finite dimensional nonsemisimple quasi-Hopf algebra and denote by $J$ the Jacobson radical of $A$. Then $J\not=0$ is a nilpotent ideal of $A$ and the quotient algebra $\frac{A}{J}$ is semisimple. For $m$ a positive integer, let $A_m:=\frac{J^m}{J^{m+1}}$; by convention, $J^0=A$. 
Denote by ${\rm gr}(A)=\bigoplus_{m\geq 0}A_m$ the graded algebra associated to the filtration of $A$ given by the powers of $J$. When $A$ is a basic algebra, this means that all the irreducible representations of $A$ are $1$-dimensional or, equivalently, $\frac{A}{J}\simeq k\times\cdots\times k$ as an algebra, 
$J$ is a quasi-Hopf ideal of $A$ (i.e. $J\subseteq A$ is an ideal satisfying $\Delta(J)\subseteq J\ot A + A\ot J$, $\va(J)=0$, $S(J)\subseteq J$; see 
\cite[Lemma 1.1]{GrSo}, the proof in the quasi-Hopf case is identical). Therefore, $A_0=\frac{A}{J}$ admits a unique quasi-Hopf algebra structure such that 
the canonical projection $\pi: A\ra A_0$ turns into a quasi-Hopf algebra morphism. Since $A_0\simeq k\times\cdots\times k$ as an algebra, it follows 
that $A_0\simeq k^G\simeq k[G]$ for some finite abelian group $G$, and therefore $A_0$ is a usual group Hopf algebra viewed as a quasi-Hopf algebra 
via a normalized $3$-cocycle $\phi$ of $G$. In other words, 
$A_0=k_\phi[G]$ and the reassociator $\Phi$ of $A$ projects through $\pi^{\ot 3}$ onto $\phi$. 

Furthermore, ${\rm gr}(A)$ admits a quasi-Hopf algebra structure such that the canonical linear maps 
$
\xymatrix{
A_0 \ar[r]<2pt>^i &\ar[l]<2pt>^{p} {\rm gr}(A)
}
$ 
are quasi-Hopf algebra morphisms obeying $pi=\Id_{A_0}$. If $p_m: J^m\ra A_m$ stands for the canonical $k$-linear projection, so that $p_0=p$, 
the quasi-coalgebra structure of ${\rm gr}(A)$ is defined by $\Delta_m: A_m\ra \sum_{s+t=m}A_s\ot A_t$, $\Delta_m(p_m(x))=\sum_{s+t=m}(p_s\ot p_t)\Delta(x)$, 
for all $x\in J^m$, $m\geq 0$; note that $\Delta_m$ is well defined since, inductively one can show that 
$\Delta(J^m)\subseteq \sum_{s+t=m}J^s\ot J^t$, and $J^{m+1}\subseteq {\rm Ker}(\sum_{s+t=m}(p_s\ot p_t)\Delta)$. Likewise, the antipode of ${\rm gr}(A)$ 
is defined by $S_m:A_m\ra A_m$ given by $S_m(p_m(x))=p_m(S(x))$, for all $x\in J^m$, $m\geq 0$. The other structures of ${\rm gr}(A)$ are induced by the 
quasi-Hopf algebra structure of $A_0$; in fact, we know that there is a braided Hopf algebra $B$ in ${}_{A_0}^{A_0}{\cal YD}$ such that 
${\rm gr}(A)\simeq B\times A_0$, as quasi-Hopf algebras, thus ${\rm dim}_kA_0$ divides ${\rm dim}_k{\rm gr}(A)$. 

Another notable thing about ${\rm gr}(A)$ is the fact that each $A_m$ can be seen as a two-sided two-cosided quasi-Hopf module over $A_0$, with natural 
$A_0$-actions and coactions; we refer to \cite{db} for the detailed definition of the category ${}_{A_0}^{A_0}{\cal M}_{A_0}^{A_0}$. By the structure theorems 
proved in {\it loc. cit.}, it follows that $A_m$ is a free left and right $A_0$-module, for all $m\geq 0$.

When ${\rm dim}_kA=4$, $\frac{A}{J}$ is a semisimple algebra of dimension at most $3$, thus $A$ is basic and the results mentioned above apply. In particular, 
they allow us to show that ${\rm dim}_kJ=2$. To this end, we need another classical result, proved in \cite[Theorem 1.4.9]{abe}. Namely, if $\mathbb{A}$ 
is a $k$-algebra with Jacabson radical $J$ and $\frac{\mathbb{A}}{J}$ is a separable $k$-algebra then there exists a semisimple $k$-subalgebra $\mathbb{S}$ of 
$\mathbb{A}$ such that $\mathbb{A}$ can be written as the direct sum of $k$-spaces $\mathbb{A}=J\oplus \mathbb{S}$; obviously, 
$\frac{\mathbb{A}}{J}\simeq \mathbb{S}$ as $k$-spaces.   
 
\begin{lemma}
Let $A$ be a nonsemisimple quasi-Hopf algebra of dimension $4$ with Jacobson radical $J$. Then ${\rm dim}_kJ=2$.
\end{lemma}  
\begin{proof}
The dimension of $J$ is at most $3$, and we show that it cannot be equal to $1$ or $3$. 

If ${\rm dim}_kJ=1$, as $J^2$ is properly included in $J$, it follows that $J^2=0$. Hence ${\rm gr}(A)=\frac{A}{J}\oplus J\simeq A$ as 
$k$-vector spaces. We get form here that $3={\rm dim}_k(A_0)$ is a divisor of $4={\rm dim}_k{\rm gr}(A)$, a contradiction.

If ${\rm dim}_kJ=3$, $A_0$ is a $1$-dimensional quasi-Hopf algebra. This means that $A_0=k$, the trivial Hopf algebra, and 
therefore ${\rm gr}(A)\simeq B\times A_0\simeq B$ is an ordinary Hopf algebra. If $J^2=0$, as above we get that ${\rm gr}(A)=\frac{A}{J}\oplus J\simeq A$ 
as $k$-spaces, hence ${\rm gr}(A)$ is a $4$-dimensional Hopf algebra whose Jacobson radical is of dimension $3$. We land to a contradiction, since the classification of $4$-dimensional Hopf algebras says that the Jacobson radical must be zero or of dimension $2$. Thus ${\rm dim}_kJ^2\in \{1, 2\}$.   

Assume that ${\rm dim}_kJ^2=1$; this forces $J^3=0$, and therefore ${\rm gr}(A)=A_0\oplus \frac{J}{J^2}\oplus J^2$ is a $4$-dimensional Hopf algebra whose 
Jacobson radical $\frac{J}{J^2}\oplus J^2$ has dimension 3; we saw that this is false. So it remains that ${\rm dim}_kJ^2=2$, in which case either 
${\rm gr}(A)=A_0\oplus A_1\oplus J^2$ (provided that $J^3=0$) or ${\rm gr}(A)=A_0\oplus A_1\oplus A_2\oplus J^3$ (provided that $J^3\not=0$, and so 
${\rm dim}_kJ^3=1$); in both cases we get that ${\rm gr}(A)$ is a $4$-dimensional Hopf algebra whose Jacobson radical is of dimension $3$, a 
contradiction we met for several times during this proof.  
\end{proof} 

We now have all the necessary ingredients to show that ${\rm gr}(A)$ is a Hopf algebra. 
 
\begin{proposition}\prlabel{graded A}
Let $A$ be a nonsemisimple quasi-Hopf algebra of dimension $4$. Then $A$, ${\rm gr}(A)$ are isomorphic as algebras and ${\rm gr}(A)$ is a Hopf algebra isomorphic 
to $H_4$.
\end{proposition}
\begin{proof}
As before, $J$ stands for the Jacobson radical of $A$. Then ${\rm dim}_kJ=2$ and $J^2=0$; we cannot have ${\rm dim}_kJ^2=1$, 
otherwise $A_1=\frac{J}{J^2}$ is a one dimensional vector space and at the same time a free module over $A_0$, a $2$-dimensional vector space. 

Thus ${\rm gr}(A)=\frac{A}{J}\oplus J\simeq A$ is a $4$-dimensional quasi-Hopf algebra, and we will next see that it is isomorphic to $A$ as an algebra. 
Towards this end, for all $a\in A$, consider $a=\mathbb{S}(a)+J(a)$ the decomposition of $a$ in $A=\mathbb{S}\oplus J$, where $\mathbb{S}$ is the semisimple subalgebra of $A$ given by \cite[Theorem 1.4.9]{abe}. Since $J^2=0$, for all $a, a'\in A$, 
$aa'=\mathbb{S}(a)\mathbb{S}(a') + \mathbb{S}(a)J(a') + J(a)\mathbb{S}(a')$, 
so $\mathbb{S}(aa')=\mathbb{S}(a)\mathbb{S}(a')$ and $J(aa')=\mathbb{S}(a)J(a')+J(a)\mathbb{S}(a')$, for all $a, a'\in A$. It follows that $\mathbb{S}$ 
induces an algebra morphism from $A$ to $\mathbb{S}$, which we denote by $\pi_{\mathbb{S}}$, such that $\pi_\mathbb{S}\iota=\Id_\mathbb{S}$ 
($\iota: \mathbb{S}\ra A$ is the inclusion, an algebra morphism, too). In addition, the canonical $k$-isomorphism $\mathbb{S}\ni s\mapsto p(s)\in A_0$ turns into a 
$k$-algebra isomorphism, its inverse being determined by $A_0\ni p(a)\mapsto \mathbb{S}(a)\in \mathbb{S}$. We deduce from here that 
$\zeta: A=\mathbb{S}\oplus J\ra A_0\oplus J={\rm gr}(A)$ given by $\zeta(a)=p(\mathbb{S}(a))+ J(a)=p(a) + J(a)$, for all $a\in A$, is a $k$-isomorphism. It is a $k$-algebra morphism as well, since the graded algebra structure of ${\rm gr}(A)$ is defined by $p_m(x)p_n(y)=p_{m+n}(xy)$, for all 
positive integers $m$ and $n$, $x\in J^m$ and $y\in J^n$, and therefore in our case $p(a)x=p_1(ax)=ax$ and 
$xp(a)=p_1(xa)=xa$, for all $a\in A$, $x\in J$; then we compute that, for all $a, a'\in A$,  
\begin{eqnarray*}
\zeta(a)\zeta(a')&=&p(aa')+p(a)J(a')+J(a)p(a')\\
&=&p(aa') + \mathbb{S}(a)J(a')+ J(a)\mathbb{S}(a')=p(aa') + J(aa')=\zeta(aa'),
\end{eqnarray*}      
as needed. Clearly, $\zeta(1_A)=p(1_A)$ is equal to the unit of ${\rm gr}(A)$, hence $A\simeq {\rm gr}(A)$ as $k$-algebras. 

Consequently, $A\simeq B\times A_0$ as an algebra, for a certain braided Hopf algebra $B\in {}_{A_0}^{A_0}{\cal YD}$. Since $A$ is not semisimple and $A_0$ is, 
\coref{2classifgen} implies that $A\simeq A_0(\theta)_{\alp, v}$, for an algebra morphism $\alp: A_0\ra k$ and $v\in A_0$ such that $\va(v)=1$, $\alp(v)=-1$. 
But, as an algebra, $A_0=k[C_2]=k[\le g\ri]$, so $\alp$ is determined by $\alp(1)=1_k$, $\alp(g)=-1$ (if $\alp(g)=1$, $\alp=\va$, fact that implies 
$1=\va(v)=\alp(v)=-1$, a contradiction); this forces $v=g$. It is clear now that, up to isomorphism, $A=H_4$ as an algebra. 

On the other hand, there exists a quasi-Hopf algebra morphism $p: A\ra A_0$ and an algebra morphism $\upsilon: A_0\simeq \mathbb{S}\hookrightarrow A$ 
given by $\upsilon(p(a))=\mathbb{S}(a)$, for all $a\in A$, such that $p\upsilon=\Id_{A_0}$. We proved that $A=H_4$ as an algebra, so 
$\mathbb{S}=k[\le g\ri]$ and $J$ is the subspace of $A$ generated by $x$ and $gx$, where $g, x$ are the algebra generators of $H_4$. By 
identifying $A_0$ to $k_\phi[\le g\ri]$ as a quasi-Hopf algebra (for a certain $3$-cocycle $\phi$ of $\le g\ri$), 
we deduce that there exists a surjective quasi-Hopf algebra morphism 
$p: A\ra k_\phi[\le g\ri]$ that acts as identity on the subalgebra $\mathbb{S}=k[\le g\ri]$ of $A$. In particular, if the quasi-coalgebra structure of $A$ is 
defined by $(\Delta, \va)$ then $\Delta(g)=g\ot g$ and $\va(g)=1$. 

Denote by $(\widetilde{\Delta}, \widetilde{\va})$ the quasi-coalgebra structure of ${\rm gr}(A)$ induced by the quasi-coalgebra structure $(\Delta, \va)$ 
of $A$. Since $x\in J$ and $J\subseteq A$ is a quasi-Hopf ideal, $\Delta(x)=a\ot x + b\ot gx + x\ot c + gx\ot d$, for some elements 
$a, b, c, d$ of $A$ such that $\va(a)=\va(c)=1$ and $\va(b)=\va(d)=0$. As $\widetilde{\Delta}(x)^2=0$ and 
\[
\widetilde{\Delta}(x)=(p\ot \Id + \Id\ot p)\Delta(x)=p(a)\ot x + p(b)\ot gx + x\ot p(c) + gx\ot p(d),
\]
a direct computation yields $p(a)=p_+ + \l p_-$, $p(b)=\mu p_-$, $p(c)=p_+ - \l p_-$ and $p(d)=-\mu p_-$, for some $\l, \mu \in k$ such that $\l^2-\mu^2=1$.    
   
It is well known that, for $k$ a field of characteristic different from $2$, the third cohomology group $H^3(\le g\ri, k^*)$ is isomorphic to $C_2$. 
The representative $3$-cocycles of $\le g\ri$ are given by the trivial one and $\phi=1-2p_-\ot p_-\ot p_-$. Since $\phi$ is as well the reassociator of 
${\rm gr}(A)$, in the former case we get that the "$p_-\ot p_-\ot -$" component of $\phi (\widetilde{\Delta}\ot \Id)(\widetilde{\Delta}(x))$ 
has $-$ equal to $x-2p_-x$ while the "$p_-\ot p_-\ot -$" component of $(\Id\ot \widetilde{\Delta})(\widetilde{\Delta}(x))\phi$ has $-$ equal to 
$\l^2x +\l\mu gx -\l\mu gx - \mu^2x - 2p_+x=x-2p_+x$. The quasi-coassociativity of $\widetilde{\Delta}$ says that 
$\phi (\widetilde{\Delta}\ot \Id)(\widetilde{\Delta}(x))=(\Id\ot \widetilde{\Delta})(\widetilde{\Delta}(x))\phi$, so $p_+x=p_-x$, a contradiction. 

It remains that $\phi$ is trivial, hence $\phi=(1\ot F)(\Id\ot \Delta)(F)(\Delta\ot \Id)(F^{-1})(F^{-1}\ot 1)$, for some twist 
$F$ of $k[C_2]$. It can be seen easily that $F=1-\kappa p_-\ot p_-$, for some $\kappa\in k$ different from $1_k$, with inverse $F^{-1}=
1-\frac{\kappa}{\kappa-1}p_-\ot p_-$. A straightforward computation leads us to $\phi=1\ot 1\ot 1$. Therefore,  
as ${\rm gr}(A)\simeq B\times k[C_2]$ is a $4$-dimensional nonsemisimple Hopf algebra 
(once more, ${\rm gr}(A)\simeq A$ as an algebra), from the classification of $4$-dimensional Hopf algebras we conclude that ${\rm gr}(A)$ and 
$H_4$ are isomorphic Hopf algebras. 
\end{proof}

As we have already mentioned, owing to Etingof and Gelaki, a nonsemisimple $4$-dimensional quasi-Hopf algebra is twist equivalent to $H_4$. In what follows we want 
to determine when two twist-deformations of $H_4$ are isomorphic as quasi-Hopf algebras. A situation when this fact occurs (in general for a quasi-Hopf algebra $H$, 
so not only for $H_4$) is when the twists $\mathbb{F}$, $\mathbb{G}$ of $H$ are related through a special invertible element $\sigma$. Namely, if there exists 
an invertible $\sigma\in H$ such that $\va(\sigma)=1$, $\sigma^{-1}=S(\sigma)$ and $\mathbb{G}=(\sigma^{-1}\ot \sigma^{-1})\mathbb{F}\Delta(\sigma)$, then 
$\varphi: H_{\mathbb{F}}\ra H_{\mathbb{G}}$ defined by $\varphi(h)=\sigma^{-1}h\sigma$, for all $h\in H$, is a quasi-Hopf algebra isomorphism. 
We leave the verification of this fact to the reader, as it can be seen as the quasi-Hopf analogue of \cite[Propositions 2.3.3 \& 2.3.5]{maj}.   

For $H=H_4$, an invertible element $\sigma\in H_4$ satisfying $\va(\sigma)=1$ has the form $\sigma=p_+ + \l p_- + \r p_+x + \varrho p_-x$, with 
$0\not=\l, \rho, \varrho\in k$; then $\sigma^{-1}=p_+ +\ov{\l}p_- + \ov{\rho} p_+x + \ov{\varrho} p_-x$, where $\ov{\l}=\l^{-1}$, 
$\ov{\rho}=-\rho \l^{-1}$ and $\ov{\varrho}=-\rho \l^{-1}$. If follows that $\sigma^{-1}=S(\sigma)$ if and only if $\sigma=p_+ + \l p_-$ and $\l^2=1$, 
thus $\sigma\in \{1, g\}$.   

The existence of $\sigma$ as in the above is only a sufficient condition for $H_\mathbb{F}$ and $H_\mathbb{G}$ to be isomorphic quasi-Hopf algebras. An example in 
this sense is given by $H_4$; for this, we need the lemma below. 

\begin{lemma}\lelabel{esslemmH4twist}
Let ${\cal F}$ be a twist on $H_4$ such that $H_4$ and $(H_4)_{\cal F}$ are isomorphic as quasi-Hopf algebras; $H_4$ is seen as a quasi-Hopf algebra 
in a canonical way. Then ${\cal F}=1-\tau gx\ot x$, for some $\tau\in k$, and the isomorphism $\varphi$ between $H_4$ and $(H_4)_{\cal F}$ is determined by a non-zero 
scalar $\omega$, in the sense that $\varphi(g)=g$ and $\varphi(x)=\omega x$.
\end{lemma}
\begin{proof}
Define $p_{\pm}:=\frac{1}{2}(1\pm g)$; they are orthogonal idempotent elements of $H_4$ such that $1=p_+ + p_-$ and $p_+ - p_-=g$. Then $\{p_\pm, p_\pm x\}$ is a basis 
for $H_4$ and, with respect to this basis, any element ${\cal F}\in H_4\ot H_4$ obeying $\va({\cal F}^1){\cal F}^2=1=\va({\cal F}^2){\cal F}^1$ is of the form 
\begin{eqnarray}
&&\hspace{-1cm}
{\cal F}=p_+\ot 1 + p_-\ot p_+ + ap_-\ot p_+ + bp_-\ot p_-x + cp_-\ot p_+x + \mu p_-\ot p_-
\nonumber\\
&&\hspace{1cm}
+\nu p_-x\ot p_-x + \tau p_-x\ot p_+x + u p_+x\ot p_- + v p_+x\ot p_-x + wp_+x\ot p_+x,\eqlabel{genformtwistH4}
\end{eqnarray}
for some scalars $a, b, c, \mu, \nu, \tau, u, v, w\in k$; in what follows, we will refer to them as being the scalars defining ${\cal F}$. 
It is easy to see that ${\cal F}$ is a twist (that is, ${\cal F}$ is, moreover, invertible in $H_4\ot H_4$) if and only if $a\not=0$, in which case 
the scalars defining ${\cal F}^{-1}$, denoted by $\ov{a},\cdots, \ov{w}$, are given by 
\[
\ov{a}=\frac{1}{a},~\ov{b}=-\frac{b}{a},~\ov{c}=-\frac{c}{a},~\ov{\mu}=-\frac{\mu}{a},~\ov{\nu}=-\frac{\nu}{a},~\ov{\tau}=\frac{c\mu}{a}-\tau,~
\ov{u}=-\frac{u}{a},~\ov{v}=\frac{ub}{a}- v,~\ov{w}=-\frac{w}{a}.
\]
Now, let $\varphi: H_4\ra (H_4)_{\cal F}$ be a quasi-Hopf algebra isomorphism. It follows from the definition of a quasi-Hopf isomorphism 
that $\a_{\cal F}=\beta_{\cal F}=1$ and $\Phi_{\cal F}=1\ot 1\ot 1$. Otherwise stated, ${\cal F}^1S({\cal F}^2)=1=S({\cal G}^1){\cal G}^2$ and 
$(1\ot {\cal F})(\Id\ot \Delta)({\cal F})=({\cal F}\ot 1)(\Delta\ot \Id)({\cal F})$, where ${\cal F}^{-1}:={\cal G}^1\ot {\cal G}^2$.  

We have $S(p_\pm)=p_\pm$, $S(p_-x)=-p_+x$ and $S(p_+x)=p_-x$, as $p_\pm x=xp_\mp$. Therefore, ${\cal F}^1S({\cal F}^2)=1$ holds if and only if 
$a=1$ and $u=c=0$. Likewise, $S({\cal G}^1){\cal G}^2=1$, which is equivalent to $S({\cal F}^1){\cal F}^2=1$, holds if and only if $a=1$ and $b=\mu=0$. Thus 
\[
{\cal F}=1\ot 1 + \nu p_-x\ot p_-x + \tau p_-x\ot p_+x + v p_+x\ot p_-x + wp_+x\ot p_+x.
\]
With this simplified formula of ${\cal F}$ in mind, since 
\begin{eqnarray*}
&&\Delta(p_-)=p_+\ot p_- + p_-\ot p_+,~\Delta(p_-x)=p_+\ot p_-x + p_+x\ot p_- - p_-\ot p_+x + p_-x\ot p_+,\\
&&\Delta(p_+)=p_+\ot p_+ + p_-\ot p_-,~\Delta(p_+x)=p_+\ot p_+x + p_+x\ot p_+ - p_-\ot p_-x + p_-x\ot p_-,
\end{eqnarray*}
we see that $(1\ot {\cal F})(\Id\ot \Delta)({\cal F})$ equals to 
\begin{eqnarray*}
&&\hspace{-1cm}
1\ot 1\ot 1 + (\nu p_-x\ot v p_+x)\ot (p_+\ot p_-x + p_+x\ot p_- - p_-\ot p_+x + p_-x\ot p_+)\\
&&\hspace{2mm}
+ (\tau p_-x + wp_+x)\ot (p_+\ot p_+x + p_+x\ot p_+ - p_-\ot p_-x + p_-x\ot p_-)\\
&&\hspace{2mm} 
+ 1\ot (\nu p_-x\ot p_-x + \tau p_-x\ot p_+x +v p_+x\ot p_-x +w p_+x\ot p_+x),
\end{eqnarray*}
and that $({\cal F}\ot 1)(\Delta\ot \Id)({\cal F})$ equals to  
\begin{eqnarray*}
&&\hspace{-1cm}
1\ot 1\ot 1 + (p_+\ot p_-x + p_+x\ot p_- - p_-\ot p_+x + p_-x\ot p_+)\ot (\nu p_-x\ot \tau p_+x)\\
&&\hspace{2mm}
+ (p_+\ot p_+x + p_+x\ot p_+ - p_-\ot p_-x + p_-x\ot p_-)\ot (vp_-x\ot wp_+x)\\
&&\hspace{2mm} 
+ (\nu p_-x\ot p_-x +\tau p_-x\ot p_+x + vp_+x\ot p_-x + wp_+x\ot p_+x)\ot 1.
\end{eqnarray*}
Hence, we have $(1\ot {\cal F})(\Id\ot \Delta)({\cal F})=({\cal F}\ot 1)(\Delta\ot \Id)({\cal F})$ if and only if 
\begin{eqnarray*}
&&\hspace{-5mm}
\nu (p_+\ot p_-x + p_+x\ot p_- - p_-\ot p_+x + p_-x\ot p_+) + \tau (p_+\ot p_+x + p_+x\ot p_+ - p_-\ot p_-x \\
&&\hspace{2mm}
+ p_-x\ot p_-)=\nu p_+\ot p_-x + \tau p_+\ot p_+x + vp_-\ot p_-x +wp_-\ot p_+x + \nu p_-x\ot \tau p_+x\ot 1;\\
&&\hspace{-5mm}
v(p_+\ot p_-x + p_+x\ot p_- - p_-\ot p_+x + p_-x\ot p_+) + w(p_+\ot p_+x + p_+x\ot p_+ - p_-\ot p_-x \\
&&\hspace{2mm}
+ p_-x\ot p_-)=\nu p_-\ot p_-x + \tau p_-\ot p_+x + vp_+\ot p_-x + wp_+\ot p_+x + vp_-x\ot 1 + wp_+x\ot 1;\\
&&\hspace{-5mm}
1\ot (\nu p_-x\ot p_-x + \tau p_-x\ot p_+x + vp_+x\ot p_-x + wp_+x\ot p_+x)=p_+\ot \big(p_-x\ot 
(\nu p_-x + \tau p_+x)\\
&&\hspace{2mm}
+ p_+x\ot (vp_-x + wp_+x)\big) - p_-\ot\big(p_+x\ot (\nu p_-x +\tau p_+x) + p_-x\ot (vp_-x + wp_+x)\big).
\end{eqnarray*}
A simple inspection ensures us that the three equalities above take place if and only if $w=v=-\tau=-\nu$, in which case
\[
{\cal F}=1+\tau p_-x\ot (p_-x + p_+x)-\tau p_+x\ot (p_-x + p_+x)=1-\tau gx\ot x, 
\]
as stated. Now, for ${\cal F}=1-\tau gx\ot x$, it can be easily checked that $\Delta_{\cal F}=\Delta$; thus $(H_4)_{\cal F}$ is a Hopf algebra equal to $H_4$. 
Consequently, given a quasi-Hopf algebra isomorphism between $H_4$ and $(H_4)_{\cal F}$ is equivalent to giving a Hopf algebra automorphism of $H_4$. Perhaps 
it is folklore that any Hopf algebra automorphism $\varphi$ of $H_4$ is defined by a non-zero scalar as in the statement of the lemma. If not, as an algebra 
automorphism of $H_4$, $\v$ is determined by $\varphi(g)=g + \l x + \r gx$ and $\varphi(x)=\omega x + \delta gx$, for some $\l, \r, \omega, \delta\in k$ with 
$(\omega, \delta)\not=0$, since $g^2=1$, $x^2=0$, $\va(g)=1$ and $\va(x)=0$. From $(\varphi\ot \varphi)(g)=\varphi(g)\ot \varphi(g)$ we get $\l=\r=0$, and 
from $\Delta(\varphi(x))=(\varphi\ot \varphi)(g\ot x+ x\ot 1)$ we deduce that $\delta=0$. Summing up, $\varphi(g)=g$ and $\varphi(x)=\omega x$ with $0\not=\omega\in k$.
\end{proof}
  
\begin{proposition}
Let $\mathbb{F}$, $\mathbb{G}$ be twists on $H_4$ determined by $a, \cdots, w\in k$, and respectively by $a',\cdots,w'\in k$. 
Then $(H_4)_\mathbb{F}$ and $(H_4)_\mathbb{G}$ are isomorphic as quasi-Hopf algebras 
if and only if $a=a'$ and there exist $0\not=\omega, \kappa\in k$ such that $(b', c', \mu', u')=\omega (b, c, \mu, u)$, $\nu'=\omega^2\nu +\kappa a$, $\tau'=\omega^2\tau + \kappa$,  
$v'=\omega^2v-\kappa$ and $w'=\omega^2w -\kappa$.  
\end{proposition}
\begin{proof}
Note that $\varphi: (H_4)_{\mathbb{F}}\ra (H_4)_{\mathbb{G}}$ is a quasi-Hopf algebra isomorphism if and only if so is 
$\varphi: H_4\ra (H_4)_{\mathbb{F}^{-1}_\varphi\mathbb{G}}$, provided that $\mathbb{F}_\varphi:=(\varphi\ot \varphi)(\mathbb{F})$ with inverse $\mathbb{F}^{-1}_\varphi$. 
Denote ${\cal F}=\mathbb{F}^{-1}_\varphi\mathbb{G}$; by \leref{esslemmH4twist}, there exist $0\not=\omega, \kappa\in k$ such that 
$\mathbb{G}=\mathbb{F}_\varphi(1\ot 1 -\kappa gx\ot x)$, where $\varphi$ is defined by $\varphi(g)=g$, $\varphi(x)=\omega x$. Equivalently, 
$\mathbb{G}=\mathbb{F}_\varphi -\kappa \mathbb{F}_\varphi (gx\ot x)$, hence $\mathbb{G}$ equals to
\begin{eqnarray*}
&&\hspace{-7mm}
p_+\ot 1 + p_-\ot p_+ + ap_-\ot p_-+ \omega bp_-\ot p_-x + \omega c p_-\ot p_+x + \omega\mu p_-x\ot p_- + (\omega^2\nu + \kappa a)p_-x\ot p_-x\\
&&+(\omega^2\tau + \kappa)p_-x\ot p_+x + \omega up_+x\ot p_- + (\omega^2v-\kappa)p_+x\ot p_- + (\omega^2w-\kappa)p_+x\ot p_+x,
\end{eqnarray*}
and we are done.
\end{proof}
\begin{remarks}
1). As a continuation to the comments made before \leref{esslemmH4twist}, note that the case when $\mathbb{F}, \mathbb{G}$ are related through a special element 
$\sigma\in H_4$ (which we saw that is either $1$ or $g$) occurs if and only if $\kappa=0$ and $\omega^2=1$. 

2). For $a\not=a'$, the quasi-Hopf algebras $(H_4)_{\mathbb{F}}$ and $(H_4)_{\mathbb{G}}$ are not isomorphic. We conclude that there are infinitely many non-isomorphic 
nonsemismple quasi-Hopf algebras of dimension $4$, and all of them are twist equivalent to $H_4$.  
\end{remarks}
\subsection{The semisimple case}
We end this section by pointing out the classes of non twist equivalent semisimple quasi-Hopf algebras $H$ of dimension $4$. Since any irreducible representation of $H$ is one 
dimensional, it follows that, as an algebra, $H$ is the group algebra of $k$ and a group with $4$ elements. Consequently, $H$ is a commutative algebra, and therefore a usual 
Hopf algebra. Otherwise stated, $H$ is a the group Hopf algebra of $k$ and a group $G$ with $4$ elements, seen as a quasi-Hopf algebra via a normalized $3$-cocycle on $G$.  

The Hopf algebra of functions on a finite abelian group $G$, $k^G$, is isomorphic as a Hopf algebra to group Hopf algebra $k[G]$, so representative $3$-cocycles for 
$k[G]$ (in the quasi-Hopf sense) correspond to representative $3$-cocycles $\omega$ on the group $G$. Recall that a normalized $3$-cocycle on a group $G$ is a map 
$\omega: G\times G\times G\ra k^*$ obeying $\omega(x, y, z)=1$ whenever $x$, $y$ or $z$ is equal to $e$ (the neutral element of $G$) and 
\begin{equation}\eqlabel{3cocycledef}
\omega(y, z, t)\omega(x, yz, t)\omega(x, y, z)=\omega(x, y, zt)\omega(xy, z, t),~\forall~x, y,z, t\in G.
\end{equation} 
If $\{P_g\}_{g\in G}$ is the basis of $k^G$ dual to the basis 
$\{g\}_{g\in G}$ of $k[G]$ then $k^G$, endowed with the 
(co)multipli-cation and (co)unit coming from the (co)algebra structure of $k^G$, can also be seen as a quasi-bialgebra with reassociator $\Phi_\omega$ given by 
\begin{equation}\eqlabel{reassC6}
            \Phi_\omega=\sum_{g_1,g_2,g_3\in G} \omega(g_1,g_2,g_3)~ P_{g_1}\ot P_{g_2}\ot P_{g_3}.
  \end{equation}     
We denote this quasi-bialgebra structure on $k^G$ by $k^G_\omega$. $k^G_\omega$ is, moreover, a quasi-Hopf algebra 
with antipode determined by the triple $(S, \a, \b)$ consisting of $S(P_g)=P_{g^{-1}}$, for all $g\in G$, $\a$ equals the counit $\va$ of $k[G]$ and $\b$ given 
by $\b(g)=\omega(g, g^{-1}, g)^{-1}$, for all $g\in G$. 

The quasi-Hopf algebra structure of $k^G_\omega$ produces a monoidal structure on the category ${\rm Rep}(k^G_\omega)$ of representations of $k^G_\omega$. 
Owing to \cite[Remark 2.6.2]{egno}, the equivalent classes corresponding to these monoidal categories ${\rm Rep}(k^G_\omega)$, 
$\omega\in Z^3(G, k^*)$, are in a one to one correspondence to the orbits of the group action  
\begin{equation}\eqlabel{outact}
{\rm Out}(G)\times H^3(G, k^*)\ni (\widehat{f}, \overline{\omega})\mapsto \overline{f^*\omega}\in H^3(G, k^*).
\end{equation}
Here ${\rm Out}(G)$ is the group of outer automorphisms of $G$, $H^3(G, k^*)$ stands for the third cohomology group of $G$ 
with coefficients in $k^*$ and $f^*\omega\in Z^3(G, k^*)$ is defined by 
$f^*\omega(g_1,g_2,g_3)=\omega(f(g_1),f(g_2),f(g_3))$, for all $g_1, g_2, g_3\in G$. Therefore, for a finite group $G$, the number of 
quasi-Hopf algebras of the form $k^G_\omega$, $\omega\in Z^3(G, k^*)$, which are not pairwise twist equivalent equals the number of 
orbits ${\cal O}_{\ov{\omega}}$ of the group action defined by \equref{outact}. 	

For $G=C_n$, a cyclic group with $n$ elements, it is well known that $3$-cocycle representatives for $C_n$ are given by $\{\phi_{\xi^a}\mid 0\leq a\leq n-1\}$, where 
$\xi$ is a primitive root of unity of degree $n$ in $k$ and     
		\begin{eqnarray}\eqlabel{cn3cc}
        && \phi_{\xi^a}(\sigma^i,\sigma^j,\sigma^l)= \xi^{ai\lfloor\frac{j+l}{n} \rfloor },~\forall~0\leq i, j, l\leq n-1. 
    \end{eqnarray} 
Then, by \cite[Lemma 3.5]{bm}, the action in \equref{outact} can be restated as follows ($0\leq a\leq n-1$):
    \begin{equation}\eqlabel{actionw}
        (f, \overline{\phi_{\xi^a}})\mapsto \ov{\phi_{\xi{s^2a}}},
    \end{equation}
	provided that $f\in {\rm Aut}(C_n)\simeq U(\mathbb{Z}_n)$ is defined by $1\leq s\leq n-1$ such that $(s, n)=1$ 
	(that is, $f$ is determined by $f(\sigma)=\sigma^s$). As $s^2\equiv 1~(mod~4)$, for all $s\in \{1, 3\}$, it follows that we have $4$ quasi-Hopf algebra 
	structures on $k[C_4]$ which are not pairwise twist equivalent.
	
	Likewise, $3$-cocycle representatives for $G=C_2\times C_2$ are described explicitly in \cite[Theorem 3.5]{bct}. For short, if $C_2\times C_2=\{e, \sigma, \tau, \rho\}$ 
	with $\tau\sigma=\sigma\tau=\rho$ then $3$-cocycle representatives for $G=C_2\times C_2$ are in a one correspondence with the subsets $X\subseteq \{\sigma, \tau, \rho\}$, 
	provided that $k$ is an algebraically closed field of characteristic zero. Furthermore, if $\phi_X$ is the $3$-cocycle corresponding to $X$, then $\phi_X$ is invariant under any 
	permutation of $S_3$, the symmetric group on $3$ letters; see the comments made after the proof of \cite[Theorem 3.5]{bct}. Since ${\rm Aut}(C_2\times C_2)\simeq S_3$, we 
	conclude that for $G=C_2\times C_2$ the group action in \equref{outact} is trivial, and therefore we have $8$ quasi-Hopf algebra 
	structures on $k[C_2\times C_2]$ which are not pairwise twist equivalent. 
	
	$k[C_4]$ and $k[C_2\times C_2]$ are isomorphic as algebras (both are isomorphic to $k^4$) 
	but not as Hopf algebras. Indeed, if there exists a Hopf algebra isomorphism 
	$\varphi:k[C_2\times C_2]\rightarrow k[C_4]$ then $\varphi$ behaves well with respect to the antipodes $S_{k[C_4]}$ and 
	$S_{k[C_2\times C_2]}$ of $k[C_4]$ and $k[C_2\times C_2]$, respectively. Since $S_{k[C_2\times C_2]}$ is the identity morphism, 
	we get that $S_{k[C_4]}\circ \phi=\phi$ or, equivalently, that $S_{k[C_4]}=\Id_{k[C_4]}$; this is false. 
	We conclude that there are precisely $12$ semisimple quasi-Hopf algebras in dimension $4$ which are not pairwise twist equivalent. 
\section{The algebra of functions on a finite  group}\selabel{funalgexp}
\setcounter{equation}{0}
Let $G$ be a finite group, $\omega$ a normalized $3$-cocycle on $G$ and $k^G_\omega$ the Hopf algebra of functions on $G$, seen as a quasi-Hopf algebra 
with reassociator $\Phi_\omega$ as in \equref{reassC6}. In this section, with the help of \thref{2dimbraidedHA} we characterize $2$-dimensional braided 
Hopf algebras $B$ in ${}_{k_{\omega}^G}^{k_{\omega}^G}{\cal YD}$. We will see that the structure of $B$ is determined by a pair $(\rho, \mfg)$ 
consisting of a map $\rho: G\ra k^*$ and an element $\mfg\in G$ satisfying certain conditions. To this end, recall that a  $2$-cocycle on $G$ is a map 
$\sigma: G\times  G\rightarrow k^*$ such that $\sigma(e, x)=\sigma(x, e)=1$ and 
\[
\sigma(x, y)\sigma(xy, z) = \sigma(x, yz)\sigma(y, z),~\forall~x, y, z\in G. 
\]
$\sigma$ is a coboundary $2$-cocycle if there exists a map $\rho: G\rightarrow k^*$ such that $\rho(e)=1$ and 
$\sigma=\partial\rho$, where $\partial\rho: G\times G\ra k^*$ is defined by $\partial\rho(x,y):=\rho(xy)^{-1}\rho(x)\rho(y)$, for all $x, y\in G$. 

The next result can be seen as the non-commutative version of the result stated at the beginning of subsection 2.2 in \cite{hlyy}.    

  \begin{lemma}
      Let $\omega$ be a normalized $3$-cocycle on $G$ and $\mfg\in G$. Define $\flat_{\mfg}\omega: G\times G\ra k^*$ by  
      \[
        \flat_\mfg\omega(x,y)=\frac{\omega(\mfg,x,y)\omega(x,y,\mfg)}{\omega(x,\mfg,y)},~\forall~x, y\in G.
      \]
		Then $\flat_\mfg$ is a normalized $2$-cocycle on $G$ if and only if 
		\begin{equation}\eqlabel{flat g omega}
        \omega(\mfg x,y,z)\omega(x,\mfg y,z)\omega(x,y,\mfg z)=\omega(x\mfg,y,z)\omega(x,y\mfg,z)\omega(x,y,z\mfg),~\forall~x, y, z\in G.
      \end{equation}
Consequently, $\flat_\mfg\omega$ is a normalized $2$-cocycle on $G$, provided that $\mfg$ is in the center of $G$.  
  \end{lemma}
  \begin{proof}
 $\flat_\mfg\omega$ is normalized since $\omega$ is so. By applying the $3$-cocycle condition \equref{3cocycledef} for several times we compute that      
      \begin{eqnarray*}
\flat_\mfg\omega(x, y)\flat_\mfg\omega(xy, z)&=&\frac{\omega(\mfg, x,y)\omega(x,y,\mfg )\omega(\mfg,xy,z)\omega(xy, z, \mfg)}{ \omega(x,\mfg, y)\omega(xy,\mfg,z )}\\
&=&\frac{\omega(\mfg x, y, z)\omega(\mfg, x, yz)\omega(x, y, \mfg)\omega(xy, z, \mfg)}{\omega(x, \mfg, y)\omega(xy, \mfg, z)\omega(x, y, z)}\\
&=&\frac{\omega(\mfg x, y, z)\omega(\mfg, x, yz)\omega(\mfg, x, yz)\omega(x, y, \mfg z)\omega(xy, z, \mfg)}{\omega(x, \mfg, y)\omega(y, \mfg, z)
\omega(x, y\mfg, z)\omega(x, y, z)}\\          
&=&\frac{\omega(\mfg x, y, z)\omega(\mfg, x, yz)\omega(x, y, \mfg z)\omega(y, z, \mfg )\omega(x, yz, \mfg)}{\omega(x, \mfg, y)\omega(y, \mfg, z)
\omega(x, y\mfg, z)\omega(x, y, z\mfg)}\\
&=&\flat_\mfg\omega(x, yz)\flat_\mfg(y, z) 
           \frac{ \omega( \mfg x, y, z)\omega(x, \mfg y, z)\omega(x, y, \mfg z)}
           {\omega(x\mfg, y, z)\omega(x, y\mfg, z)\omega(x, y, z\mfg)},
      \end{eqnarray*}
   for all $x, y, z\in G$, and our assertions follow.
  \end{proof}

  \begin{lemma}\lelabel{1d fun yd}
      Let $G$ be a finite group and $\omega$ a normalized $3$-cocycle on $G$. Giving a pair $(\sigma, v)$ for $k^G_\omega$ satisfying the conditions in 
			\deref{Balpv} is equivalent to giving a pair $(\mfg, \rho)$ consisting of an element $\mfg$ in the center of $G$ and a map $\rho: G\ra k^*$ obeying 
			$\rho(e)=1$, $\rho(\mfg)=-1$ and $\flat_\mfg\omega=\partial\rho$.
  \end{lemma}
  \begin{proof}
  Keeping in mind the quasi-Hopf algebra structure of $k^G$, an algebra map $\sigma: k^G\ra k$ is completely determined by an element $\mfg\in k[G]$ 
	satisfying $\Delta(\mfg)=\mfg\ot \mfg$ and $\va(\mfg)=1$. Thus, $\mfg\in G$ and $\sigma(\varphi)=\varphi(\mfg)$, for all $\varphi\in k^G$. 
			
			An element $v\in k^G$ is completely determined by a map $\rho: G\ra k$, in the sense that $v=\sum\limits_{x\in G}\rho(x)P_x$, where, as before, 
			$\{P_g\}_{g\in G}$ is the basis of $k^G$ dual to the basis $\{g\}_{g\in G}$ of $k[G]$. We have $\va(v)=1$ and $\sigma(v)=-1$ if and only if 
			$\rho(e)=1$ and $\rho(\mfg)=-1$, respectively.
			
			We look now at the condition $\sigma(\varphi_2)\varphi_1v=\sigma(\varphi_1)v\varphi_2$, which must be valid for all $\varphi\in k^G_\omega$ or, equivalently, 
			for any $P_h$, $h\in G$. Since $\Delta(P_h)=\sum_{xy=h}P_x\ot P_y$, we get that $\sum_{xy=h}\sigma(P_y)P_xv=\sum_{xy=h}\sigma(P_x)vP_y$, for all $h\in H$. 
			Thus $P_{h\mfg^{-1}}v=vP_{\mfg^{-1}h}$, for all $h\in H$, and this fact holds if and only if $\mfg h=h\mfg$, for all $h\in G$; otherwise stated, if and 
			only if $\mfg$ is an element in the center of $G$. 
			
Finally, we claim that $\Delta(v)=\alp(y^1x^3X^2)x^1X^1vy^2\ot x^2vX^3y^3$ is equivalent to $\flat_\mfg\omega=\partial\rho$. Indeed, since 
$\Phi_\omega^{-1}=\sum_{g_1, g_2, g_3\in G}\omega(g_1, g_2, g_3)^{-1}P_{g_1}\ot P_{g_2}\ot P_{g_3}$, we compute that 
\[
\sigma(X^2)X^1v\ot vX^3=\sum\limits_{g_1, g_3\in G}\omega(g_1, \mfg, g_3)P_{g_1}v\ot vP_{g_3}
=\sum\limits_{x, y\in G}\omega(x, \mfg, y)\rho(x)\rho(y)P_x\ot P_y
\]
and, likewise, $\alp(y^1X^2)X^1vy^2\ot vX^3y^3=\sum\limits_{x, y\in G}\frac{\omega(x, \mfg, y)}{\omega(\mfg, x, y)}\rho(x)\rho(y)P_x\ot P_y$. 
We then have  
\[
\alp(y^1x^3X^2)x^1X^1vy^2\ot x^2vX^3y^3=\sum\limits_{x, y\in G}\frac{\omega(x, \mfg, y)}{\omega(\mfg, x, y)\omega(x, y, \mfg)}
\rho(x)\rho(y)P_x\ot P_y=\sum\limits_{x, y\in G}\frac{\rho(x)\rho(y)}{\flat_\mfg\omega(x, y)}P_x\ot P_y. 
\]
On the other hand, $\Delta(v)=\sum\limits_{g, y\in G}\rho(g)P_{gy^{-1}}\ot P_y$, and so 
$\Delta(v)=\alp(y^1x^3X^2)x^1X^1vy^2\ot x^2vX^3y^3$ if and only if $\sum_{g\in G}\rho(g)P_{gy^{-1}}=
\sum_{x\in G}\frac{\rho(x)\rho(y)}{\flat_\mfg\omega(x, y)}P_x$, for all $y\in G$. Clearly, the latest assertion is equivalent to $\rho(xy)\flat_\mfg\omega(x, y)=
\rho(x)\rho(y)$, for all $x, y\in G$. If there is $x\in G$ such that $\rho(x)=0$ then $\rho(xy)=0$, for all $y\in G$, meaning that $\rho(g)=0$, for all $g\in G$, 
which is false. So we can consider $\rho: G\ra k^*$, and we conclude that $\flat_\mfg\omega=\partial\rho$.
  \end{proof}

The result below is an immediate consequence of \leref{1d fun yd} and \thref{2dimbraidedHA}.

  \begin{proposition}\prlabel{fun chara}
Non-trivial $2$-dimensional braided Hopf algebras in ${}_{k^G_\omega}^{k^G_\omega}{\cal YD}$ are in a one to one correspondence with couples $(\mf{g}, \rho)$ 
consisting of an element $\mfg$ in the center of $G$ and a map $\rho: G\ra k^*$ obeying $\rho(e)=1$, $\rho(\mfg)=-1$ and $\flat_\mfg\omega=\partial\rho$. 
  \end{proposition}
	
We apply \prref{fun chara} to various finite groups $G$.	
 
\subsection{Finite abelian groups}
Let $G=C_{1}\times \ldots\times C_{n}$ be a finite abelian group; $C_j$ stands for the cyclic group of order $m_j$, generated by $g_j$, $1\leq j\leq n$. 
In general, for $m$ a non-zero positive integer, denote by $\xi_{m}$ a primitive $m$th root of unity in $k$. 
Then $1_{j_1,\ldots,j_n}$ ($0\leq j_l\leq m_l-1$, for all $1\leq l\leq n$) given by 
\begin{equation}\eqlabel{canisomgralg}
1_{j_1,\ldots,j_n}:=\frac{1}{m_1\ldots m_n}\prod_{l=1}^{n}\left(\sum_{t=0}^{m_l-1}\xi_{m_l}^{-tj_l}g_l^t\right) 
\end{equation}
define a family of orthogonal idempotents of $k[G]$. 
It is well known that $k^G$ and $k[G]$ identify (as Hopf algebras) through the correspondence 
$P_{g_1^{j_1}\ldots g_n^{j_n}}\mapsto 1_{j_1,\ldots , j_n}$. The inverse of this correspondence maps $g_1^{j_1}\cdots g_n^{j_n}$ to 
$\prod_{l=1}^n\left(\sum_{t=0}^{m_l-1}\xi_{m_l}^{tj_l}\widetilde{P}_{g_l^t}\right)$, where $\widetilde{P}_{g_l^t}\in k^G$ is determined by 
$\widetilde{P}_{g_l^t}(g_1^{a_1}\cdots g_n^{a_n})=\delta_{t, a_l}$, for all $g_1^{a_1}\cdots g_n^{a_n}\in G$.

A set of representative $3$-cocycles on $G$ was given in \cite{hlyy}. Namely, denote by $\cal A$ the set of sequences $\underline a$ with integer elements, 
having the form 
\begin{equation}\eqlabel{c datum}
(c_1,\ldots,c_l,\ldots,  c_n,c_{12},\ldots,c_{ij},\ldots, c_{n-1,n},c_{123},\ldots,c_{rst},\ldots, c_{n-2,n-1,n})
\end{equation}
with $0\leq c_l\leq m_l-1$, $0\leq c_{ij}\leq (m_i,m_j)$ and $0\leq c_{rst}\leq (m_r,m_s,m_t)$, for all $1\leq l\leq n$, $1\leq i < j \leq n$ and 
$1\leq r<s<t\leq n$; the $c_{ij}$'s and the $c_{rst}$'s are ordered by the lexicographical order and $(u, v)$ stands for the greatest common divisor of the 
integer $u, v$, and similar for $(u, v, w)$. 
 
For $\underline a$ in  $\cal A$, define $\omega_{\underline a}:G\times G\times G\rightarrow k^*$ by 
\begin{equation}\eqlabel{3cocyfabgr}
     \omega_{\underline a}(g_1^{i_1}\ldots  g_n^{i_n},g_1^{j_1}\ldots g_n^{j_n},g_1^{k_1}\ldots g_n^{k_n})
		:=\prod_{l=1}^n \xi_{m_l}^{c_li_l\lfloor   \frac{j_l+k_l}{m_l}\rfloor}\hspace{-4mm}  
        \prod_{1\leq s<t\leq n} \xi_{m_t}^{c_{st}i_t\lfloor   \frac{j_s+k_s}{m_s}\rfloor}\hspace{-4mm}
        \prod_{1\leq r<s<t\leq n}\xi_{(m_r,m_s,m_t)}^{-c_{rst}k_rj_si_t}.
\end{equation}
Owing to \cite[Proposition 3.8]{hlyy2}, $\{\omega_{\underline a}\}_{\underline a\in\cal A}$ is a set of representative normalized $3$-cocycles for $G$. 
Furthermore, $\omega_{\underline a}$ is an abelian $3$-cocycle 
if and only if $c_{rst}=0$ for all $1\leq r<s<t\leq n$ (for the definition of an abelian $3$-cocycle we refer to \cite{hlyy2}). 

In computations, sometimes we need a common root of unity, one from which we can derive all the others roots of unity. This is why, we set 
$M:=LCM\{(m_r, m_s, m_t)|1\leq r<s<t\leq n\}$ and fix $\xi_M$ a primitive $M$th root of unity in $k$; then, without loss of generality, we assume 
that $\xi_{(m_r, m_s, m_t)}=\xi_M^{M_{rst}}$, where, for all $1\leq r<s<t\leq n$, $M_{rst}:=\frac{M}{(m_r, m_s, m_t)}$. In addition, we define 
$M^c_{rst}=M_{rst}c_{rst}$, for all $1\leq r<s<t\leq n$. 

\begin{proposition}\prlabel{flat coboundary}
    For $\mf g=g_1^{f_1}\ldots  g_n^{f_n}
    \in C_{1}\times \ldots\times C_{n}$ and $\underline a=(c_1,\ldots,c_{n-2,n-1,n})$ as in \equref{c datum}, $\flat_{\mf g}{\omega_{\underline a}}$ is a 
		coboundary $2$-cocycle if and only if $n=2$ or $n\geq 3$ and for any $1\leq r<s\leq n$ the following divisibility hold:
		\begin{equation}\eqlabel{divis}
		M|\sum\limits_{u=1}^{r-1}M^c_{urs}f_u - \sum\limits_{v=r+1}^{s-1}M^c_{rvs}f_v +\sum\limits_{w=s+1}^nM^c_{rsw}f_w; 	
		\end{equation}
		we make the convention that a sum for which its summands do not exist is equal to zero.
		
    If this is the case then $\flat_{\mf g}{\omega_{\underline a}} =\flat_{\mf g}{\omega_{\bar{a}}}$, where $\bar a= (\bar c_1,\ldots,\bar c_{n-2,n-1,n})$ defined by 
		$\bar c_i = c_i$, $\bar c_{ij}=c_{ij}$ and $\bar c_{rst}=0$ induces the abelian $3$-cocycle $\omega_{\bar{a}}$ on $G$ associated to $\underline{a}$. 
\end{proposition}
\begin{proof}
     We start by noting that, for $\underline a\in {\cal A}$ and $\mf g=g_1^{f_1}\ldots  g_n^{f_n}\in G$, 
		$\flat_{\mf g}\omega_{\underline a}( g_1^{ j_1}\ldots  g_n^{ j_n},g_1^{ k_1}\ldots  g_n^{ k_n}  )$ equals to
    \begin{equation*}
        \prod_{l=1}^n \xi_{m_l}^{c_lf_l\lfloor   \frac{j_l+k_l}{m_l}\rfloor}  
        \prod_{1\leq s<t\leq n} \xi_{m_t}^{c_{st}f_t\lfloor   \frac{j_s+k_s}{m_s}\rfloor}
        \prod_{1\leq r<s<t\leq n}\xi_{(m_r,m_s,m_t)}^{-c_{rst}(k_rj_sf_t + f_rk_sj_t - k_r f_s j_t)}.
    \end{equation*}
    With ${\bar a}$ as in the statement of the proposition, we show that $\flat_{\mf g}\omega_{\bar a}$ is a coboundary $2$-cocycle. Indeed, 
		by specializing the above formula for $\ov{a}$, we have that 
    \begin{eqnarray*}
		&&\hspace{-5mm}\flat_{\mf g}\omega_{\bar a}( g_1^{ j_1}\ldots  g_n^{ j_n},g_1^{ k_1}\ldots  g_n^{ k_n}  )\\
        &&=\prod_{l=1}^n \xi_{m_l}^{c_lf_l\lfloor   \frac{j_l+k_l}{m_l}\rfloor}  
        \prod_{1\leq s<t\leq n} \xi_{m_t}^{c_{st}f_t\lfloor   \frac{j_s+k_s}{m_s}\rfloor}
        =\prod_{l=1}^n 
        \xi_{{m_l}}^{c_lf_l \frac{j_l+k_l - (j_l+k_l)_{m_l}'}{m_l}}
        \prod_{1\leq s<t\leq n} \xi_{m_t}^{c_{st}f_t   \frac{j_s+k_s- (j_s+k_s)_{m_s}'}{m_s}}\\
        &&= \prod_{l=1}^n \frac{\xi_{m_l^2}^{c_lf_lj_l}\xi_{m_l^2}^{c_lf_lk_l}}{\xi_{m_l^2}^{c_lf_l(j_l+ k_l)_{m_l}'}}
        \prod_{1\leq s<t\leq n}
        \frac{\xi_{m_t m_s}^{c_{st}f_t  j_s } \xi_{m_t m_s}^{c_{st}f_t  k_s }}{\xi_{m_tm_s}^{c_{st}f_t (j_s + k_s)_{m_s}'}}
				=\frac{f_{\mf g}(g_1^{ j_1}\ldots  g_n^{ j_n})~f_{\mf g}(g_1^{ k_1}\ldots  g_n^{ k_n} )  }
				{f_{\mf g}(g_1^{ (j_1+k_1)_{m_1}'}\ldots  g_n^{  (j_1+k_1)_{m_n}'} ) },
    \end{eqnarray*}
    provided that $f_{\mf g}:G\rightarrow k^*$ is defined by $f_{\mf g}( g_1^{ j_1}\ldots  g_n^{ j_n}):=  \prod_{l=1}^n   \xi_{m_l^2}^{c_lf_lj_l} 
        \prod_{1\leq s<t\leq n}
         {\xi_{m_t m_s}^{c_{st}f_t  j_s }  } $. As the reader expects, for $1\leq l\leq n$, $(j_l+ k_l)'_{m_l}$ is the remainder of the division of 
				$j_l+k_l$ by $m_l$ and $\xi_{m_l^2}\in k$ is such that $\xi_{m_l^2}^{m_l}=\xi_{m_l}$; similarly, $\xi_{m_tm_s}\in k$ such that 
				$\xi_{m_tm_s}^{m_s}=\xi_{m_t}$. 

    Next, define $\flat_{\mfg}\omega_{\ov{\ov{a}}}: G\times G\ra k^*$ by 
    	\[\flat_{\mfg}\omega_{\ov{\ov{a}}}( g_1^{ j_1}\ldots  g_n^{ j_n},g_1^{ k_1}\ldots  g_n^{ k_n}  )=
    		\prod_{1\leq r<s<t\leq n}\xi_{(m_r,m_s,m_t)}^{-c_{rst}(k_rj_sf_t+ f_rk_sj_t    - k_r f_s j_t  )}. 
    		\]
    		We claim that $\flat_{\mfg}\omega_{\ov{\ov{a}}}$ is a coboundary 2-cocycle if and only if it is trivial, this means that   
		\begin{equation}\eqlabel{condodcob2cocyclenonab}
		\prod_{1\leq r<s<t\leq n}\xi_{(m_r,m_s,m_t)}^{-c_{rst}(k_rj_sf_t + f_rk_sj_t - k_r f_s j_t)}=1,~\forall~0\leq j_l, k_l\leq m_l-1;~1\leq l\leq n.
		\end{equation}

				Indeed, assume that there exists a map $\rho :G\rightarrow k^*$ such that 
    \begin{equation*}
        \rho( g_1^{ j_1}\ldots  g_n^{ j_n}  )
        \rho(  g_1^{ k_1}\ldots  g_n^{ k_n}  )
        =\flat_\mfg\omega_{\overline{\overline{a}}}( g_1^{ j_1}\ldots  g_n^{ j_n},
        g_1^{ k_1}\ldots  g_n^{ k_n}  )\rho( g_1^{ (j_1+ k_1)_{m_1}'}\ldots  g_n^{ (j_n+k_n)_{m_n}'} ),
    \end{equation*}
    for all $0\leq j_l, k_l\leq m_l-1$, $1\leq l\leq n$.  For $1\leq i\leq n-1$, we then have that 
		\[
		\flat_\mfg\omega_{\overline{\overline{a}}}(g_i^{a_i}, 
        g_{i+1}^{a_{i+1}}\ldots  g_n^{a_n} )
				=\xi_M^{-\sum\limits_{1\leq r<s<t\leq n}M_{rst}k_rj_sf_t -\sum\limits_{1\leq r<s<t\leq n}M_{rst}j_t(f_rk_s-k_rf_s)}=1, 
		\]		
		since both sums that appear in the exponent of $\xi_M$ are equal to zero (for the first sum the only possible non-zero summand is obtained when $s=i$ 
		but then $r<i$ implies $k_r=0$, while for the second sum the only possible non-zero summand is obtained when $t=i$ in which case $r<s<i$ imply $k_r=k_s=0$).  		
		Thus, for any $1\leq i\leq n-1$, we have that   
      \[
      \rho(g_i^{a_i})\rho(g_{i+1}^{a_{i+1}}\ldots g_{n}^{a_{n}})
      =\flat_{\mfg}\omega_{\overline{\overline{a}}}(g_i^{a_i}, g_{i+1}^{a_{i+1}}\ldots g_{n}^{a_{n}})
      	\rho(g_i^{a_i}\ldots g_{n }^{a_{n }})=\rho(g_i^{a_i}\ldots g_{n }^{a_{n }}),
      	\]
      	for all $0\leq a_l\leq m_l-1$, $i\leq l\leq n$.  
      	 By mathematical induction, it follows that 
      	 $\rho(g_1^{a_1}\ldots g_{n }^{a_{n }})=\rho(g_1^{a_1})\ldots \rho(g_{n }^{a_{n }})$, for all $g_l^{a_l}\in C_l$, $1\leq l\leq n$. 
				
				Likewise, for $1\leq l\leq n$, $j_s=\delta_{s, l}a$ and $k_t=\delta_{t, l}b$ for all $1\leq s, t\leq n$ ($0\leq a, b\leq m_l-1$ are arbitrary fixed and 
				$\delta_{s, l}$ is the Kronecker's delta of $s$ and $l$, etc.), 
				we have 
				\begin{eqnarray*}
				\flat_{\mfg}\omega_{\overline{\overline{a}}}(g_l^{a},g_l^{b})&=&\prod_{1\leq r<s<t\leq n}\xi_{(m_r,m_s,m_t)}^{-c_{rst}k_rj_sf_t}
				\prod_{1\leq r<s<t\leq n}\xi_{(m_r,m_s,m_t)}^{-c_{rst}f_rk_sj_t}
				\prod_{1\leq r<s<t\leq n}\xi_{(m_r,m_s,m_t)}^{c_{rst}k_r f_s j_t}\\
				&=&
				\prod_{1\leq r<l<t\leq n}\xi_{(m_r,m_l,m_t)}^{-c_{rlt}ak_rf_t}
				\prod_{1\leq r<s<l\leq n}\xi_{(m_r,m_s,m_l)}^{-c_{rsl}af_rk_s}
				\prod_{1\leq r<s<l\leq n}\xi_{(m_r,m_s,m_l)}^{c_{rsl}ak_r f_s}=1,
				\end{eqnarray*}
				and therefore $\rho(g_l^{a})\rho(g_l^{b})=\rho(g_l^{(a+b)_{m_l}'})$. It is clear now that  
      \[
		\rho( g_1^{ (j_1+ k_1)'_{m_1}}\ldots  g_n^{ (j_n+k_n)'_{m_n}} )=\rho( g_1^{ j_1})\rho(g_1^{ k_1})\ldots \rho( g_n^{ j_n}  ) \rho( g_n^{ k_n}  )=
			\rho( g_1^{ j_1}\ldots  g_n^{ j_n}  )\rho(  g_1^{ k_1}\ldots  g_n^{ k_n}  ), 
			\]
 and we conclude that $\flat_{\mf g}\omega_{\ov{\ov{a}}}(g_1^{ j_1}\ldots  g_n^{ j_n},g_1^{ k_1}\ldots  g_n^{ k_n}  )=1$, for all 
		$g_1^{j_1}\cdots g_n^{j_n}, g_1^{k_1}\cdots g_n^{k_n}\in G$, as stated.
		
		It remains to show that \equref{condodcob2cocyclenonab} is equivalent to the divisibility in \equref{divis}. For this, we assume that 
		$n\geq 3$ (for $n=2$ there is nothing to prove).
		
		Clearly, \equref{condodcob2cocyclenonab} is equivalent to 
		\begin{equation}\eqlabel{maindivis}
		M|S:=-\sum\limits_{1\leq r<s<t\leq n}M^c_{rst}(k_rj_sf_t+f_rk_sj_t-k_rf_sj_t),~\forall~ 
		0\leq k_l, j_l\leq m_l-1,~1\leq l\leq n.
		\end{equation}
		We will rewrite $S$ so that (formally speaking) we can get as common factor the scalar $k_rj_s$, where $r, s$ are given; 
		$1\leq r<s\leq n$. Towards this end, we split 
		$S$ into a sum of three sums, $S_1:=-\sum_{1\leq r<s<t\leq n}M^c_{rst}k_rj_sf_t$, $S_2:=-\sum_{1\leq r<s<t\leq n}M^c_{rst}f_rk_sj_t$ and 
		$S_3:=\sum_{1\leq r<s<t\leq n}M^c_{rst}k_rf_sj_t$. For a given pair $(r, s)$ with $1\leq r<s\leq n$, the scalar $k_rj_s$ appears as a summand of 
		$S_1$ defined by $r<s<w\leq n$, as a summand of $S_2$ determined by $1\leq u<r<s$ and as a summand of $S_3$ given by $r<v<s$. 
		Hence, \equref{maindivis} is equivalent to the fact that, for all $0\leq k_l, j_l\leq m_l-1$, $1\leq l\leq n$, 
		\[
		M|\sum\limits_{1\leq r<s\leq n}k_rj_s\left(-\sum\limits_{u=1}^{r-1}M^c_{urs}f_u + 
		\sum\limits_{v=r+1}^{s-1}M^c_{rvs}f_v -\sum\limits_{w=s+1}^nM^c_{rsw}f_w \right).
		\]
		It is clear now that this fact happens if and only if the divisibility in \equref{divis} hold.  
\end{proof}

For $n\geq 3$, the divisibility in \equref{divis} reduce to a system 
of $\left(\begin{array}{c}n\\2\end{array}\right)$ linear congruences ${\rm mod}~M$, in $n$ unknowns 
$f_1,\cdots, f_n$. Equivalently, \equref{divis} reduce to a homogeneous linear system in $n+\left(\begin{array}{c}n\\2\end{array}\right)$ unknowns 
$f_1,\cdots, f_n, x_{rs}$, $1\leq r<s\leq n$; namely, 
\begin{equation}\eqlabel{linsyst}
XA=0,~\mbox{where}~X=(f_1,\cdots,f_n, x_{12},\cdots,x_{1n}, x_{23}, \cdots, x_{2n}, \cdots, x_{n-1n})~\mbox{and}
\end{equation} 
\[\footnotesize{
A=\left(
\begin{array}{ccccc}
A_{n-1}&A_{n-2}&\cdots&A_2&A_1\\
MI_{n-1}&0&\cdots&0&0\\
0&MI_{n-2}&\cdots&0&0\\
\vdots&\vdots&\vdots&\vdots&\vdots\\
0&0&\cdots&0&MI_1
\end{array} 
\right)~;~
A_{n-j}=\left(
\begin{array}{cccc}
M^c_{1jj+1}&M^c_{1jj+2}&\cdots&M^c_{1jn}\\
\vdots&\vdots&\vdots&\vdots\\
M^c_{j-1jj+1}&M^c_{j-1jj+2}&\cdots&M^c_{j-1jn}\\
0&0&\cdots&0\\
0&-M^c_{jj+1j+2}&\cdots&-M^c_{jj+1n}\\
M^c_{jj+1j+2}&0&\cdots&-M^c_{jj+2n}\\
M^c_{jj+1j+3}&M^c_{jj+2j+3}&\cdots&-M^c{jj+3n}\\
\vdots&\vdots&\vdots&\vdots\\
M^c_{jj+1n}&M^c_{jj+2n}&\cdots&0
\end{array}
\right)
}
\]
is an $n$ by $n-j$ matrix, for all $1\leq j\leq n-1$. The system \equref{linsyst} can be solved as follows. 

$A$ is an $n+{\mf n}$ by ${\mf n}$ matrix with integer coefficients, where, for simplicity, we have denoted $\left(\begin{array}{c}n\\2\end{array}\right)$ 
by ${\mf n}$. As the ring of 
integers $\mathbb{Z}$ is a $PID$, there exist unimodular matrices $U$, $V$ such that $UAV={\rm diag}(d_1,\cdots,d_r,0,\cdots,0)$; here $U$ is an 
$n+{\mf n}$-square matrix, $V$ is an ${\mf n}$-square matrix and $d_1,\cdots,d_r$ 
are non-zero integers such that $d_1|d_2|\cdots|d_r$, $1\leq r\leq {\mf n}$. Thus \equref{linsyst} 
is equivalent to $(XU^{-1})D=0$, where $D={\rm diag}(d_1,\cdots,d_r,0,\cdots,0)$. This says that $XU^{-1}=(0,\cdots,0,y_{r+1},\cdots,y_{{\mf n}})$ for some  
integers $y_{r+1},\cdots, y_{\mf n}$. 
Then $X=U(0,\cdots,0,y_{r+1},\cdots,y_{\mf n})$ is the general solution 
of \equref{linsyst}, solution determined by an arbitrary family of integers $y_{r+1},\cdots, y_{\mf n}$. Furthermore, we show that 
$r={\mf n}$ and that $U$ can be somehow obtained from the Schmidt normal form $D'=U'A'V'$ of $A':=(A_{n-1},\cdots, A_1)$, 
an $n$ by ${\mf n}$ matrix. 
 
Indeed, if, in general, $O_{m, p}$ stands for the null $m$ by $p$ matrix and $O_m=O_{m, m}$, we see that
\[
{\footnotesize
\left(\begin{array}{cc}
U'&O_{n, {\mf n}}\\
O_{{\mf n}, n}&V'^{-1}
\end{array}\right)
\left(\begin{array}{c}
A'\\
MI_{\mf n}
\end{array}\right)V'=\left(\begin{array}{c}
U'A'\\
MV'^{-1} \\
\end{array}\right)V'=\left(\begin{array}{c}
D'\\
MI_{\mf n}
\end{array}\right)\equal{{\rm not}}B.
}
\]  
Assume now that $D'={\rm diag}(d'_1,\cdots,d_{r'},0,\cdots, 0)$, for non-zero integers $d'_1,\cdots, d'_{r'}$ which satisfy $d'_1|\cdots|d'_{r'}$; $r'\leq n$ 
is the rank of $A'$. For each $1\leq j\leq r'$ consider integers $\a_j, \b_j$ such that $\a_j d'_j + \b_j M=(d'_j, M)$, and define $U''$ as being the matrix  
\[
\footnotesize{\left(\hspace{-2mm}\begin{array}{cccccccccccc}
\a_1&\hspace{-4mm}0&\hspace{-4mm}\cdots&\hspace{-4mm}0&\hspace{-4mm}0&\hspace{-4mm} &\hspace{-4mm}\b_1&\hspace{-4mm}0&\hspace{-4mm}\cdots&\hspace{-4mm}0
&\hspace{-4mm}0&\hspace{-4mm}\\
0&\hspace{-4mm}\a_2&\hspace{-4mm}\cdots&\hspace{-4mm}0&\hspace{-4mm}0&\hspace{-4mm}  &\hspace{-4mm}0&\hspace{-4mm}\b_2&\hspace{-4mm}\cdots&\hspace{-4mm}0
&\hspace{-4mm}0&\hspace{-4mm}\\
\vdots&\hspace{-4mm}\vdots&\hspace{-4mm}\cdots&\hspace{-4mm}\cdots&\hspace{-4mm}\vdots&\hspace{-4mm}O_{r', n-r'} 
&\hspace{-4mm}\vdots&\hspace{-4mm}\vdots&\hspace{-4mm}\cdots&\hspace{-4mm}\vdots&\hspace{-4mm}\vdots&\hspace{-4mm}O_{r', {\mf n}-r'}\\
0&\hspace{-4mm}0&\hspace{-4mm}\cdots&\hspace{-4mm}\a_{r'-1}&\hspace{-4mm}0&\hspace{-4mm}                          
&\hspace{-4mm}0&\hspace{-4mm}0&\hspace{-4mm}\cdots&\hspace{-4mm}\b_{r'-1}&\hspace{-4mm}0&\hspace{-4mm} \\
0&\hspace{-4mm}0&\hspace{-4mm}\cdots&\hspace{-4mm}0&\hspace{-4mm}\a_{r'}&\hspace{-4mm} &\hspace{-4mm}0&\hspace{-4mm}0&\hspace{-4mm}\cdots&\hspace{-4mm}0 
&\hspace{-4mm}\b_{r'}&\hspace{-4mm}\\
 &\hspace{-4mm} &\hspace{-4mm}O_{{\mf n}-r', r'}&\hspace{-4mm} &\hspace{-4mm} &\hspace{-4mm}O_{{\mf n}-r', n-r'}&\hspace{-4mm} 
&\hspace{-4mm} &\hspace{-4mm} O_{{\mf n}-r', r'}&\hspace{-4mm} &\hspace{-4mm} &\hspace{-4mm}I_{{\mf n}-r'}\\
-\frac{M}{(d'_1, M)}&\hspace{-4mm}0&\hspace{-4mm}\cdots&\hspace{-4mm}0&\hspace{-4mm}0&\hspace{-4mm} &\hspace{-4mm}\frac{d'_1}{(d'_1, M)}
&\hspace{-4mm}0&\hspace{-4mm}\cdots&\hspace{-4mm}0&\hspace{-4mm}0&\hspace{-4mm} \\
0&\hspace{-4mm}-\frac{M}{(d'_2, M)}&\hspace{-4mm}\cdots&\hspace{-4mm}0&\hspace{-4mm}0&\hspace{-4mm} &\hspace{-4mm}0
&\hspace{-4mm}\frac{d'_2}{(d'_2, M)}&\hspace{-4mm}\cdots&\hspace{-4mm}0&\hspace{-4mm}0&\hspace{-4mm} \\
\vdots&\hspace{-4mm}\vdots&\hspace{-4mm}\cdots&\hspace{-4mm}\vdots&\hspace{-4mm}\vdots&\hspace{-4mm} O_{r', n-r'}&\hspace{-4mm}\vdots
&\hspace{-4mm}\vdots&\hspace{-4mm}\cdots&\hspace{-4mm}\vdots&\hspace{-4mm}\vdots&\hspace{-4mm}O_{r', {\mf n}-r'}\\
0&\hspace{-4mm}0&\hspace{-4mm}\cdots&\hspace{-4mm}-\frac{M}{(d'_{r'-1}, M)}&\hspace{-4mm}0&\hspace{-4mm} &\hspace{-4mm}0&\hspace{-4mm}0
&\hspace{-4mm}\cdots &\hspace{-4mm}\frac{d'_{r'-1}}{(d'_{r'-1}, M)}&\hspace{-4mm}0&\hspace{-4mm} \\
0&\hspace{-4mm}0&\hspace{-4mm}\cdots&\hspace{-4mm}0&\hspace{-4mm}-\frac{M}{(d'_{r'}, M)}&\hspace{-4mm} &\hspace{-4mm}0&\hspace{-4mm}0&\hspace{-4mm}\cdots
&\hspace{-4mm}0&\hspace{-4mm}\frac{d'_{r'}}{(d'_{r'}, M)}&\hspace{-4mm} \\
&\hspace{-4mm} &\hspace{-4mm} O_{n-r', r'}&\hspace{-4mm} &\hspace{-4mm} &\hspace{-4mm} I_{n-r'}&\hspace{-4mm} &\hspace{-4mm} &\hspace{-4mm} O_{n-r', r'}
&\hspace{-4mm} &\hspace{-4mm} &\hspace{-4mm} O_{n-r', {\mf n}-r'}
\end{array}\hspace{-2mm}\right).
}
\]
The Laplace's cofactor expansion along the first two rows of $U''$ together with a mathematical induction on $r'$ ensures us that $U''$ is of determinant 
equal to $(-1)^{r'}$. Also, an elementary computation leads us to the equality 
\[
U''B={\rm diag}((d'_1, M),\cdots,(d'_{r'},M),M,\cdots,M)=D,
\]
which says that the invariant factors of $B$, and therefore of $A$, are $(d'_1, M)|\cdots|(d'_{r'},M)|M\cdots|M$, thus in number of ${\mf n}$. Finally, we can take 
$U=U''{\footnotesize
\left(\begin{array}{cc}
U'&O_{n, {\mf n}}\\
O_{{\mf n}, n}&V'^{-1}
\end{array}\right)}$ and $V=V'$ as unimodular matrices obeying $UAV=D$, from where we deduce that \equref{linsyst} has the solution 
$X=(0,\cdots, 0, y_1,\cdots, y_{n})U$ (we have renumbered the $y$'s). It follows now that 
\begin{equation}\eqlabel{sollinsyst}
(f_1,\cdots, f_n)=
X\left(\begin{array}{cc}
I_n&O_{n,\mf n}\\
O_{{\mf n},   n}& O_{{\mf n},{\mf n}}
\end{array} \right)
=
\left(-\frac{y_1M}{(d'_1,M)},\cdots, -\frac{y_{r'}M}{(d'_{r'},M)}, y_{r'+1},\cdots, y_{n}\right)U'
\end{equation} 
is the solution of \equref{divis}; $y_1,\cdots,y_{n}$ are arbitrary integers. 
 
The method presented above allows us to concretely determine the solution of \equref{linsyst} in the case where $n=3$. 
In general, for an $m$ by $p$ matrix $A$ with integer entries and $1\leq j\leq {\rm min}(m, p)$, $\Delta_j(A)$ is our notation for the GCD of all $j$ by $j$ minors of $A$; we make the convention that if 
all $j$ by $j$ minors of $A$ are equal to zero, then $\Delta_j(A)=0$. It is well known that, modulo a sign, the invariants factors $d_1,\cdots, d_r$ of $A$ 
are given by $d_j=\frac{\Delta_{j}(A)}{\Delta_{j-1}(A)}$, for all $1\leq j\leq r$, where $\Delta_0(A):=1$. 

\begin{example}\exlabel{caseneq3}
With notation as in \prref{flat coboundary}, specialized to the case $n=3$, $\flat_{\mf g} \omega_{\underline a}$ is coboundary if and only if 
$f_1=\mu M'$, $f_2=\nu M'$, $f_3=\l M'$, for some natural numbers $\l, \mu, \nu$; $M'=\frac{M}{(M,c_{123})}$.
\end{example} 
\begin{proof}
The matrix $A$ takes the form $\left(\begin{array}{c}A'\\MI_3\end{array}\right)$; here 
$A'=\left(\begin{array}{ccc}0&0&\upsilon\\0&-\upsilon&0\\ \upsilon&0&0\end{array}\right)$, 
where $\upsilon:=M^c_{123}=\frac{c_{123}M}{(m_1, m_2, m_3)}=c_{123}$. If ${\mf d}:=(\upsilon, M)$, one can see easily that $\Delta_1(A)={\mf d}$, 
$\Delta_2(A)=(\upsilon^2, M^2, \upsilon M)={\mf d}^2\left(\left(\frac{\upsilon}{\mf d}\right)^2, \left(\frac{M}{\mf d}\right)^2,\frac{\upsilon}{\mf d}\frac{M}{\mf d}\right)={\mf d}^2$ and, similarly, $\Delta_3(A)=(\upsilon^3, M^3, \upsilon^2M, \upsilon M^2)={\mf d}^3$. Thus, $d_1=d_2=d_3={\mf d}$. Also, 
one can see directly that for  
$u, w\in \mathbb{Z}$ such that $u\upsilon + w M={\mf d}$, the matrix 
\[
U=\left(\begin{array}{cccccc}
0&0&u&w&0&0\\
0&-u&0&0&w&0\\
u&0&0&0&0&w\\
0&0&-M'&\upsilon'&0&0\\
0&M'&0&0&\upsilon'&0\\
-M'&0&0&0&0&\upsilon'
\end{array}\right)~\mbox{satisfies}~
UA=\left(\begin{array}{ccc}
{\mf d}&0&0\\
0&{\mf d}&0\\
0&0&{\mf d}\\
0&0&0\\
0&0&0\\
0&0&0
\end{array}\right)
\]
and that $U$ is unimodular, provided that $M'=\frac{M}{\mf d}$ and $\upsilon'=\frac{\upsilon}{\mf d}$. Hence, for $n=3$, the system \equref{linsyst} has  
$\{(0, 0, 0, -\l, \nu -\mu)U|\mu, \nu, \l\in \mathbb{Z}\}=
\{(\mu M', \nu M', \l M', -\l\upsilon', \nu \upsilon', -\mu\upsilon')|\mu, \nu, \l\in \mathbb{Z}\}$ as space of solutions. We are done. 
\end{proof}

\begin{remark}
By the proof of \prref{flat coboundary}, $\flat_{\mf g}\omega_{\underline a}$ is coboundary if and only if \equref{condodcob2cocyclenonab} holds, 
in which case $\flat_{\mf g}\omega_{\underline a}=\flat_{\mf g}\omega_{\bar{a}}$. It is easy to see that \equref{condodcob2cocyclenonab} holds in the 
case when $(m_r, m_s, m_t)|c_{rst}(f_r, f_s, f_t)$, for all $1\leq r<s<t\leq n$. For $n=3$, \exref{caseneq3} tells us that 
$M=(m_1, m_2, m_3)\mid (M, c_{123})(f_1, f_2, f_3)$, which is equivalent to $(m_1, m_2, m_3)|c_{123}(f_1, f_2, f_3)$, 
is a necessary and sufficient condition for $\flat_{\mf g}\omega_{\underline a}$ to be coboundary. 

As a concrete example, take $m_1=m_2=m_3=4$. Then ${\mf g}\in G$ satisfying \equref{condodcob2cocyclenonab} is $\mf g=1_G$ if $c_{123}\in\{1,3\}$, and 
respectively $\mf g\in\{g_1^ig_2^jg_3^k\mid i,j,k\in\{0,2\}\}$ if $c_{123}=2$.
\end{remark}

\begin{remark}
The rank of $A'$ is not, in general, equal to $n$. For instance, we next see that for $n=4$ the rank of $A'$ is always less or equal to $3$.
For this, if $a_1, \cdots, a_4\in \mathbb{Z}$, denote by $a_{i_1\cdots i_k}$ the GCD of $a_{i_1},\cdots, a_{i_k}$, for all 
$1\leq k\leq 4$ and $\{i_1,\cdots,i_k\}\subseteq \{1,\cdots,4\}$.   

For $n=4$, $A=\left(\begin{array}{c}A'\\MI_6\end{array}\right)$, where $A'={\footnotesize \left(\begin{array}{cccccc}
0&0&0&-a_1&-a_2&-a_3\\
0&a_1&a_2&0&0&-a_4\\
-a_1&0&a_3&0&a_4&0\\
-a_2&-a_3&0&-a_4&0&0
\end{array}\right)}
$, provided that $a_1=-M^c_{123}$, $a_2=-M^c_{124}$, $a_3=-M^c_{134}$ and $a_4=-M^c_{234}$. We have that $\Delta_1(A')=a_{1234}$, 
\begin{eqnarray*}
\Delta_2(A')&=&(a_i^2, a_ia_j|1\leq i<j\leq 4)=(a_1a_{12}, a_1a_{34}, a_2a_{23}, a_2a_{14}, a_3a_{34}, a_3a_{12}, a_4a_{34}, a_4a_{12})\\
&=&(a_1a_{1234}, a_{2}a_{1234}, a_3a_{1234}, a_4a_{1234})=a_{1234}^2,\\
\Delta_3(A')&=&(a_ia_ja_l|1\leq i, j, l\leq 4)=a_{1234}(a_t^2, a_ia_j|1\leq t\leq 4, 1\leq i<j\leq 4)=a_{1234}^3,
\end{eqnarray*}
and $\Delta_4(A')=0$. Thus, $r'\leq 3$ and $r'=3$ if and only if at least one of the $a_j$'s is non-zero. 

If $U', V'$ are unimodular matrices obeying $U'A'V'={\footnotesize \left(\begin{array}{cc}
a_{1234}I_3&O_3\\
O_{1,3}&O_{1,3}
\end{array}\right)}$, it follows that the solution of \equref{linsyst} is parametrized by $y_1, y_2, y_3, y_4\in \mathbb{Z}$, in the sense that 
\[ 
(f_1, f_2, f_3, f_4)=\left(-\frac{y_1M}{(a_{1234}, M)}, -\frac{y_2M}{(a_{1234, M})}, -\frac{y_3M}{(a_{1234}, M)}, y_4\right)U'.
\]
\end{remark}

Let $n\geq 3$. For $G=C_{m_1}\times \ldots \times C_{m_n}$ a finite abelian group and $\omega_{\underline a}$ a $3$-cocycle on $G$ determined by 
$\underline a\in \cal A$ as in \equref{3cocyfabgr}, 
denote by $Y_{\un{y}}$ the $1$ by $n$ matrix defined by $\un{y}=(y_1,\cdots,y_n)\in \mathbb{Z}^n$ as follows:
\[
Y_{\un{y}}=\left(\frac{-y_1M}{(d'_1,M)},\cdots,
\frac{-y_{r'}M}{(d'_{r'},M)}, y_{r'+1},\cdots,y_n\right). 
\]
As before, $M=LCM\{(m_r, m_s, m_t)|1\leq r<s<t\leq n\}$ and $d'_1|\cdots|d'_{r'}$ are the invariant factors of the matrix $A'$ that is part of the definition 
of the matrix $A$ associated to the system \equref{linsyst}. So, if $U',V'$ are unimodular and such that $U'A'V'={\rm diag}(d'1,\cdots,d'_{r'},0,\cdots,0)$, 
then $Y_{\un{y}}U'$ is solution for \equref{linsyst} and any solution of \equref{linsyst} is of this type. 

Furthermore, denote by $C$ the symmetric matrix which has the main diagonal given by the elements $\frac{c_1}{m^2_1},\cdots,\frac{c_n}{m^2_n}$ 
and on the position $(s, t)$ the element $\frac{c_{st}}{2m_sm_t}$, for all $1\leq s<t\leq n$. 
Also, by $C'$ denote the matrix $U'C{}^tU'$; in general, 
${}^tZ$ stands for the transpose matrix of $Z$. Clearly, $C'$ is a symmetric matrix with rational entries.    

\begin{proposition}\prlabel{chara abel}
    Let $G=C_{1}\times \ldots \times C_{n}$ be a finite abelian group as in the above, and $\omega_{\underline a}$ a $3$-cocycle on $G$ defined 
		by $\underline a\in \cal A$. 
		Let $N={\rm LCM}\{m_j^2, m_sm_t|1\leq j\leq n, 1\leq s<t\leq n\}$. 
		
		(i) If there exists a non-trivial $2$-dimensional braided Hopf algebra in 
		${}_{k^G_{\omega_{\underline a}}}^{k^G_{\omega_{\underline a}}}{\cal YD}$ then $N$ is even; consequently, at least one of $m_1,\cdots,m_n$ is even.  
		
		(ii) Assuming $N$ even, there is a one to one correspondence between non-trivial $2$-dimensional braided Hopf algebras in 
		${}_{k^G_{\omega_{\underline a}}}^{k^G_{\omega_{\underline a}}}{\cal YD}$ and	
		$\un{y}=(y_1,\cdots,y_n)\in\mathbb{Z}^n$ for which there exists 
		$(\l_1,\cdots,\l_n)\in\mathbb{Z}^n$ such that, for $\Lambda':=\left(\frac{\l_1}{m_1},\cdots,\frac{\l_n}{m_n}\right)U'$, 
		\begin{equation}\eqlabel{toughcondy}
		\Lambda'{}~^tY_{\un{y}} + Y_{\un{y}}~C'~{}^tY_{\un{y}}-\frac{1}{2}\in\mathbb{Z}.
		\end{equation}			
\end{proposition}
\begin{proof}
Owing to \prref{fun chara}, the non-trivial $2$-dimensional Hopf algebras within 
${}_{k^G_{\omega_{\underline a}}}^{k^G_{\omega_{\underline a}}}{\cal YD}$ are characterized by pairs $({\mf g}, \rho)$ 
consisting of an element ${\mf g}=g_1^{f_1}\cdots g_n^{f_n}\in G$ and a map $\rho: G\ra k^*$ fulfilling 
$\rho(e)=1$, $\rho({\mf g})=-1$ and $\flat_{\mf g}{\omega_{\underline a}}=\partial\rho$. By keeping the notation established in the 
proof of \prref{flat coboundary}, such a pair $({\mf g}, \rho)$ is determined by a solution $(f_1,\cdots,f_n)$ of \equref{linsyst} and a group morphism 
$\vartheta: G\ra k^*$ obeying $\vartheta({\mf g})f_{\mf g}({\mf g})=-1$ (in general, if $\sigma$ is a coboundary $2$-cocycle and 
$\sigma=\partial \rho_1=\partial \rho_2$, for some maps $\rho_1, \rho_2: G\ra k^*$, 
then $\vartheta: G\ra k^*$ given by $\vartheta(g)=\rho_1(g)\rho_2(g)^{-1}$, for all $g\in G$, is a 
group morphism). But $\vartheta$ is given by an $n$-tuple $(\l_1,\cdots,\l_n)$ with $0\leq \l_j\leq m_j-1$, for all $1\leq j\leq n$, 
in the sense that $\vartheta(g_1^{a_1}\cdots g_n^{a_n})=\xi_{m_1}^{\l_1a_1}\cdots\xi_{m_n}^{\l_na_n}$, for all $0\leq a_j\leq m_j-1$; 
$1\leq j\leq n$.   

Take $\xi_N$ be a primitive root of unity of order $N$ in $k$ and assume, without loss of generality, that $\xi_{m_j}=\xi_N^{\frac{N}{m_j}}$, 
$\xi_{m_j^2}=\xi_N^{\frac{N}{m_j^2}}$ and $\xi_{m_sm_t}=\xi_N^{\frac{N}{m_sm_t}}$, for all $1\leq j\leq n$ and 
$1\leq s<t\leq n$. Then, 
\begin{equation}\eqlabel{varfg}
\vartheta(g_1^{a_1}\cdots g_n^{a_n})=\xi_{N}^{\sum\limits_{j=1}^n\frac{\l_ja_j}{m_j}N}~\mbox{and}~
f_{\mf g}(g_1^{a_1}\cdots g_n^{a_n})=\xi_N^{\sum\limits_{j=1}^n\frac{c_ja_jf_j}{m^2_j}N + \sum\limits_{1\leq s<t\leq n}\frac{c_{st}f_ta_s}{m_tm_s}N},
\end{equation} 
for all $0\leq a_j\leq m_j-1$, $1\leq j\leq n$. 
So, giving a non-trivial $2$-dimensional Hopf algebra within 
${}_{k^G_{\omega_{\underline a}}}^{k^G_{\omega_{\underline a}}}{\cal YD}$ is equivalent to giving a solution 
$(f_1,\cdots,f_n)$ of \equref{linsyst} such that, for some integers $\l_1,\cdots,\l_n$,  
\begin{equation}\eqlabel{toughcond}
\zeta_N^{\sum\limits_{j=1}^n\frac{\l_jf_j}{m_j}N + 
\sum\limits_{j=1}^n\frac{c_jf^2_j}{m^2_j}N + \sum\limits_{1\leq s<t\leq n}\frac{c_{st}f_sf_t}{m_sm_t}N}=-1.
\end{equation} 
Clearly, \equref{toughcond} holds if and only if $N$ is even and 
$\sum\limits_{j=1}^n\frac{\l_jf_j}{m_j} + 
\sum\limits_{j=1}^n\frac{c_jf^2_j}{m^2_j} + \sum\limits_{1\leq s<t\leq n}\frac{c_{st}f_sf_t}{m_sm_t}-\frac{1}{2}$ is an integer. 
As any solution $(f_1,\cdots,f_n)$ of \equref{linsyst} is of the form $Y_{\un{y}}U'$, for some $\un{y}\in \mathbb{Z}^n$, our proof follows. 
\end{proof}

We use the isomorphism provided by \equref{canisomgralg} in order to return to our initial context: $G$ is a finite abelian group 
and $k_{\Phi_{\un{a}}}[G]$ is the Hopf group algebra of $k$ and $G$ considered as quasi-Hopf algebra with reassociator 
\[
\Phi_{\un{a}}=\sum\limits_{{\mf i}, {\mf j}, {\mf l}\in \{\mf m\}}
\omega_{\underline a}(g_1^{i_1}\ldots  g_n^{i_n},g_1^{j_1}\ldots g_n^{j_n}, g_1^{l_1}\ldots g_n^{l_n})
1_{i_1,\cdots ,i_n}\ot 1_{j_1,\cdots ,j_n}\ot 1_{l_1,\cdots , l_n}.
\]  
Here, assuming $C=C_1\times\cdots\times C_n$ as in the above, we have denoted by 
${\mf i}$ an $n$-tuple $(i_1,\cdots, i_n)\in \mathbb{Z}^n$ with the property that $0\leq i_l\leq m_l-1$, for all $1\leq l\leq n$, and 
by $\{\mf m\}$ the set of all these $n$-tuples ${\mf i}$. 

\begin{corollary}
With notation as in \prref{chara abel}, if there exists a non-trivial $2$-dimensional Hopf algebra within 
${}_{k_{\Phi_{\un{a}}}[G]}^{k_{\Phi_{\un{a}}}[G]}{\cal YD}$ then $N$ is even. Furthermore, assuming $N$ even, a non-trivial 
Hopf algebra of dimension $2$ in ${}_{k_{\Phi_{\un{a}}}[G]}^{k_{\Phi_{\un{a}}}[G]}{\cal YD}$ is of the form $B_{\alp, v}$, 
with $\alp: k[G]\ra k$ an algebra 
morphism and $v\in k[G]$, both determined by a solution ${\mf f}=(f_1,\cdots,f_n)\in \mathbb{Z}^n$ of \equref{linsyst} and 
$(\l_1,\cdots,\l_n)\in \mathbb{Z}^n$ obeying \equref{toughcond} as follows:
\[\alp(g_1^{j_1}\cdots g_n^{j_n})=\prod\limits_{l=1}^n\xi_{m_l}^{f_lj_l},~\forall~g_1^{j_1}\cdots g_n^{j_n}\in G;~
v=\sum\limits_{{\mf l}\in \{\mf m\}}\zeta_N^{\left(\sum\limits_{j=1}^n\frac{\l_j l_j}{m_j}+\sum\limits_{j=1}^n\frac{c_jf_jl_j}{m_j^2} + 
\sum\limits_{1\leq s<t\leq n}\frac{c_{st}l_sf_t}{m_sm_t}\right)N}1_{l_1,\cdots ,l_n}.
\]
\end{corollary}
\begin{proof}
It follows from \prref{chara abel} and the quasi-Hopf algebra identification $k^G\cong k[G]$ produced by the correspondence 
$P_{g_1^{j_1}\ldots g_n^{j_n}}\mapsto 1_{j_1,\ldots , j_n}$. Keeping in mind the proof of \leref{1d fun yd}, for ${\mf f}$ and 
$(\l_1,\cdots, \l_n)$ as in the statement, we deduce that ($\vartheta$ and $f_{\mf g}$ are as in \equref{varfg})  
\begin{eqnarray*}
&&\alp(g_1^{j_1}\cdots g_n^{j_n})=\prod\limits_{l=1}^n\left(\sum\limits_{t=0}^{m_l-1}\xi_{m_l}^{tj_l}
\widetilde{P}_{g_l^t}(g_1^{f_1}\cdots g_n^{f_n})\right)=\prod\limits_{l=1}^n\xi_{m_l}^{f_lj_l},~\mbox{and}\\
&&v=\sum\limits_{{\mf l}\in \{\mf m\}}(\vartheta f_{\mf g})(g_1^{l_1}\cdots g_n^{l_n})1_{l_1,\cdots, l_n}=
\sum\limits_{{\mf l}\in \{\mf m\}}\zeta_N^{\left(\sum\limits_{j=1}^n\frac{\l_j l_j}{m_j}+\sum\limits_{j=1}^n\frac{c_jf_jl_j}{m_j^2} + 
\sum\limits_{1\leq s<t\leq n}\frac{c_{st}l_sf_t}{m_sm_t}\right)N}1_{l_1,\cdots ,l_n},
\end{eqnarray*}	
as stated. Note that $\sigma(v)=-1$ is equivalent to \equref{toughcond}, since $\alp(1_{j_1,\cdots,j_n})=\delta_{{\mf j}, {\mf f}}$. 
\end{proof}
 
\begin{examples}
Take $n=3$, $G=C_{m_1}\times C_{m_2}\times C_{m_3}$ and $\omega_{\un{a}}$ a $3$-cocycle on $G$. Recall that 
$M=(m_1, m_2, m_3)$, $M'=\frac{M}{(M, c_{123})}$ and that any solution of \equref{linsyst} is of the form 
$(f_1, f_2, f_3)=(\mu_1M', \mu_2M', \mu_3M')$, with $\mu_1, \mu_2,\mu_3\in \mathbb{Z}$. It is immediate that $M'|M|m_j$, for all 
$1\leq j\leq 3$. Set $m'_j:=\frac{m_j}{M'}$, for all $1\leq j\leq 3$. 
    
In this particular case, \equref{toughcond} becomes $\sum_{j=1}^3\frac{\l_j\mu_j}{m'_j} +\sum_{j=1}^3\frac{c_j\mu_j^2}{m'^{2}_j} + 
\sum_{1\leq r<s\leq 3}\frac{c_{rs}\mu_r\mu_s}{m'_rm'_s}-\frac{1}{2}\in \mathbb{Z}$. For this to happen, at least one of 
the integers $m'_1$, $m'_2$, $m'_3$ must be even. This is why, we assume that $J_{e}:=\{1\leq j\leq 3| m'_j\in 2\mathbb{Z}\}$ is 
a non-empty set. Furthermore, we set $\mu_j=0$, for all $j\in \{1, 2, 3\}\backslash J_e$, and $\mu_j=\frac{m'_j}{2}$, for all 
$j\in J_e$. Then, if ${\mf c}:=\sum_{j\in J_e}c_j + \sum_{r, s\in J_e; r<s}c_{rs}$, \equref{toughcond} becomes 
$\sum_{j\in J_e}\frac{\l_j}{2} +\frac{\mf c}{4}-\frac{1}{2}\in \mathbb{Z}$. The latest condition is fulfilled if and only if 
${\mf c}$ is even and $\sum_{j\in J_e}{\l_j} +\frac{\mf c}{2}$ is an odd number.  
Thus, if $J_e$ is non-empty and ${\mf c}$ is even, there exists at least one non-trivial $2$-dimensional Hopf algebra within 
${}_{k_{\Phi_{\un{a}}}[G]}^{k_{\Phi_{\un{a}}}[G]}{\cal YD}$. 

The condition ${\mf c}$ even is just sufficient. For instance, take $m_1=12$, $m_2=18$, $m_3=30$ and $c_{123}=4$, so that $M=6$, 
$M'=3$, $m'_1=4$, $m'_2=6$, $m'_3=10$ and $J_e=\{1, 2, 3\}$. Then, for $c_1=3$, $c_2=7$, $c_3=14$, $c_{12}=1$, $c_{13}=3$ and 
$c_{23}=5$ we compute ${\mf c}=33$, an odd number. Furthermore, for $(f_1, f_2, f_3)=(3\mu_1, 3\mu_2, 3\mu_3)$ with $0\leq \mu_1\leq 3$, 
$0\leq \mu_2\leq 5$ and $0\leq \mu_3\leq 9$, a solution of \equref{linsyst} in this particular case, the condition 
\equref{toughcond} becomes 
\[
\frac{3}{4}\l_1\mu_1 + \frac{1}{2}\l_2\mu_2 + \frac{3}{10}\l_3\mu_3 + \frac{27}{16}\mu_1^2 + \frac{7}{4}\mu_2^2 + \frac{63}{50}\mu_3^2 
+ \frac{3}{8}\mu_1\mu_2+\frac{27}{40}\mu_1\mu_3 +\frac{3}{4}\mu_2\mu_3 -\frac{1}{2}\in \mathbb{Z}.
\]
For $\mu_2=2$ and $\mu_3=5$, it reduces to 
$\frac{3}{4}\l_1\mu_1 + \frac{3}{2}\l_3 + \frac{27}{16}\mu_1^2+ \frac{33}{8}\mu_1 -\frac{1}{2}\in \mathbb{Z}$, 
and is satisfied in each of the cases: (1) $\mu_1=0$, $0\leq \l_1\leq 11$, $0\leq \l_2\leq 17$ and $\l_3\in \{1, 3,\cdots, 29\}$; 
(2) $\mu_1=2$, $0\leq \l_2\leq 17$, and $0\leq \l_1\leq 11$ and $0\leq \l_3\leq 29$ such that $\l_1+\l_3$ is odd. We conclude that 
this case provides plenty of non-trivial $2$-dimensional braided Hopf algebras, although ${\mf c}$ is odd. 

We end the list of examples by pointing out that it might happen to have no non-trivial $2$-dimensional braided Hopf algebras; for instance, 
when $n=3$ and $m'_j$ is odd, for all $1\leq j\leq 3$. 
\end{examples}

Next, we deal with the case $n=2$. We expected things to be simpler here but, as we will see, this is not the case. So, assume 
$G=C_1\times C_2$ with $C_j$ generated by $g_j$, an element of order $m_j$; $1\leq j\leq 2$. As before, set $M=(m_1, m_2)$, the ${\rm GCD}$ 
of $m_1$, $m_2$, and let $\omega_{\un{a}}$ be the $3$-cocycle on $G$ defined by $\un{a}=(c_1,c_2,c_{12})$, $0\leq c_j\leq m_j-1$ and 
$0\leq c_{12}\leq M$, as follows: 
\[
\omega_{\underline a}(g_1^{i_1}g_2^{i_2}, g_1^{j_1}g_2^{j_2}, g_1^{k_1}g_2^{k_2}) 
		=\xi_{m_2}^{c_{12}i_2\lfloor   \frac{j_1+k_1}{m_1}\rfloor} \prod_{l=1}^2 \xi_{m_l}^{c_li_l\lfloor   \frac{j_l+k_l}{m_l}\rfloor}.
\]
As we have seen in the proof of \prref{flat coboundary}, $\flat_{\mf g}{\omega_{\underline a}}=\partial f_{\mf g}$ is a coboundary $2$-cocycle, for all 
${\mf g}=g_1^{f_1}g_2^{f_2}\in G$; here $f_{\mf g}: G\ra k^*$ is given by 
$f_{\mf g}( g_1^{ j_1}g_2^{ j_2}):= {\xi_{m_1m_2}^{c_{12}f_2j_1}} \prod_{l=1}^2\xi_{m_l^2}^{c_lf_lj_l}$, for all $g_1^{j_1}g_2^{j_2}\in G$. 
 
If $k_{\Phi_{\un{a}}}[G]$ is the  quasi-Hopf algebra with reassociator 
\[
\Phi_{\un{a}}=\sum_{{\mf i}, {\mf j}, {\mf l}\in \{\mf m\}}
\omega_{\underline a}(g_1^{i_1}g_2^{i_2}, g_1^{j_1}g_2^{j_2}, g_1^{k_1}g_2^{k_2})1_{i_1,i_2}\ot 1_{j_1, j_2}\ot 1_{l_1, l_2},
\] 
built on the Hopf group algebra $k[G]$, by using the same arguments as in the proof of \prref{chara abel} we get the following result.

\begin{corollary}\colabel{nontrvialneq2}
Let $N={\rm LCM}(m_1^2, m_2^2, m_1m_2)$. If there exists a non-trivial $2$-dimensional Hopf algebra within 
${}_{k_{\Phi_{\un{a}}}[G]}^{k_{\Phi_{\un{a}}}[G]}{\cal YD}$ then $N$ is even. 

Assuming $N$ even, a non-trivial Hopf algebra of dimension $2$ in ${}_{k_{\Phi_{\un{a}}}[G]}^{k_{\Phi_{\un{a}}}[G]}{\cal YD}$ is of the form $B_{\alp, v}$, 
with $\alp: k[G]\ra k$ an algebra morphism and $v\in k[G]$, determined by ${\mf f}=(f_1, f_2), {\mf \l}=(\l_1, \l_2)\in \{\mf m\}$ satisfying  
\begin{equation}\eqlabel{tough2}
\frac{\l_1f_1}{m_1} + \frac{\l_2f_2}{m_2} + c_1\frac{f_1^2}{m_1^2} + c_2\frac{f_2^2}{m_2^2} + c_{12}\frac{f_1f_2}{m_1m_2}-\frac{1}{2}\in \mathbb{Z},
\end{equation}
as follows: $\alp(g_1^{j_1}g_2^{j_2})=\xi_N^{\left(\frac{f_1j_1}{m_1} + \frac{f_2j_2}{m_2}\right)N}$, for all $g_1^{j_1}g_2^{j_2}\in G$, 
and 
\[
v=\sum\limits_{{\mf l}\in \{\mf m\}}\zeta_N^{\left(\sum\limits_{j=1}^2\frac{\l_j l_j}{m_j}+\sum\limits_{j=1}^n\frac{c_jf_jl_j}{m_j^2} + 
\frac{c_{12}l_1f_2}{m_1m_2}\right)N}1_{l_1, l_2}.
\]
\end{corollary}

The classification provided by \coref{nontrvialneq2} entails to condition \equref{tough2}. We will see in \seref{Appendix} that \equref{tough2} reduces to 
the finding of the rational points of a conic section, an arithmetic problem with a rather laborious solution.  
However, there are situations when \equref{tough2} can be solved on the spot, by using elementary methods. Such a situation occurs for instance 
when $m_1=m_2=2$, so in the case when $G$ is the Klein four group. As before, denote by $g_1, g_2$ the generators of $G$ and by 
$\sigma_{f_1f_2}: G\ra k$ the algebra morphism determined by $\sigma_{f_1f_2}(g_1^{j_1}g_2^{j_2})=(-1)^{f_1j_1+f_2j_2}$, for all $g_1^{j_1}g_2^{j_2}\in G$; 
here $(f_1, f_2)$ is a fixed pair of integers $f_1, f_2\in\{0, 1\}$. We also denote by ${\rm i}$ a primitive root of unity of order $4$ in $k$.
 
\begin{example}\exlabel{Klein}
For $G$ the Klein four group, the $2$-dimensional non-trivial Hopf algebras in ${}_{k_{\Phi_{\un{a}}}[G]}^{k_{\Phi_{\un{a}}}[G]}{\cal YD}$ 
are of the form $B_{\alp, v}$, with $\un{a}\in {\cal A}$, $\alp: k[G]\ra k$ and $v\in k[G]$ given by the table below,
\[
\begin{array}{|c||c|c|c||c|c|c|}
\hline
\un{a}=(c_1, c_2, c_3)&\alp_{f_1f_2}&(\l_1,\l_2)&v&\alp_{f_1f_2}&(\l_1, \l_2)&v\\
\hline
       &\alp_{01}&(0,1)&g_2&\alp_{10}&(1,0)&g_1\\
\cline{3-4}\cline{6-7}	
(0,0,0)&         &(1,1)&g_1g_2&      &(1,1)&g_1g_2\\
\cline{2-7}
       &\alp_{11}&(0,1)&g_2&\alp_{11}&(1,0)&g_1\\
\hline
(0,0,1)&\alp_{01}&(0,1)&{\mf g}^+:=\frac{1+{\mf i}}{2}g_2+\frac{1-{\mf i}}{2}g_1g_2&\alp_{10}&(1,0)&g_1\\
\cline{3-4}\cline{6-7}
       &         &(1,1)&{\mf g}^-:=\frac{1-{\mf i}}{2}g_2+\frac{1+{\mf i}}{2}g_1g_2&         &(1,1)&g_1g_2\\
\hline
(0, 1, 1)&\alp_{11}&(0,0)&{\mf h}^+:=\frac{1}{2}(g_1+g_2)+\frac{{\mf i}}{2}(1-g_1g_2)&\alp_{10}&(1,0)&g_1\\
\cline{3-4}\cline{6-7}
       &           &(1,1)&{\mf h}^-:=\frac{1}{2}(g_1+g_2)-\frac{{\mf i}}{2}(1-g_1g_2)& &(1, 1)&g_1g_2\\
\hline
(1,0,0)&\alp_{01}&(0,1)&g_2&\alp_{01}&(1,1)&g_1g_2\\
\hline
(1,0,1)&\alp_{01}&(0,1)&{\mf g}^+&\alp_{11}&(0,0)&g_1\\
\cline{3-4}\cline{6-7}
       &         &(1,1)&{\mf g}^-&         &(1,1)&g_2\\
\hline
(0,1,0)&\alp_{10}&(1,0)&g_1&\alp_{10}&(1,1)&g_1g_2\\
\hline
(1,1,0)&\alp_{11}&(0,0)&{\mf h}^+&\alp_{11}&(1,1)&{\mf h}^-\\
\hline
\end{array}~.
\]
\end{example}
\begin{proof}
The $2$-dimensional non-trivial Hopf algebras in ${}_{k_{\Phi_{\un{a}}}[G]}^{k_{\Phi_{\un{a}}}[G]}{\cal YD}$ 
are parametrized by $f_1, f_2, \l_1, \l_2\in \{0, 1\}$ such that 
$
4|2(\l_1f_1+\l_2f_2)+c_1f_1^2+c_2f_2^2+c_{12}f_1f_2-2
$. Equivalently, $2|c_1f_1^2+c_2f_2^2+c_{12}f_1f_2$ and $\l_1f_1+\l_2f_2 + \frac{c_1f_1^2+c_2f_2^2+c_{12}f_1f_2}{2}$ is an odd integer. 
But $c_1f_1^2, c_2f_2^2, c_{12}f_1f_2\in \{0, 1\}$, so either all are equal to zero or two of them are equal to one and the third one is equal to zero; 
in the former case $\l_1f_1+\l_2f_2$ is odd (this means that $\l_1f_1=1$ and $\l_2f_2=0$ or vice verse), while in the latter case 
$\l_1f_1+\l_2f_2$ is even (and so $\l_1f_1=\l_2f_2\in\{0, 1\}$).  

When $c_1f_1^2=c_2f_2^2=c_{12}f_1f_2=0$ and $\l_1f_1+\l_2f_2$ is odd, we cannot have $f_1, f_2$ (resp. $\l_1, \l_2$) simultaneously equal to zero. 
By considering the cases $(f_1, f_2)$ equal to $(0, 1)$, $(1, 0)$ and respectively $(1, 1)$, we obtain the solutions exposed in the table for  
$\un{a}\in \{(0, 0, 0), (0, 0, 1), (1, 0, 0), (0, 1, 0)\}$, as well as those mentioned for $\un{a}=(1, 0, 1)$ and $\un{a}=(0, 1, 1)$ in 
the case when $\alp_{f_1f_2}$ is $\alp_{01}$ and $\alp_{10}$, respectively. In a similar manner we treat the cases: 
(i) $c_1f_1^2=0$, $c_2f_2^2=c_{12}f_1f_2=1$ and $\l_1f_1+\l_2f_2$ is even; (ii) $c_1f_1^2=c_2f_2^2=1$, $c_{12}f_1f_2=0$ and $\l_1f_1+\l_2f_2$ is even; 
(iii) $c_1f_1^2=c_{12}f_1f_2=1$, $c_2f_2^2=0$ and $\l_1f_1+\l_2f_2$ is even. They produce the remaining six parameterizations in the table, we leave to the 
reader to verify this fact. Note only that $v$ can be computed in each case by using the formulas
\[
1_{0,0}=\frac{1}{4}(1+g_1)(1+g_2);~1_{1,0}=\frac{1}{4}(1-g_1)(1+g_2);~1_{0,1}=\frac{1}{4}(1+g_1)(1-g_2);~1_{1,1}=\frac{1}{4}(1-g_1)(1-g_2).
\]  
So our proof ends.
\end{proof}

It remains to study the case when $n=1$, so that $G$ is a cyclic group of order $m\geq 2$; by $g$ we denote the generator of $G$. 
Towards this end, we fix a primitive root of unity ${\mf q}$ of order $m^2$ in $k$ and take $q={\mf q}^m$, a primitive root of unity of order $m$ in $k$. Then, 
for $0\leq c\leq m-1$, $k_{\Phi_c}[G]$ is the Hopf group algebra $k[G]$, considered as a quasi-Hopf algebra via the reassociator  
\begin{equation}
    \Phi_c=\sum_{i,j,l=0}^{m-1}  q^{ci\lfloor \frac{j+l}{m}\rfloor}   1_{i} \ot 1_{j}\ot 1_{l}. 
\end{equation} 
As before, $1_j=\frac{1}{m}\sum_{s=1}^{m-1}q^{-js}g^s$, for all $0\leq j\leq m-1$. 

Again, by using the same arguments as in the proof of \prref{chara abel} we get the following result.

\begin{proposition}
If there exists a non-trivial $2$-dimensional Hopf algebra within 
${}_{k_{\Phi_{c}}[G]}^{k_{\Phi_{c}[G]}}{\cal YD}$ then the ${\rm GCD}$ of $c$ and $m$ is even. 

Assuming $(c, m)$ even, a non-trivial Hopf algebra of dimension $2$ in ${}_{k_{\Phi_{c}}[G]}^{k_{\Phi_{c}}[G]}{\cal YD}$ is of the form $B_{\alp, v}$, 
with $\alp: k[G]\ra k$ an algebra morphism and $v\in k[G]$, determined by $f, \l\in \{0, \cdots, m-1\}$ satisfying  
\begin{equation}\eqlabel{toughcyclic}
\frac{\l f}{m} + \frac{cf^2}{m^2} - \frac{1}{2}\in \mathbb{Z},
\end{equation}
as follows: $\alp(g^j)=q^{fj}$, for all $g^j\in G$, and 
$
v=\sum_{j=1}^m{\mf q}^{m\l j + cfj}1_{j}
$. 
\end{proposition}
\begin{proof}
If there is a non-trivial Hopf algebra of dimension $2$ in 
${}_{k_{\Phi_{c}}[G]}^{k_{\Phi_{c}}[G]}{\cal YD}$ then there exist $f, \l\in \{0,\cdots, m-1\}$ and $k\in \mathbb{Z}$ 
such that $2cf^2+2m\l f -(2k+1)m^2=0$. It follows that $m$ is even, so we only need to prove that $c$ is even, too. For $c=0$ there is nothing 
to prove, while for $c\not=0$ the equality $(2cf+m\l)^2=2c(2k+1)m^2+\l^2m^2$ says that $2c(2k+1)+\l^2$ is a square, let us say it is equal to 
$\omega^2$. Then $\l$ and $\omega$ have the same parity, and this fact tells us that $c$ must be even. 
\end{proof}

\begin{corollary}
Keeping the above notation, if $(c, m)$ is even there is always a non-trivial Hopf algebra of dimension $2$ in 
${}_{k_{\Phi_{c}}[G]}^{k_{\Phi_{c}}[G]}{\cal YD}$.
\end{corollary}
\begin{proof}
Let $d=(c, m)$, a non-zero even number, and consider $c', m'$ coprime integers such that $c=dc'$ and $m=dm'$. 
We saw that \equref{toughcyclic} is equivalent to 
the existence of $k\in \mathbb{Z}$ such that $(2cf+m\l)^2=2c(2k+1)m^2+\l^2m^2$. It follows that $m|2cf+m\l$, so $\frac{m}{2}|cf$; if we write 
$2cf=mt$, for some $t\in \mathbb{Z}$, then $(t+\l)^2=2c(2k+1)+\l^2$, hence $t^2+2\l t=2c(2k+1)$. Therefore, $t$ is even and we can write 
$cf=mv$, for some integer $v$. Furthermore, as $c'f=m'v$, we find $s\in \mathbb{Z}$ such that $f=sm'$ and $v=sc'$. 
The proof ends if we show that we can always find $s, \l, k$ such that $2s(c's+\l)=(2k+1)d$. Indeed, 
one can take $s=\frac{d}{2}$ and $\l, k$ such that $\l=2k+1-\frac{c}{2}$; note that, there is at least one $k\in \mathbb{Z}$ such that 
$\frac{c}{2}\leq 2k+1<m+\frac{c}{2}$, because $m\geq 2$. 
\end{proof}

\begin{remark}
In the Hopf algebra case, this means when $c=0$, non-trivial Hopf algebras of dimension $2$ in 
${}_{k[G]}^{k[G]}{\cal YD}$ are determined by pairs $(f, \l)$ with $\l f=(2k+1)\frac{m}{2}$, for some $k\in \mathbb{Z}$.
\end{remark}

\subsection{A non-abelian example} 
We want to apply \leref{1d fun yd} to a finite non-abelian group $G$. There are two obstacles we must avoid; namely, we need the center of $G$ to not be trivial 
and, on the other hand, to be able to find explicitly the representatives for the $3$-cocycles of $G$. In this direction, a first candidate 
is the dihedral group $D_n$ with $2n$ elements, $n$ an even natural number (for $n$ odd, $D_n$ is centerless); but, although we have a full description 
for $H^3(D_n, k^*)$ (see \cite{ham, han}) we don't have the concrete formulas for the representatives of the $3$-cocycles of $D_n$. Nevertheless, we can overcome this issue by working with the so called double dihedral group $\bar D_n$ instead of $D_n$. We take $n\geq 2$, for otherwise $\bar D_4\simeq C_4$, an abelian group.     

$\bar D_n$ is a group of order $4n$, generated by $R$ and $X$  with relations $R^{2n}=e$, $X^2=R^n$ and $XR=R^{-1}X$, where $e$ is the neutral element. The elements 
of $\bar D_n$ are identified with pairs $(A, a)$, in the sense that $(A, a)\equiv X^AR^a$; here $A\in \{0, 1\}$ and $a\in \{-n+1,\cdots, n\}$. Thus, $a'_{2n}$, 
the reminder of the division of $a$ by $2n$ is replaced by $\le a\ri_{2n}$, a new reminder modulo $2n$ which lies in the range $-n+1, \cdots, n$. 
Note that $\le a\ri_{2n}$ is uniquely determined by the property that $R^a=R^{\le a\ri_{2n}}$ and $\le a\ri_{2n}\in \{-n+1, \cdots, n\}$, for all $a\in \mathbb{Z}$. 
To find a connection between these two reminders, observe that 
$R^a=R^{a-2\lfloor \frac{a}{2n}\rfloor n}$ with $a-2\lfloor \frac{a}{2n}\rfloor n\in \{0, \cdots, 2n-1\}$, provided that $a\in \{-n+1, \cdots, n\}$. 
Likewise, $R^a=R^{a+2\lfloor \frac{n-a}{2n}\rfloor n}$ with $a+2\lfloor \frac{n-a}{2n}\rfloor n\in \{-n+1, \cdots, n\}$, provided that 
$a\in \{0, \cdots, 2n-1\}$. All these facts say that, for all $a\in \mathbb{Z}$,  
\begin{equation}\eqlabel{connrems}
\le a\ri_{2n}:=a'_{2n}+2\lfloor \frac{n-a'_{2n}}{2n}\rfloor n=a+2\lfloor \frac{n-a}{2n}\rfloor n.
\end{equation}
Indeed, $R^{2n}=e$ implies $R^a=R^{\le a\ri_{2n}}$, and since $\le a\ri_{2n}$ can also be written in the form $\le a\ri_{2n}=n-(n-a)'_{2n}$ 
it follows that $\le a\ri_{2n}\in \{-n+1, \cdots, n\}$.
 
\begin{lemma}\lelabel{centerbarD}
    For any $n\in \mathbb N$, the center $Z(\bar D_n)$ of the double dihedral group $\bar D_n$ equals $\{{e, R^n}\}$.
\end{lemma}
\begin{proof}
For $a\in \{0, \cdots, 2n-1\}$ we have that $RXR^a=XR^{a-1}$ and $XR^aR=XR^{a+1}$, so $XR^a$ does not belong to $Z(\bar D_n)$, unless $n=1$. On the other hand, 
$XR^a=R^{-a}X$, so $R^a$ is in the center of $\bar D_n$ if and only if $R^{2a}=e$, if and only if $a\in \{0, n\}$.   
\end{proof}
 
Owing to \cite{wp},  $H^3(\bar D_n,k^*)\simeq\mathbb{Z}_{4n}$. But, only half of the representatives for the cohomology classes of 
$H^3(\bar D_n,k^*)$ can be defined in an explicit way. To this end, assume that $k$ contains a $2n^2$ primitive root of unity in $k$, $\xi_{2n^2}$ (so 
$k$ is of characteristic different from $2$). Denote $\xi_n=\xi^{2n}$, $\xi_{2n}=\xi^{n}$ and $\xi_{4}=\xi^{\frac{n^2}{2}}$ (all are primitive roots),  
and for $p\in \{0, \cdots, 2n-1\}$ define $\omega_p: \bar D_n\times \bar D_n\times \bar D_n\ra k^*$ by 
\begin{eqnarray}
    \omega_p(X^AR^a, X^BR^b, X^CR^c)&=&\xi_{2n^2}^{p\left((-1)^{B+C}a((-1)^Cb+c -\le (-1)^Cb+c+nBC\ri_{2n})
		-\frac{n^2}{2}ABC\right)}\nonumber\\
		&\equal{\equref{connrems}}&\xi_{2n^2}^{p\left((-1)^{B+C+1}a(2\lfloor \frac{n-(-1)^Cb-c-nBC}{2n}\rfloor n + nBC)
		-\frac{n^2}{2}ABC\right)}\nonumber\\
		&=&\xi_{2n }^{p(-1)^{B+C+1}aBC}
		\xi_{n}^{p(-1)^{B+C+1}a\lfloor  \frac{n-(-1)^Cb - c - nBC}{2n} \rfloor} \xi_{4}^{-pABC},\eqlabel{dodih trico}
\end{eqnarray} 
for any couples $(A, a), (B, b), (C, c)$ with $A, B, C\in \{0, 1\}$ and $a, b, c\in \{-n+1,\cdots, n\}$. Then $(\omega_p)_{0\leq p\leq 2n-1}$ 
are non-cohomologous normalized $3$-cocycles for $\bar D_n$. 

\begin{lemma}\lelabel{cob2cocycleddg}
The normalized $2$-cocycle $\flat_{R^n}\omega_p$ on $\bar D_n$ is determined by 
\[
\flat_{R^n}\omega_p(X^AR^a, X^BR^b)=(-1)^{pAB}\xi_n^{\frac{pa}{2}((-1)^B-1)},
\] 
for all $A, B\in \{0, 1\}$ and $a, b\in \{-n+1, \cdots, n\}$. Furthermore, if $f: \bar D_n\ra k^*$ is given by 
$f(X^AR^a)=(-1)^A\xi_{2n}^{-pa}$, for all $A\in\{0, 1\}$ and $a\in \{-n+1, \cdots, n\}$, then $\flat_{R^n}\omega_p=\partial f$; therefore, 
$\flat_{R^n}\omega_p$ is a coboundary $2$-cocycle. 
\end{lemma} 
\begin{proof}
We compute, for all $(A, a)$ and $(B, b)$ as in the statement of the lemma, that 
\begin{eqnarray*}
\flat_{R^n}\omega_p(X^AR^a, X^BR^b)&=&\frac{\omega_p(R^n, X^AR^a, X^BR^b)\omega_p(X^AR^a, X^BR^b, R^n)}{\omega_p(X^AR^a, R^n, X^BR^b)}\\
&=&\frac{\xi_{2n }^{p(-1)^{A+B+1}nAB}\xi_{n}^{p(-1)^{A+B+1}n\lfloor  \frac{n-(-1)^Ba - b - nAB}{2n} \rfloor}
		\xi_{n}^{p(-1)^{B+1}a\lfloor  -\frac{b}{2n} \rfloor}}
		{\xi_{n}^{p(-1)^{B+1}a\lfloor  \frac{n-(-1)^Bn - b}{2n} \rfloor}}\\
		&=&(-1)^{pAB}\xi_n^{p(-1)^{B+1}a\left(\lfloor  -\frac{b}{2n} \rfloor - \lfloor  \frac{n-(-1)^Bn - b}{2n} \rfloor\right)}.
\end{eqnarray*}
Our assertion follows from the fact that 
\[
\lfloor  -\frac{b}{2n} \rfloor - \lfloor  \frac{n-(-1)^Bn - b}{2n} \rfloor=\left\{\begin{array}{c}
0,~\mbox{if $B=0$}\\
-1,~\mbox{if $B=1$}
\end{array}\right.=\frac{1}{2}((-1)^B-1).
\]
Finally, for $f:\bar D_n\ra k^*$ as in the statement of the lemma, we have that 
\begin{eqnarray*}
\partial f(X^AR^a, X^BR^b)&=&\frac{f(X^AR^a)f(X^BR^b)}{f(X^AR^aX^BR^b)}=\frac{(-1)^{A+B}\xi_{2n}^{-p(a+b)}}{f(X^{(A+B)'_2}R^{\le (-1)^Ba + b +nAB\ri_{2n}})}\\
&\equal{\equref{connrems}}&(-1)^{2\lfloor \frac{A+B}{2}\rfloor}\xi_{2n}^{-p(a-(-1)^Ba -nAB)}=(-1)^{pAB}\xi_n^{\frac{pa}{2}((-1)^B-1)}\\
&=&\flat_{R^n}\omega_p(X^AR^a, X^BR^b),
\end{eqnarray*}
for all $A, B\in \{0, 1\}$ and $a, b\in \{-n+1,\cdots, n\}$, as stated. So our proof is complete.
\end{proof}

\begin{proposition}\prlabel{double dihedral exe}
    For the quasi-Hopf algebra $(k^{\bar D_n},\Phi_{\omega_p})$, there are non-trivial $2$-dimensional braided Hopf algebras 
		in ${}_{(k^{\bar D_n},\Phi_{\omega_p})}^{(k^{\bar D_n},\Phi_{\omega_p})} {\cal YD}$ if and only if either $p$ is odd or $p, n$ have different parities.
\end{proposition}
\begin{proof}
    Owing to \leref{centerbarD} and \prref{fun chara}, giving a non-trivial $2$-dimensional braided Hopf algebra $B$  
		in ${}_{(k^{\bar D_n},\Phi_{\omega_p})}^{(k^{\bar D_n},\Phi_{\omega_p})} {\cal YD}$ is equivalent to giving a map $\rho: \bar D_n\ra k^*$ 
		such that $\rho(R^n)=-1$ and $\flat_{R^n}\omega_p=\partial\rho$. Moreover, as we have observed in the proof of \prref{chara abel}, any two 
		maps $\r_1, \r_2: \bar D_n\ra k^*$ satisfying $\partial\rho_1=\partial\rho_2$ are related through a group homomorphism 
		$\vartheta: \bar D_n\ra k^*$, in the sense that $\rho_1(g)=\rho_2(g)\vartheta(g)$, 
			for all $g\in G$. Therefore, by \leref{cob2cocycleddg}, $B$ is determined by a map $\rho: \bar D_n\ra k^*$ given by 
			$\rho(X^AR^a)=(-1)^{A}\xi_{2n}^{-pa}\vartheta(X^AR^a)$, for all $A\in\{0, 1\}$ and $a\in \{-n+1,\cdots, n\}$, where $\vartheta: \bar D_n\ra k^*$ 
			is a suitable group homomorphism.
			
			It is immediate that a group homomorphism $\vartheta: \bar D_n\ra k^*$ is parametrized by $(\upsilon, \xi_{2n}^s,)\in k^2$, $0\leq s\leq 2n-1$, 
			obeying $\upsilon^2=(-1)^s$ and $\upsilon\xi_{2n}^s=\upsilon\xi_{2n}^{-s}$. If follows that $s\in \{0, n\}$ and $\upsilon^2=(-1)^s$, 
			and therefore there are four group homomorphisms $\vartheta: \bar D_n\ra k^*$; namely, they are determined by the couples 
			$(\vartheta(X), \vartheta(R))\in\{(1, 1), (-1, 1), (\xi_4^n, -1), (-\xi_4^n, -1)\}$. 
		
		Summing up, there are exactly four maps $\rho: \bar D_n\ra k^*$ for which $\flat_{R^n}\omega_p=\partial\rho$. They are given by 
		$\rho_1: X^AR^a\mapsto (-1)^{A}\xi_{2n}^{-pa}$, $\rho_2: X^AR^a\mapsto \xi_{2n}^{-pa}$, $\rho_3: X^AR^a\mapsto (-\xi_4^n)^A(-1)^a\xi_{2n}^{-pa}$, and 
		respectively by $\rho_4: X^AR^a\mapsto \xi_4^{nA}(-1)^a\xi_{2n}^{-pa}$, for all $A\in\{0, 1\}$ and $a\in \{-n+1,\cdots, n\}$. Our assertion follows 
		now from the fact that $\rho_1(R^n)=\rho_2(R^n)=(-1)^p$ and $\rho_3(R^n)=\rho_4(R^n)=(-1)^{n+p}$. 
\end{proof}

 \begin{example}[Quaternion group]
     When $n=2$, $\bar D_2$ is the quaternion group $Q$, the group generated by $X$ and $R$ with relations $R^4=e$, $X^2=R^2$ and $XR=R^3X$ (the original presentation 
		of $Q$ can be obtained by taking $-1=R^2$, $\bar i = X$, $\bar j = R$ and $\bar k = XR$; indeed, we can see easily that 
		${\bar i} ^2 ={\bar j}^2 = {\bar k}^2 = \bar i \bar j \bar k =-1$).
          
    Owing to \prref{double dihedral exe}, there are non-trivial $2$-dimensional braided Hopf algebras 
		within ${}_{(k^{\bar D_2},\Phi_{\omega_p})}^{(k^{\bar D_2},\Phi_{\omega_p})} {\cal YD}$ if and only if $p\in \{1, 3\}$. If this is the case, we get $8$ 
		non-trivial $2$-dimensional braided Hopf algebras which give rise to $8$ quasi-Hopf algebras of dimension $16$, all of the form $H(\theta)_{v, \sigma}$; 
		recall that $v=\sum_{A, a}\rho(X^AR^a)P_{X^AR^a}$, and that $\sigma: \bar k^{D_2}\ra k$ is given by $\sigma(P_{X^AR^a})=\delta_{A, 0}\delta_{a, n}$, for all 
		$A\in\{0, 1\}$ and $a\in\{-n+1,\cdots, n\}$.
 \end{example}

\section{Some non-commutative, non-cocommutative examples}\selabel{noncommnoncocommexp}
\setcounter{equation}{0}
The classification of quasi-Hopf algebras with radical of codimension $2$ was obtained in \cite{eg3}. Namely, it was proved in \cite{eg3} 
that there are precisely four types of such quasi-Hopf algebras, denoted by $H(2)$, $H_+(8)$, $H_-(8)$ and $H(32)$. Note that $H_+(8)$, $H_-(8)$ and $H(32)$ 
contain $H(2)$ as a quasi-Hopf algebra, the latter being the Hopf group algebra of $k$ and $C_2=\le g\ri$, the cyclic group of size $2$, seen as a quasi-Hopf algebra 
via the reassociator $\Phi=1-2p_-\ot p_-\ot p_-$; as before, $p_\pm=\frac{1}{2}(1\pm g)$.

Throughout this section, $k$ is a field which contains a primitive $4$th root of unit $\mfq$. 
Our next aim is to construct 
new classes of quasi-Hopf algebras from an $H\in \{H(2), H_\pm(8), H(32)\}$ and a pair $(\sigma, v)$ of $H$ obeying \equref{2dim bialgebra cond} and 
$\sigma(h_2)h_1v=\sigma(h_1)vh_2$, for all $h\in H$; see \prref{qHa with a proj of rank 2}.   

By using the classification of the $4$-dimensional quasi-Hopf algebras, if follows that for $H(2)$ a pair $(\sigma, v)$ like the one above cannot exist. 
A direct argument for this assertion can be obtained by looking first at the elements $v=\mu 1 + \nu g\in H(2)$ for which $\va(v)=1$ and $\sigma(v)=-1$. As  
we cannot have $\sigma(g)=1$ (otherwise $\sigma=\va$, and so $1=-1$) we get $\sigma(g)=-1$; thus $\mu+\nu=1$ and $\mu-\nu=-1$, and hence $v=g$. Then 
$\sigma(p_-)=1$ and the last equality in \equref{2dim bialgebra cond} reads as $\Delta(g)=\sigma(x^3)x^1g\ot x^2g$ or, equivalently, as 
$g\ot g=g\ot g - 2p_-\ot p_-$; the latter equality is false.   

So, for $H=H(2)$, \prref{qHa with a proj of rank 2} only produces quasi-Hopf algebras of the form $k[C_2]\ot H(2)$, this means quasi-Hopf algebras build on 
the Hopf group algebra structure of $k[C_2\times C_2]$ with the help of the reassociator $\Phi$ of $H(2)$. We will see below that a similar situation also occurs 
if $H\in \{H_\pm(8), H(32)\}$. Recall from \cite{eg3} that $H(32)$ is the $k$-algebra generated by $g$, $x$ and $y$ with relations $g^2=1$, $gx=-xg$, $gy=-yg$, 
$x^4=y^4=0$ and $yx= \mfq xy$. The quasi-Hopf algebra structure of $H(32)$ is defined in such a way that $k[\le g\ri]$ is a quasi-Hopf subalgebra of $H(32)$, and by  
\begin{gather*}
    \Delta(x)=x\ot (p_+  + \mfq p_-) + 1\ot p_+ x + g\ot p_- x ,\quad 
    \Delta(y)=y\ot (p_+  - \mfq p_-) + 1\ot p_+ y + g\ot p_- y , \\
    S(x) = -x(p_+ + \mfq p_-),\quad S(y) = -y(p_+ - \mfq p_-),\quad
    \epsilon(x)=\epsilon(y)=0.
\end{gather*}

$H_\mfq(8)$ is the unital $k$-algebra generated by $g, x$ subject to relations $g^2=1$, $x^4=0$ and $gx=-xg$; the quasi-Hopf algebra structure of $H_\mfq(8)$ 
is defined similar to that of $H(32)$. In notation of \cite{eg3}, $H_+(8)$ and $H_-(8)$ correspond to the case when $\mfq=i$ and $\mfq=-i$, respectively. 

\begin{lemma}\lelabel{no tri reas}
   For $H\in \{H_\mfq(8), H(32)\}$, a $2$-dimensional braided Hopf algebra in ${}_H^H{\cal YD}$ is isomorphic to the Hopf group algebra $k[C_2]$, 
	regarded in ${}_H^H{\cal YD}$ via the trivial action and coaction. 
\end{lemma} 
    \begin{proof} 
		An algebra morphism $\sigma: H\ra k$ is determined by $\sigma(g)\in \{\pm 1\}$ and $\sigma(x)=0$, provided that $H=H_\mfq(8)$; in the case when 
		$H=H(32)$, in addition we must have $\sigma(y)=0$.
		
		Let $v\in H$ be such that $\va(v)=1$ and $\sigma(v)=-1$. We cannot have $\sigma(g)=1$; otherwise $\sigma=\va$, and so $1=-1$, a contradiction. 
We conclude that $\sigma(g)=-1$. Set $v=v_0 + v_1$ with $v_0\in k[\le g\ri]$ and $v_1=\sum_{j=0,l=1}^{1, 3}\l_{jl}g^jx^l$, for some $\l_{jl}$ in $k$   
(provided that $H=H_\mfq(8)$; when $H=H(32)$, $v_1=\sum_{j, l, t=0; l+t\not=0}^{1, 3, 3}\l_{jlt}g^jx^ly^t$ for some $\l_{jlt}$ in $k$). 
We have $\va(v_0)=1$ and $\sigma(v_0)=-1$, relations which force $v_0=g$; thus $v=g + v_1$. 

We now ask for $v$ that, in addition, $\sigma(h_2)h_1v=\sigma(h_1)vh_2$, for all $h\in H$. By taking $h=g$, we get that $vg=gv$, hence $gv_1=v_1g$, fact which gives us 
$v_1=(\sum_{j=0}^1\l_jg^j)x^2$, for some $\l_j\in k$, in the case when $H=H_\mfq(8)$, and respectively 
$v_1=\sum_{j, l, t=0; l+t\in\{2, 4\}}^{1, 3, 3}\l_{jlt}g^jx^ly^t$ in the case when $H=H(32)$. Likewise, by taking $h=x$, as $\sigma(p_+)=0$ and $\sigma(p_-)=1$ 
we get that $\mfq xv=vp_+x - vp_-x$, relation equivalent to $\mfq xv=vgx$. For $H=H_\mfq(8)$ the latest relation is equivalent to 
$-\mfq gx + \mfq (\sum_{j=0}^1(-1)^j\l_jg^{j})x^3=x + (\sum_{j=0}^1\l_jg^{j+1})x^3$, and for $H=H(32)$ it yields to 
\[
-\mfq gx +\mfq\sum\limits_{j, l, t=0; l+t\in \{2, 4\}}^{1, 3, 3}\l_{jlt}(-1)^jg^jx^{l+1}y^t=x + 
\sum\limits_{j, l, t=0; l+t\in\{2, 4\}}^{1, 3, 3}\l_{jlt}\mfq^tg^{j+1}x^{l+1}y^t.
\] 
In both cases we land to a contradiction, so our proof is finished.  	
\end{proof}

Hence, for $H=H_\mfq(8)$, our techniques give rise to a single quasi-Hopf algebra of dimension $16$, which will be denoted in what follows by 
$H_\mfq(16)$. $H_\mfq(16)$ equals $H_\mfq(8)_{g_2}=k[C_2]\otimes H_\mfq(8)$, see \prref{qHa with a proj of rank 2}. For short, $H_\mfq(16)$ contains 
$H_\mfq(8)$ as a quasi-Hopf subalgebra (now considered as being generated as an algebra by $g_1$ and $x$) and a second group like element $g_2$ which commutes with all 
the elements of $H_\mfq(8)$ and such that $g_1, g_2$, $x$ generate $H_\mfq(16)$ as an algebra. 

\begin{proposition}
For $H_\mfq(16)$, the only pair $(\sigma, v)$ satisfying \equref{2dim bialgebra cond} and $\sigma(h_2)h_1v=\sigma(h_1)vh_2$, for all $h\in H_\mfq(16)$, 
is given by the $k$-algebra morphism $\sigma: H_\mfq(16)\ra k$ determined by $\sigma(g_1)=1$, $\sigma(g_2)=-1$ and $\sigma(x)=0$, and the element $v=g_2$. 
Furthermore, $H_\mfq(8)(\theta)_{(\sigma, v)}$ is a $32$-dimensional quasi-Hopf algebra isomorphic to the tensor product quasi-Hopf algebra of $H_4$ and  
$H_\mfq(8)$. 
\end{proposition}
\begin{proof}
We write $v=\sum_{j, l=0}^{1, 1}\mu_{jl}g_1^jg_2^l + \sum_{j, l=0; t=1}^{1, 1; 3}\l_{jlt}g_1^jg_2^lx^t$, for some $\mu_{jl}, \l_{jlt}\in k$. Also, we denote 
$\sigma(g_j)=\nu_j\in \{\pm 1\}$, $0\leq j\leq 1$; as we remarked, we cannot have both $\nu_1$, $\nu_2$ equal to $1$ (otherwise $\sigma=\va$, and thus $1=-1$). 
Since $\va(v)=1$, we have $\sum_{j, l=0}^{1, 1}\mu_{ij}=1$; similarly, since $\sigma(v)=-1$ we have $\sum_{j, l=0}^{1, 1}\nu_1^j\nu_2^l\mu_{jl}=-1$.   

For $h=g_1$, the equality $\sigma(h_2)h_1v=\sigma(h_1)vh_2$ reads as $vg_1=g_1v$, and is equivalent to 
\[
\sum\limits_{j, l=0}^{1, 1}\mu_{jl}g_1^{j+1}g_2^l +\sum\limits_{j, l=0; t=1}^{1, 1; 3}(-1)^t\l_{jlt}g_1^{j+1}g_2^lx^t=
\sum\limits_{j, l=0}^{1, 1}\mu_{jl}g_1^jg_2^j +\sum\limits_{j, l=0; t=1}^{1, 1; 3}\l_{jlt}g_1^{j+1}g_2^lx^t.
\]
It follows that $\l_{jlt}=0$ for any $t$ odd, and therefore $v=\sum_{j, l=0}^{1, 1}\mu_{jl}g_1^jg_2^l +\left(\sum_{j, l=0}\l_{jl}g_1^jg_2^l\right)x^2$, for some 
$\mu_{jl}, \l_{jl}\in k$. With this formula of $v$, one can easily check that $\sigma(h_2)h_1v=\sigma(h_1)vh_2$ is satisfied by $h=g_2$, too. When we check it for 
$h=x$, we land at $\omega xv=v(p_+ + \nu_1p_-)x$, where $\omega:=\left(\frac{1}{2}(1+\l_1)+\frac{\mfq}{2}(1-\nu_1)\right)$. The last equality is equivalent to
\begin{eqnarray*}
&&\hspace{-1cm}
\omega\big(\sum\limits_{j, l=0}^{1, 1}(-1)^j\mu_{jl}g_1^jg_2^l\big)x + \omega\big(\sum\limits_{j, l=0}^{1, 1}(-1)^j\l_{jl}g_1^jg_2^l\big)x^3\\
&&\hspace{1cm}=
\big(\sum\limits_{j, l=0}^{1, 1}\mu_{jl}(p_+g_2^l+(-1)^j\nu_1p_-g_2^l)\big)x + \big(\sum\limits_{j, l=0}^{1, 1}\l_{jl}(p_+g^2_l+(-1)^j\nu_1p_-g_2^l)\big)x^3,
\end{eqnarray*}
where now $p_\pm=\frac{1}{2}(1\pm g_1)$. The above relation leads to ($0\leq j\leq 1$) 
\[\left\{
\begin{array}{l}
\omega (\mu_{0j}1-\mu_{1j}g_1)=\mu_{0j}(p_+ + \nu_1p_-) + \mu_{1j}(p_+ - \nu_1p_-),\\
\omega(\l_{0j}1-\l_{1j}g_1)=\l_{0j}(p_+ + \nu_1p_-) + \l_{1j}(p_+ - \nu_1p_-)
\end{array}\right.\Leftrightarrow
\left\{
\begin{array}{l}
\omega(\mu_{0j}-\mu_{1j})=\mu_{0j}+\mu_{1j},\\
\omega(\mu_{0j}+\mu_{1j})=\nu_1(\mu_{0j}-\mu_{1j}),\\
\omega(\l_{0j}-\l_{1j})=\l_{0j}+\l_{1j},\\
\omega(\l_{0j}+\l_{1j})=\nu_1(\l_{0j}-\l_{1j}).
\end{array}\right.
\]
\un{Case 1}: $\nu_1=1$. We get $\omega=1$, $\nu_2=-1$ and $\mu_{1j}=\l_{1j}=0$, for all $0\leq j\leq 1$. As in the above, we conclude that 
$v=g_2 + (\l 1 +\mu g_2)x^2$, for some $\l, \mu\in k$. On the other hand, the last condition on $v$ required in \equref{2dim bialgebra cond} turns out as 
$\Delta(v)=v\ot v$, because $\sigma(p_-)=0$ (we recall that $p_-=\frac{1}{2}(1-g_1)$ and that $\sigma(g_1)=\nu_1=1$). A simple computation yields 
\[
\Delta(x^2)=x^2\ot g_1 + (1+\mfq)x\ot p_+x + (1-\mfq)g_1x\ot p_-x + g_1\ot x^2,
\]
and this helps to show that $\Delta(v)=v\ot v$ if and only if $\l=\mu=0$. 

Summing up, $v=g_2$ and $\sigma$ is determined by $\sigma(g_1)=1$, $\sigma(g_2)=-1$ and $\sigma(x)=0$. Furthermore, $H_\mfq(8)(\theta)_{(\sigma, v)}$ is the 
quasi-Hopf algebra generated as a $k$-algebra by $g_1, g_2, x, \theta$ subject to relations $g_1^2=g_2^2=1$, $x^2=\theta^2=0$, $g_1g_2=g_2g_1$, 
$xg_1=-g_1x$, $g_2x=xg_2$, $g_1\theta=\theta g_1$, $g_2\theta=-\theta g_2$ and $x\theta=\theta p_+x+\theta p_-x=\theta x$. The quasi-coalgebra 
structure is given such that $g_1, g_2$ are grouplike elements, $\Delta(x)=x\ot (p_+  + \mfq p_-) + 1\ot p_+ x + g_1\ot p_- x$, 
$\Delta(\theta)=g_2\ot\theta + \theta\ot 1$, $\va(x)=\va(\theta)=0$ and reassociator $\Phi=1-2p_-\ot p_-\ot p_-$; the antipode is defined in a manner similar 
to the one for the antipode of $H_\mfq(8)$. A simple inspection shows us that the $k$-algebra generated by $g_2, \theta$ is a Hopf subalgebra of 
$H_\mfq(16)(\theta)_{(\sigma, v)}$ equals to $H_4$, and that $H_\mfq(16)(\theta)_{(\sigma, v)}=H_4\ot H_\mfq(8)$ as quasi-Hopf algebras. 

\un{Case 2}: $\nu_1=-1$. It provides us $\omega=\mfq$, $\mu_{1j}=\mfq \mu_{0j}$ and $\l_{1j}=\mfq \l_{0j}$, for all $0\leq j\leq 1$. Thus 
$v=\sum_{j=0}^1\mu_{0j}(1+\mfq g_1)g_2^j + \big(\sum_{j=0}^2\l_{0j}(1+\mfq g_1)g_2^j\big)x^2$. Now, $\va(v)=1$ implies 
$\mu_{00}+\mu_{01}=\frac{1}{2}(1-\mfq)$, while $\sigma(v)=-1$ implies $\mu_{01} +\nu_2 \mu_{01}=-\frac{1}{2}(1+\mfq)$. Hence, 
$(1-\nu_2)\mu_{01}=1$, and therefore $\nu_2=-1$, $\mu_{00}=-\frac{\mfq}{2}$ and $\mu_{01}=\frac{1}{2}$. We conclude that 
\[
v=-\frac{\mfq}{2}(1+\mfq g_1) +\frac{1}{2}(1+\mfq g_1)g_2 +\big(\sum\limits_{j=0}^1\tau_j(1+\mfq g_1)g_2^j\big)x^2,
\]    
for some $\tau_1, \tau_2\in k$. Finally, we have $\sigma(p_-)=1$ and a simple computation shows that $p_-v=vp_-$. Hence, the last equality 
in \equref{2dim bialgebra cond} reduces to $\Delta(v)=v\ot v -2p_-v\ot p_-v=p_+v\ot v + p_-v\ot g_1v$. With $v$ as in the above, $1\ot 1$ appears in the formula of 
$\Delta(v)$ with coefficient $-\frac{\mfq}{2}$. On the other hand, since 
\begin{eqnarray*}
&&p_+v=\frac{1}{2}(1-\mfq)p_+ +\frac{1}{2}(1+\mfq)p_+g_2 + (1+\mfq)p_+\big(\sum\limits_{j=0}^1\tau_jg_2^j\big)x^2,\\
&&p_-v=-\frac{1}{2}(1+\mfq)p_- +\frac{1}{2}(1-\mfq)p_-g_2 + (1-\mfq)p_-\big(\sum\limits_{j=0}^1\tau_jg_2^j\big)x^2
~\mbox{and}\\
&&g_1v=\frac{1}{2}(1-\mfq g_1) + \frac{\mfq}{2}(1 -\mfq g_1)g_2 + \mfq (1 -\mfq g_1)\big(\sum\limits_{j=0}^1\tau_jg_2^j\big)x^2  
\end{eqnarray*}  
it follows that $1\ot 1$ appears in the formula of $p_+v\ot v + p_-\ot g_1v$ with coefficient 
$-\frac{\mfq}{8}(1-\mfq)-\frac{1}{8}(1+\mfq)=-\frac{1}{4}(1+\mfq)$. We conclude that $\mfq =1$, so this case cannot occur. 
\end{proof}

We end this section by pointing out that only in the case when $H$ is a twist deformation of a Hopf algebra we can find non-trivial $2$-dimensional Hopf algebras 
within ${}_H^H{\cal YD}$, provided that $H$ is with radical of codimension $2$. To this end, recall that, for $n$ a non-zero natural number, the Nichols Hopf 
algebra $H_{2^n}$ is the $k$-algebra generated by $g, x_1, \cdots, x_n$ with relations $g^2=1$, $x_i^2=0$, $gx_i=-x_i g$, for all $1\leq i\leq n-1$, 
and $x_i x_j = -x_jx_i$, for all $1\leq i\neq j\leq n-1$. The Hopf algebra structure of $H_{2^n}$ is designed in such a way that $g$ is a grouplike element 
and $x_i$ is $(g, 1)$-skew primitive, for all $1\leq i\leq n$.  

    \begin{proposition}
        Let $H$ be a finite dimensional quasi-Hopf algebra with radical of codimension $2$. There are non-trivial $2$-dimensional braided Hopf algebras 
				in ${}_H^H{\cal YD}$ if and only if $H$ is twist equivalent to 
				a Nichols Hopf algebra $H_{2^n}$, for some $n\geq 1$.  
    \end{proposition}
    \begin{proof}    
		For the direct implication, by \leref{no tri reas} and the comments made at the beginning of this section, $H$ is not twist equivalent to one 
		of the quasi-Hopf algebras $H(2)$, $H_{\mf q}(8)$ or $H(32)$. Otherwise stated, owing to \cite[Theorem 3.4]{eg3}, the reassociator of $H$ in trivial 
		on $1$-dimensional $H$-modules; by \cite[Theorem 3.1]{eg3} it follows now that  
		$H$ is twist equivalent to $H_{2^n}$, for some $n\geq 1$. 
		
		For the converse, by \cite{bn} it suffices to show that there are non-trivial $2$-dimensional Hopf algebras within ${}_{H_{2^n}}^{H_{2^n}}{\cal YD}$. 
		Indeed, \thref{2dimbraidedHA} says that $2$-dimensional braided Hopf algebras over $H_{2^n}$ are parametrized by couples $(\sigma, v)$ 
		consisting of an algebra morphism $\sigma: H_{2^n}\rightarrow k$ and a grouplike element $v\in H_{2^n}$ such that $\sigma(v)=-1$ and 
		$\sigma(h_2)h_1 v = \sigma(h_1)v h_2$, for all $h\in H_{2^n}$. Thus, $v=g$ and $\sigma$ must obey $\sigma(g)=-1$ (the second condition is a tautology, as 
		it reduces to $x_ig=-gx_i$, for all $1\leq i\leq n$). Concluding, $(\sigma, g)$ provides a non-trivial braided Hopf algebra of dimension $2$ 
		in ${}_{H_{2^n}}^{H_{2^n}}{\cal YD}$, as needed. Furthermore, the resulting Hopf algebra $H_{2^n}(\theta)_{\sigma, g}$ identifies with 
		$H_{2^{n+1}}$, so the $H_{2^n}$'s can be viewed as a tower of biproduct Hopf algebras.  
    \end{proof}
\section{Appendix}\selabel{Appendix}
\setcounter{equation}{0}
Assume that we are in the context and notation of of \coref{nontrvialneq2}. In addition, we write 
$\frac{f_1}{m_1}=\frac{x_0}{L}$ and $\frac{f_2}{m_2}=\frac{y_0}{L}$, for some integers $x_0, y_0$, where 
$L:={\rm LCM}\{m_1, m_2\}$. For ${\mf \l}=(\l_1, \l_2)\in \mathbb{Z}^2$ and $k\in \mathbb{Z}$, define 
${\cal C}_{{\mf \l}, k}: \mathbb{Z}\times \mathbb{Z}\ra \mathbb{Z}$ by 
\begin{equation}\eqlabel{conicsect}
{\cal C}_{{\mf \l}, k}(x, y)=2c_1x^2 + 2c_2y^2 + 2c_{12}xy + 2\l_1 Lx + 2\l_2 Ly - (2k+1)L^2,~\forall~(x, y)\in \mathbb{Z}^2.
\end{equation}  
It is immediate that \equref{tough2} reduces to the finding of ${\mf \l}, k$ for which the conic section ${\cal C}_{{\mf \l}, k}(x, y)=0$ has 
rational points in $\mathbb{Z}^2$; this means, to the study of the existence of integer solutions for the equation ${\cal C}_{{\mf \l}, k}(x, y)=0$, a Gauss type equation. In the following, we will indicate a way to determine the solutions of this equation.
\subsection{Rational points}\sselabel{crtpoint}
The rational points ${\cal C}_{{\mf \l}, k}(x, y)=0$ can be determined by reducing ${\cal C}_{{\mf \l}, k}$ to its standard form. More exactly, 
we must deal with the following situations.

$\un{{\rm Case~ I}}$: $c_1\not=0$. We rewrite ${\cal C}_{{\mf \l}, k}(x, y)$ as 
\[
{\cal C}_{{\mf \l}, k}(x, y)=
2c_1\left(x+\frac{c_{12}}{2c_1}y+\frac{\l_1}{2c_1}L\right)^2 + \frac{4c_1c_2-c_{12}^2}{2c_1}y^2 + \frac{2\l_2c_1 -\l_1c_{12}}{c_1}Ly - 
\left(\frac{\l_1^2}{2c_1} + 2k+1\right)L^2.
\]  
Let $\delta:=4c_1c_2-c_{12}^2$ be the discriminant of ${\cal C}_{{\mf \l}, k}$.

$\un{{\rm Case}~ I_1}$: $\delta=0$. We reduce to the solving in $\mathbb{Z}$ of the system 
\begin{equation}\eqlabel{syst1n2}
\left\{\begin{array}{l}
2c_1x+c_{12}y+\l_1L=\widetilde{x}\\
\widetilde{x}^2+ 2(2\l_2c_1-\l_1c_{12})Ly-(\l_1^2+2(2k+1)c_1)L^2=0.
\end{array}\right. 
\end{equation} 
The Diophantine equation $2c_1x+c_{12}y=\gamma$ has solutions in $\mathbb{Z}$ if and only if $(2c_1, c_{12})|\gamma$. If this is the case 
and $(x_0, y_0)$ is a particular solution of $2c_1x+c_{12}y=\gamma$ then the general solution of it is 
$(x, y)=(-c_{12}u + x_0, 2c_1u + y_0)$, $u\in\mathbb{Z}$. So, for $\widetilde{x}=\l_1L + (2c_1, c_{12})v$, we have to look for those integers 
$\l_1, \l_2, k$ for which the equation in $(u, v)\in \mathbb{Z}^2$,  
\[
(\l_1L + (2c_1, c_{12})v)^2 + 2(2\l_2c_1-\l_1c_{12})L(2c_1u+y_0)-(\l_1^2+2(2k+1)c_1)L^2=0
\]
has solutions. This equation is solved case by case, by finding first $u$ such that   
$(\l_1^2+2(2k+1)c_1)L^2-2(2\l_2c_1-\l_1c_{12})L(2c_1u+y_0)=\omega^2$ (a square), and then $v$ obeying $\l_1L + (2c_1, c_{12})v=\pm\omega$.  

$\un{{\rm Case}~I_2}$: $\delta\not=0$. We rewrite ${\cal C}_{{\mf \l}, k}(x, y)$ as 
\[
\frac{1}{2c_1}(2c_1x+c_{12}y+\l_1L)^2+\frac{1}{2c_1\d}(\d y+ (2\l_2c_1-\l_1c_{12})L)^2-\frac{1}{2c_1\d}d,
\]
where $d:=(\l_1^2\d +(2\l_2c_1-\l_1c_{12})^2 + 2(2k+1)c_1\d)L^2$. We have to solve in $\mathbb{Z}$ the equation 
\begin{equation}\eqlabel{eq2}
\widetilde{x}^2+\d\widetilde{y}^2=d,
\end{equation}
and then to find the integers $\l_1, \l_2, k$ for which the linear system
$\left\{\begin{array}{l}
\d y+ (2\l_2c_1-\l_1c_{12})L=\widetilde{x}\\
2c_1x+c_{12}y+\l_1L=\widetilde{y}
\end{array}\right.
$
has solutions $(x, y)$ in $\mathbb{Z}^2$. Actually, the values of $x, y$ are limited by $m_1, m_2$, respectively, 
so the values of $\widetilde{x}$, $\widetilde{y}$ are limited, too. We get in this way bounds for $d$, and therefore for $k$. 
Then for any possible $k$ we solve first \equref{eq2} and then the linear system.
Note also that the Pell-type equation \equref{eq2} has a finite number of solutions, provided that $\delta>0$ or 
$\delta<0$ and $-\delta$ is a square. When $\delta<0$ and square free, 
the equation \equref{eq2} has no solution or an infinite number of solutions (in the latter case there is an algorithm to find them, see \cite{kmd}).  

$\un{{\rm Case}~ II}$: $c_1=0$. We distinguish four possibilities.

$\un{{\rm Case}~II_1}$: $c_2\not=0$ and $c_{12}=0$. We write 
\[
{\cal C}_{{\mf \l}, k}(x, y)=\frac{1}{2c_2}(2c_2y +\l_2L)^2+2\l_1Lx-\frac{1}{2c_2}(\l_2^2 + 2(2k+1)c_2)L^2.
\]
Our problem reduces to the finding of those integers $\l_1, \l_2, k$ for which there exists $x, y, \omega\in \mathbb{Z}$ such that 
$(\l_2^2 + 2(2k+1)c_2)L^2-4\l_1c_2Lx=\omega^2$ and $2c_2y +\l_2L=\pm \omega$.

$\un{{\rm Case}~II_2}$: $c_2\not=0$ and $c_{12}\not=0$. We have ${\cal C}_{{\mf \l}, k}(x, y)$ equals    
\[
\frac{2}{c_{12}^2}(c_{12}y+\l_1L)(c_{12}^2x + c_2c_{12}y - (\l_1c_2-\l_2c_{12})L) + \left(\frac{2\l_1(\l_1c_2-\l_2c_{12})}{c_{12}^2}-2k-1\right)L^2.
\]
So, we have to find integers $\l_1,\l_2, k$ such that the elementary equation 
\[
2(c_{12}y+\l_1L)(c_{12}^2x + c_2c_{12}y - (\l_1c_2-\l_2c_{12})L)=((2k+1)c_{12}^2-2\l_1(\l_1c_2-\l_2c_{12}))L^2
\]
has solutions in $\mathbb{Z}$.

$\un{{\rm Case}~II_3}$: $c_2=0$ and $c_{12}\not=0$. We have that 
\[
{\cal C}_{{\mf \l}, k}(x, y)=\frac{2}{c_{12}}(c_{12}x + \l_2L)(c_{12}y + \l_1L)-\frac{1}{c_{12}}\left((2k+1)c_{12} + 2\l_1\l_2\right)L^2,
\] 
and so we must find integers $\l_1, \l_2, k$ for which the elementary equation 
\[
2(c_{12}x + \l_2L)(c_{12}y + \l_1L)=(2k+1)c_{12} + 2\l_1\l_2
\]
has solutions in $\mathbb{Z}$. Observe that there is no solution in this case, provided that $c_{12}$ is odd. 

$\un{{\rm Case}~II_4}$: $c_2=0$ and $c_{12}=0$. It is easy to see that in this case we reduce our problem to the solving of the 
Diophantine equation $2\l_1x +2\l_2y=(2k+1)L$, which has solutions in $\mathbb{Z}$ if and only if we choose $\l_1, \l_2, k$ such that 
$(\l_1, \l_2)|(2k+1)\frac{L}{2}$ (we have observed that $L$ must be even).  
\subsection{An example}\selabel{example}
We exemplify the method of finding rational points for a conic section as in the above, for the case when $m_1=2$ and $m_2=6$. 
Thus, $M=2$, $L=6$, $N=36$, $\l_1, c_1, c_{12}\in \{0, 1\}$, $\l_2, c_2\in \{0,\cdots,5\}$ and $\frac{f_1}{2}=\frac{x}{6}$, $f_2=y$, so that 
$x\in \{0, 3\}$ and $y\in \{0, \cdots, 5\}$.  
The desired braided Hopf algebra structures are then parametrized by $(x, y)\in \mathbb{Z}^2$ for which ${\cal C}_{{\mf \l}, k}(x, y)=0$ for some 
${\mf \l}$ and $k$, where  
\[
{\cal C}_{{\mf \l}, k}(x, y)=2c_1x^2+2c_2y^2 +2c_{12}xy+ 12\l_1x+12\l_2y-36(2k+1).
\] 
Note that $\delta=4c_1c_2-c_{12}^2$ is equal to zero if and only if $c_1c_2=c_{12}=0$. Parallel to the cases considered in 
\sseref{crtpoint}, we have the following possibilities. The notation below is the one we used in the subsection that we just mentioned. In addition, 
by $(x, y; \l_1, \l_2, k)$ we understand that $(x, y)$ is a rational point for ${\cal C}_{{\mf \l}, k}$, determined by ${\mf \l}=(\l_1, \l_2)$ and $k$. 

${\rm Case}~I_1$: $c_1=1$ and $\delta=0$, that is, $c_1=1$ and $c_2=c_{12}=0$. The Diophantine equation reduces to $2x+6\l_1=\widetilde{x}$, and forces 
$36(\l_1^2+4k+2)-24\l_2y$ to be a square. Thus, $3\mid \l_2y$ and $\l_1^2 + 4k+2-2\frac{\l_2y}{3}$ is a square. Observe that we cannot have $\l_2=0$, for 
otherwise $4k+2$ or $4k+3$ is a square, a contradiction (recall that $\l_1\in \{0, 1\}$). 

$\bullet$ If $\l_2\in \{1, 2, 4, 5\}$, $y=3$ and we must have $\l_1^2 -2\l_2 +4k+2$ a square. A square is congruent to $0$ or $1$ modulo $4$.  
If we consider our square congruent to $0$ modulo $4$ then  
$(\l_1, \l_2, k)\in \{(0, 1, t^2), (0, 5, t^2+2)\}$, $t\in \mathbb{Z}$. When $(\l_1, \l_2, k)=(0, 1, t^2)$ we get $x=\pm 6t$, $y=3$, $\l_1=0$, $\l_2=1$ 
and $k=t^2$, and we are forced to take $t=0$. Hence, $(0, 3; 0, 1, 0)$ is a rational point. Likewise, when 
$(\l_1, \l_2, k)=(0, 5, t^2+2)$ we deduce that $x=\pm 6t$, $y=3$; we get that $(0, 3; 0, 5, 2)$ is a rational point, too. Analogously, if the above mentioned 
square is congruent to $1$ modulo $4$ we get that $(\l_1, \l_2, k)\in \{(1, 1, t^2+t), (1, 5, t^2+t+1)\}$, $t\in \mathbb{Z}$; only the first triple yields a 
rational point, namely $(0, 3; 1, 1,0)$.

$\bullet$ If $\l_2=3$, $\l_1^2-2y+4k+2$ is a square; we proceed as in the previous case. We deduce that $(\l_1, y, k)\in \{(1, 1, t^2+t), (1, 3, t^2+t+1), 
(1, 5, t^2+t+2), (0, 1, t^2), (0, 3, t^2+1), (0, 5, t^2+2)|t\in \mathbb{Z}\}$. 
Since $\widetilde{x}^2=36(\l_1^2+4k+2)-72y$ and $\widetilde{x}=2x+6\l_1\in \{6, 12\}$ for $\l_1=1$ and $\widetilde{x}\in \{0, 6\}$ for 
$\l_1=0$, we conclude that the rational points are:  
$(0, 1; 1, 3, 0)$, $(0, 3; 1, 3, 1)$, $(0, 5; 1, 3, 2)$, $(0, 1; 0, 3, 0)$, $(0, 3; 0, 3, 1)$ and $(0, 5; 0, 3, 2)$.   

${\rm Case~I_2}$: $c_1=1$ and $\delta\not=0$; equivalently, $c_1=c_{12}=1$ or $c_1=1$, $c_{12}=0$ and $c_2\in\{1, \cdots, 5\}$. 

${\rm I}_{2_1}$: $c_1=c_{12}=1$. We have $\d<0$ if and only if $c_2=0$, and this is why we treat this case separately.

$\bullet$ For $c_2=0$, the Pell's type equation reads $\widetilde{x}^2-\widetilde{y}^2=d$, with $\widetilde{x}=-y+6(2\l_2-\l_1)$, $\widetilde{y}=2x+y+6\l_1$ 
and $d=36(4\l_2^2-4\l_1\l_2-4k-2)$. Having this particular shape, it can be solved directly by considering its equivalent form  
\[
(x+6\l_2)(x+y-6(\l_2-\l_1))=18(-2\l_2^2+2\l_1\l_2 +2k+1).
\]  
As we observed, $x\in \{0, 3\}$. If $x=0$, we reduce to $\l_2y=6k+3$. So $(\l_2, y)\in \{(3, 1), (1, 3)\}$, 
and therefore $(0, 3; \l_1, 1, 0)$ and $(0, 1; \l_1, 3, 0)$, $\l_1\in\{0, 1\}$, are rational points. Likewise, for $x=3$ we reduce to 
$(2\l_2+1)y+6\l_1=3(4k+1)$, which ensures us that $3|(1+2\l_2)y$. If $\l_2=0$, $y=3$ and $\l_1=2k$. If $\l_2=1$ then $y+2\l_1=4k+1$, thus 
$(y, \l_1, k)\in\{(1, 0, 0), (3, 1, 1), (5, 0, 1)\}$. If $\l_2\in\{2, 3\}$, $y=3$ and $\l_1+\l_2=2k$. If $\l_2=4$ then $3y+2\l_1=4k+1$, 
hence $(y, \l_1, k)=(1, 1, 1)$. Finally, if $\l_2=5$ then $y=3$ 
and $\l_1=2k-5$. We deduce that the rational points are 
$(3, 3; 0, 0, 0)$, $(3, 1; 0, 1, 0)$, $(3, 3; 1, 1, 1)$, $(3, 5; 0, 1, 1)$, $(3, 3; 0, 2, 1)$, $(3, 3; 1, 3, 2)$, $(3, 1; 1, 4, 1)$, $(3, 3; 1, 5, 3)$.

$\bullet$ Consider $c_2\in \{1,\cdots, 5\}$, so that $\d=4c_2-1\in \{3, 7, 11, 15, 19\}$. As $\l_1\in \{0, 1\}$ and $x\in\{0, 3\}$, this case can be 
treated easily by considering the four cases that we might have. More precisely, we notice first that, in terms of $x$ and $y$, our equation 
$\widetilde{x}^2+\d\widetilde{y}^2=d$ reads as 
\[
(\d+1)y^2 + 4(x+6\l_2)y + 4x(x+6\l_1) -72(2k+1)=0. 
\]
- If $\l_1=0$, for $x=0$ we get $c_2y^2 +6\l_2y - 18(2k+1)=0$ which provides solutions only when $c_2\in \{2, 4\}$: 
$(0, 3; 0, 0, 0)$, $(0, 3; 0, 2, 1)$, $(0, 3; 0, 4, 2)$ and, respectively, $(0, 3; 0, 1, 1)$, $(0, 3; 0, 3, 2)$, $(0, 3, 0, 5, 3)$. 
For $x=3$ the equation becomes $c_2y^2+3(2\l_2+1)y-9(4k+1)=0$, and admits solutions only when $c_2\in\{2, 4\}$; namely:  
$(3, 3; 0, 1, 1)$, $(3, 3; 0, 3, 2)$, $(3, 3; 0, 5, 3)$ and, respectively,  $(3, 3; 0, 0, 1)$, $(3, 3; 0, 2, 2)$, $(3, 3; 0, 4, 3)$.

- If $\l_1=1$, for $x=0$ we have $c_2y^2 + 6\l_2 y -18(2k+1)=0$, the same equation as in the case when $\l_1=x=0$; so, it gives the 
rational points $(0, 3; 1, 0, 0)$, $(0, 3; 1, 2, 1)$, $(0, 3; 1, 4, 2)$ when $c_2=2$, and  $(0, 3; 1, 1, 1)$, $(0, 3; 1, 3, 2)$, $(0, 3; 1, 5, 3)$ when $c_2=4$. 
Likewise, $x=3$ yields $c_2y^2 + 3(2\l_2 + 1)y - 9(4k-1)=0$, and we have solutions only when $c_2\in \{2, 4\}$: 
$(3, 3; 1, 0, 1)$, $(3, 3; 1, 2, 2)$, $(3, 3; 1, 4, 3)$ and, respectively,  $(3, 3; 1, 1, 2)$, $(3, 3; 1, 3, 3)$, $(3, 3; 1, 5, 4)$.

${\rm I}_{2_2}$: $c_1=1$, $c_{12}=0$ and $c_2\in\{1, \cdots, 5\}$. It is more convenient to turn $\widetilde{x}^2+\d \widetilde{y}^2=d$ in terms of $x, y$; we have that $c_2y^2 +6\l_2y + x^2 +6\l_1x-18(2k+1)=0$. For $x=0$, we obtain that $6|c_2y^2$ and $y\not=0$; we have $y=3$, otherwise $y$ is even, and so 
$y^2+ 2\l_2y$ is a multiple of $4$ and equal to $6(2k+1)$ which is false. Thus, for $x=0$, we must have $c_2\in\{2, 4\}$ and $y=3$, and this entails to 
$c_2+2(\l_2-2k-1)=0$; so, for $c_2=2$ we get the rational points $(0, 3; \l_1, 2t, t)$, $\l_1\in\{0, 1\}$ and $t\in\{0, 1, 2\}$, while for 
$c_2=4$ we obtain the rational points $(0, 3; \l_1, 2t+1, t+1)$, $\l_1\in\{0, 1\}$ and $t\in\{0, 1, 2\}$. 

Similarly, for $x=3$, we have $c_2y^2 + 6\l_2y + 18\l_1 - 9(4k+1)=0$, so $y\not=0$ and $3|c_2y^2$. It follows that either $c_2=3$ or 
$y=3$. One can see that in the first case, when $c_2=3$, 
we have the rational points $(3, 1; 0, 1, 0)$, $(3, 3; 0, 2t-1, t)$ with $t\in\{1, 2, 3\}$ and $(3, 5; 0, 5, 6)$, provided that $\l_1=0$, 
and respectively $(3, 1; 1, 4, 1)$, $(3, 3; 1, 2(t-1), t)$ with 
$t\in\{1, 2, 3\}$ and $(3, 5; 1, 2, 4)$, provided that $\l_1=1$. The second case, when $y=3$, leads to the following rational points: 
$(3, 3; \l_1, \l_2, k)$ with $\l_1+\l_2=2k+\frac{1-c_2}{2}$ (there are $6$ solutions for each possible value of $c_2$,  hence $18$ in total), more exactly:  
for $c_2=1$ they are $(3, 3; 0, 2t, t)$ and $(3, 3; 1, 2t+1, t+1)$ with $t\in\{0, 1, 2\}$; for $c_2=3$ they are given by  
$(3, 3; 0, 2t+1, t+1)$ and $(3, 3; 1, 2t, t)$ with $t\in\{0, 1, 2\}$; for $c_2=5$ the rational points are $(3, 3; 0, 2t, t+1)$ and $(3, 3; 1, 2t+1, t+2)$ 
with $t\in \{0, 1, 2\}$.       

${\rm Case~II}_1$: $c_2\in\{1, \cdots, 5\}$, $c_1=c_{12}=0$. We check when $9(\l_2^2 + 2(2k+1)c_2)-6\l_1c_2x$ is a square.

$\bullet$ When $\l_1x=0$, if $\l_2^2-\omega^2=-2(2k+1)c_2$, for some integer $\omega$, then $\l_2$ and $\omega$ have the same parity; this forces 
$c_2$ even, so $c_2\in \{2, 4\}$. Note that $\pm 3\omega=c_2y+3\l_2$. 
 
$-$ For $c_2=2$, since any square is $0$, $1$ or $4$ modulo $8$, $\l_2^2+4(2k+1)$ a square implies that $\l_2^2$ is $0$ or $4$ 
modulo $8$, so $\l_2\in\{0, 2, 4\}$. When $\l_2=0$, $2k+1$ is a square, hence $k=2t(t+1)$, $t\in \mathbb{Z}$; we get that $\omega=\pm 2(2t+1)$ and 
$y=\pm 3(2t+1)$, forcing $y=3$ and $t=0$. Concluding, we have $(x, 3; \l_1, 0, 0)$ as rational points, provided that $\l_1x=0$ (they are in number of $3$). 
Likewise, when $\l_2=4$, we have $2k+2=\left(\frac{\omega}{12}\right)^2$ and $2y+6=\pm\frac{\omega}{2}$, for some even integer $\omega$, so  
$k=2t^2-1$ and $y=\pm 6t -3$, for some $t\in \mathbb{Z}$, forcing $y=3$ and $k=1$; we get in this case the rational points $(x, 3; \l_1, 2, 1)$ with 
$\l_1x=0$ (again in number of $3$). Finally, for $\l_2=4$, 
$2k+5=\left(\frac{\omega}{12}\right)^2$ is a square and $y+6=\pm\frac{\omega}{4}$, and therefore 
$k=2(t^2+t-1)$ and $y+6=\pm 3(2t+1)$, for some $t\in \mathbb{Z}$; these facts entail to $y=3$ and $k=2$, and so we get 
another $3$ rational points: $(x, 3; \l_1, 4, 2)$, whenever $\l_1x=0$. 

$-$ For $c_2=4$, $\l_2^2+8(2k+1)$ a square implies that $\l_2\in\{1, 3, 5\}$, since a square equals $0$, $1$, $4$ or $9$ modulo $16$. 
For $\l_2=1$, $16k+9=\left(\frac{\omega}{6}\right)^2$ and $4y+3=\pm\frac{\omega}{2}$, for some even integer $\omega$, hence $k=4t^2+3t$ or 
$k=4t^2+5t+1$, for some $t\in \mathbb{Z}$; we get that $4y+3=\pm 3(8t+3)$ or $4y+3=\pm 3(8t+5)$, $t\in \mathbb{Z}$, thus $y=3$ and $k=1$, 
leading to the rational points $(x, 3; \l_1, 1, 1)$, where $\l_1x=0$ (there are $3$ in total). 
In the case when $\l_2=3$, $16(k+1)+1=\left(\frac{\omega}{6}\right)^2$ is a square and $4y+9=\pm\frac{\omega}{2}$, 
and so $k=4t^2+t-1$ or $k=4t^2+7t+2$, $t\in \mathbb{Z}$; we obtain that 
$4y+9=\pm 3(8t+1)$ or $4y+9=\pm 3(8t+7)$, for some $t\in \mathbb{Z}$, so $y=3$ and $k=2$, and we have the rational points $(x, 3; \l_1, 3, 2)$ whenever 
$\l_1x=0$. Similarly, when $\l_1=5$, $16(k+2)+1=\left(\frac{\omega}{6}\right)^2$ is a square 
and $4y+15=\pm\frac{\omega}{2}$, and we must have $k=4t^2+t-2$ or $k=4t^2+7t+1$, for some $t\in \mathbb{Z}$; we deduce that 
$4y+15=\pm 3(8t+1)$ or $4y+15=\pm 3(8t+7)$, hence $y=3$ and $k=3$, and we have the rational points 
$(x, 3; \l_1, 5, 3)$, whenever $\l_1x=0$.      

$\bullet$ Assume that $\l_1x\not=0$ or, equivalently, that $\l_1=1$ and $x=3$. We have $\l_2^2 + 4kc_2=\left(\frac{\omega}{6}\right)^2$ 
and $c_2y+3\l_2=\pm\frac{\omega}{2}$, for some even integer $\omega$.  
Clearly, $3|c_2y$ and we distinguish two possibilities. 

$-$ If $c_2=3$, $\l_2^2+12k=(y+\l_2)^2$, and therefore $y^2+2\l_2y -12k=0$. We get $y\in\{0, 2, 4\}$, and the rational points 
$(3, 0; 1, \l_2, 0)$, $0\leq \l_2\leq 5$, and $\{(3, 2; 1, 3t-1, t), (3, 4; 1, 3t-2, 2t)|t\in \{1, 2\}\}$.

$-$ If $c_2\not=3$, we have $y=3$ and $c_2+2\l_2=4k$. This entails to the following rational points: $(3, 3; 1, 2t-1, t)$ with $t\in\{1, 2, 3\}$, 
provided that $c_2=2$, and $(3, 3; 1, 2t-2, t)$ with $t\in \{1, 2, 3\}$, provided that $c_2=4$.

${\rm Case~II}_2$: $c_1=0$, $c_2\not=0$ and $c_{12}\not=0$. We must solve in $\mathbb{Z}$ the equation 
\[
(y+6\l_1)(x+c_2y-6(\l_1c_2-\l_2))=18(2k+1-2\l_1(\l_1c_2-\l_2)).
\]

$\bullet$ If $\l_1=0$, we reduce to $y(x+c_2y+6\l_2)=18(2k+1)$, which imposes the following discussion.

$-$ $y=1$: $x+c_2 + 6\l_2=18(2k+1)$, with the unique solution $x=c_2=3$, $\l_2=2$ and $k=0$; the associated rational point is $(3, 1; 0, 2, 0)$.

$-$ $y=2$: $x+2c_2+6\l_2=9(2k+1)$, which imposes $x=3$; thus $c_2+3\l_2=3(3k+1)$, leading to $c_2=3$ and $\l_2=3k$, and so to the 
rational points $(3, 2; 0, 3t, t)$, $t\in \{0, 1\}$.

$-$ $y=3$: $x+3c_2+6\l_2=6(2k+1)$, which for $x=0$ gives $c_2\in\{2, 4\}$ such that $c_2+2\l_2=2(2k+1)$, while for 
$x=3$ gives $c_2\in\{1, 3, 5\}$ such that $c_2+2\l_2=4k+1$; so we get the following rational points:  
$(0, 3; 0, 2k, k)$ with $k\in \{0, 1, 2\}$, provided that $c_2=2$;  
$(0, 3; 0, 2k-1, k)$ with $k\in\{1, 2, 3\}$, provided that $c_2=4$; 
$(3, 3; 0, 2k, k)$ with $k\in \{0, 1, 2\}$, provided that $c_2=1$; 
$(3, 3; 0, 2k-1, k)$ with $k\in \{1, 2, 3\}$, provided that $c_2=3$; $(3, 3; 0, 2k-2, k)$ with $k\in\{1, 2, 3\}$, provided that $c_2=5$.
 
$\bullet$ If $\l_1=1$, we proceed as in the previous case. The equation becomes $(y+6)(x+c_2y-6c_2+6\l_2)=18(2k+1-2c_2+2\l_2)$ and we cannot have $y=2$.

$-$ $y=0$: $x=3(2k+1)$; thus $x=3$, $0\not=c_2, \l_2\in\{0,\cdots,5\}$ and $k=0$, so we have 30 rational points of the form $(3,0; 1, \l_2, 0)$, no matter how we choose $c_2\in \{1,\cdots, 5\}$.

$-$ $y=1$: $c_2+7x+6\l_2=18(2k+1)$, hence $6|c_2+x$; we get $c_2=x=3$ and $\l_2=6k-1$, so the rational points are $(3,1; 1, 5, 1)$, provided that $c_2=3$.

$-$ $y=3$: $x+c_2+2\l_2=2(2k+1)$, which furnishes the rational points $(0, 3; 1, 2k, k)$ with $k\in\{0, 1, 2\}$, provided that $c_2=2$; 
$(0, 3; 1, 2k-1, k)$ with $k\in\{1, 2, 3\}$, provided that $c_2=4$; $(3, 3; 1, 2k-1, k)$ with $k\in\{1, 2, 3\}$, provided that $c_2=1$;
$(3, 3; 1, 2k-2, k)$ with $k\in\{1, 2, 3\}$, provided that $c_2=3$; $(3, 3; 1, 2k-3, k)$ with $k\in\{2, 3, 4\}$, provided that $c_2=5$;

$-$ $y=4$: $5(x-2c_2+6\l_2)=9(2k+1-2c_2+2\l_2)$, with general solution $x-2c_2+6\l_2=9t$ and $2k+1-2c_2+2\l_2=5t$, $t\in \mathbb{Z}$. It follows that 
$x-(2k+1)+4\l_2=4t$, so $x=3$, and since $x-3(2k+1)+4c_2=-6t$ it follows that $c_2=3$, too. We reduce to $t=k-2$ and $2\l_2=3k-5$, and so we get the following rational points: $(3, 4; 1, 3s-1, 2s+1)$ with $s\in\{1, 2\}$, provided that $c_2=3$.

$-$ $y=5$: $11(x-c_2+6\l_2)=18(2k+1-2c_2+2\l_2)$, and therefore $x-c_2+6\l_2=18t$ and $k+1-2c_2+2\l_2=11t$, $t\in \mathbb{Z}$. 
As $6|x-c_2$, we have $x=c_2=3$, which leads to $\l_2=3t$ and $2k-5=5t$. We get the rational point $(3, 5; 1, 3, 5)$, provided that $c_2=3$.

${\rm Case~II}_3$: $c_1=c_2=0$ and $c_{12}\not=0$. There is no rational point in this case since $c_{12}=1$ is odd.

${\rm Case~II}_4$: $c_1=c_2=c_{12}=0$. We must solve the Diophantine equation $\l_1x+\l_2y=3(2k+1)$. 

$-$ If $\l_1x=0$, $\l_2y=3(2k+1)$ forces $k\in \{0, 1, 2, 3\}$. We obtain the following $15$ rational points $(x, 1; \l_1, 3, 0)$, $(x, 3; \l_1, 1, 0)$, 
$(x, 3; \l_1, 3, 1)$, $(x, 3; \l_1, 5, 2)$ and $(x, 5; \l_1, 3, 2)$, whenever $\l_1x=0$.

$-$ If $\l_1x\not=0$, we have $\l_1=1$, $x=3$ and $\l_2y=6k$ with $k\in \{0, 1, 2\}$. When $k=0$, either $y=0$ or $\l_2=0$, giving $12$ rational points: 
$(3, 0; 1, \l_2, 0)$, $0\leq \l_2\leq 5$, and $(3, y; 1, 0, 0)$ with $0\leq y\leq 5$. For $k=1$, either $\l_2=3$ or $y=3$, 
and we get the rational points $(3, 2; 1, 3, 1)$ and $(3, 3; 1, 2, 1)$. For $k=2$, $(y, \l_2)\in \{(3, 4), (4, 3)\}$, and we get the rational 
points $(3, 3; 1, 4, 2)$ and $(3, 4; 1, 3, 2)$.

We summarize our classification result in the table below. The notation we use is the same as in \exref{Klein}.  
\[
\begin{array}{|c||c|c||c|c|c|}
\hline
\un{a}=(c_1, c_2, c_{12})&\alp_{f_1f_2}&(\l_1,\l_2)&\un{a}=(c_1, c_2, c_{12})&\alp_{f_1f_2}&(\l_1, \l_2)\\
\hline
       &\alp_{01}&(0, 3), (1, 3)& &\alp_{03}&(0, 1), (1, 1), (0, 3)\\
\cline{2-3}
       &\alp_{12}&(1, 0), (1, 3)& &         &(1, 3), (0, 5), (1, 5)\\ 
\cline{2-3}\cline{5-6}
(0, 0, 0) &\alp_{05}&(0, 3), (1, 3)&(0, 0, 0) &\alp_{10}&(1, 0), (1, 1), (1, 2)\\
\cline{2-3}
       &\alp_{11}&(0, 3), (1, 0)   & & &(1, 3), (1, 4), (1, 5)\\
\cline{2-3}\cline{5-6}			
			&\alp_{15}& (0, 3), (1, 0)&		&			\alp_{13}&(0, 1), (0, 3), (0, 5)\\	    
\cline{2-3}
			&\alp_{14}& (1, 0), (1, 3)& & & (1, 0), (1, 2), (1, 4)\\
\hline
(1, 0, 0)&\alp_{01}&(0, 3), (1, 3)&(1, 0, 0)&\alp_{03}&(0, 1), (0, 3), (0, 5)\\
\cline{2-3}
         &\alp_{05}&(0, 3), (1, 3)&     &     &(1, 1), (1, 3)\\
\hline
(0, 1, 1)&\alp_{10}&(1, 0), (1, 1), (1, 2)&(0, 1, 1)&\alp_{13}&(0, 0), (0, 2), (0, 4)\\
				 &         &(1, 3), (1, 4), (1, 5)&         &         &(1, 1), (1, 3), (1, 5)\\
\hline
      &\alp_{01}&(0, 3), (1, 3)& &\alp_{03}&(0, 1), (1, 1)\\
\cline{2-3}\cline{5-6}
(1, 0, 1)&\alp_{11}&(0, 1), (0, 2), (1, 4)&(1, 0, 1)&\alp_{13}&(0, 0), (0, 2)\\
\cline{2-3}
         &\alp_{15}&(0, 1)&    &   &(1, 1), (1, 3), (1, 5)\\
\hline
(1, 1, 0)&\alp_{13}&(0, 0), (0, 2), (0, 4)&(1, 2, 0)&\alp_{03}&(0, 0), (0, 2), (0, 4)\\
         &         &(1, 1), (1, 3), (1, 5)&         &         &(1, 0), (1, 2), (1, 4)\\
\hline
(1, 3, 0)&\alp_{13}&(0, 1), (0, 3), (0, 5)&(1, 3, 0)&\alp_{11}&(0, 1), (1, 4)\\
\cline{5-6}
         &         &(1, 0), (1, 2), (1, 4)&         &\alp_{15}&(0, 5), (1, 2)\\
\hline
(1, 4, 0)&\alp_{03}&(0, 1), (0, 3), (0, 5)&(1, 5, 0)&\alp_{13}&(0, 0), (0, 2), (0, 4)\\
         &         &(1, 1), (1, 3), (1, 5)&         &         &(1, 1), (1, 3), (1, 5)\\
\hline
         &\alp_{03}&(0, 0), (0, 2), (0, 4)&         &\alp_{03}&(0, 1), (0, 3), (0, 5)\\
(1, 2, 1)&         &(1, 0), (1, 2), (1, 4)&(1, 4, 1)&         &(1, 1), (1, 3), (1, 5)\\
\cline{2-3}\cline{5-6}
         &\alp_{13}&(0, 1), (0, 3), (0, 5)&         &\alp_{13}&(0, 0), (0, 2), (0, 4)\\
				 &         &(1, 0), (1, 2), (1, 4)&         &         &(1, 1), (1, 3), (1, 5)\\
\hline
         &\alp_{03}&(0, 0), (0, 2), (0, 4)&         &\alp_{10}&(1, 0), (1, 1), (1, 2)\\
(0, 2, 0)&         &(1, 0), (1, 2), (1, 4)&(0, 3, 0)&         &(1, 3), (1, 4), (1, 5)\\
\cline{2-3}\cline{5-6}
         &\alp_{13}&(0, 0), (0, 2), (0, 4)&         &\alp_{12}&(1, 2), (1, 5)\\
\cline{5-6}
         &         &(1, 1), (1, 3), (1, 5)&         &\alp_{14}&(1, 1), (1, 4)\\
\hline								
         &\alp_{03}&(0, 1), (0, 3), (0, 5)&         &\alp_{03}&(0, 0), (0, 2), (0, 4)\\
(0, 4, 0)&         &(1, 1), (1, 3), (1, 5)&(0, 2, 1)&         &(1, 0), (1, 2), (1, 4)\\
\cline{2-3}\cline{5-6}
         &\alp_{13}&(0, 1), (0, 3), (0, 5)&         &\alp_{10}&(1, 0), (1, 1), (1, 2)\\				
         &         &(1, 0), (1, 2), (1, 4)&         &         &(1, 3), (1, 4), (1, 5)\\
\hline
         &\alp_{10}&(1, 0), (1, 1), (1, 2)&         &\alp_{03}&(0, 1), (0, 3), (0, 5)\\
         &         &(1, 3), (1, 4), (1, 5)&(0, 4, 1)&         &(1, 1), (1, 3), (1, 5)\\
\cline{2-3}\cline{5-6}
         &\alp_{11}&(0, 2), (1, 5)        &         &\alp_{10}&(1, 0), (1, 1), (1, 2)\\
\cline{2-3}
         &\alp_{12}&(0, 0), (0, 3)        &         &         &(1, 3), (1, 4), (1, 5)\\
\cline{2-3}\cline{4-6}				
(0, 3, 1)&\alp_{13}&(0, 1), (0, 3), (0, 5)&         &\alp_{13}&(0, 0), (0, 2), (0, 4)\\
         &         &(1, 0), (1, 2), (1, 4)&(0, 5, 1)&         &(1, 1), (1, 3), (1, 5)\\
\cline{2-3}\cline{5-6}
				 &\alp_{14}&(1, 2), (1, 5)        &         &\alp_{10}&(1, 0), (1, 1), (1, 2)\\
\cline{2-3}				
         &\alp_{15}&(1, 3)                &         &         &(1, 3), (1, 4), (1, 5)\\
\hline
\end{array}~.
\]



\begin{thebibliography}{99}

\bibitem{abe}
E. Abe, "Hopf algebras", Cambridge: Cambridge University Press 1977.

\bibitem{a}
N. Andruskiewitsch, {\sl Pointed Hopf algebras}, “New directions in Hopf algebras”, MSRI series Cambridge Univ. Press (2002), 1–68.

\bibitem{an}
I. Angiono, {\sl Basic quasi-Hopf algebras over cyclic groups}, Adv. Math. {\bf 225} (2010), 3545-–3575.

\bibitem{ai}
I. Angiono, A.G. Iglesias, {\sl Pointed Hopf algebras: a guided tour to the liftings}, Revista Colombiana de Matem\'aticas {\bf 53} (2019), 1--44.







\bibitem{db} 
D. Bulacu, {\sl A structure theorem for quasi-Hopf bimodule coalgebras}, 
Theory and Applications of Categories (TAC) {\bf 32}(1) (2017), 1--30. 


\bibitem{bc1}
D. Bulacu, S. Caenepeel, {\sl Integrals for (dual) quasi-Hopf
algebras. Applications}. J. Algebra {\bf 266} (2003), 552--583.


\bibitem{bcpvo}
D. Bulacu, S. Caenepeel, F. Panaite, F. Van Oystaeyen, "Quasi-Hopf algebras: A monoidal approach", 
{\em Encyclopedia Math. Appl.} {\bf 171}, {\em Cambridge University Press} 2019.

\bibitem{bct}
D. Bulacu, S. Caenepeel, B. Torrecillas, {\sl The braided monoidal structures on the category of vector spaces graded by the Klein group}, 
Proc. Edinburgh Math. Soc. {\bf 54} (2011), 613--641.

\bibitem{bm}
D. Bulacu, M. Misurati, {\sl Quasi-Hopf algebras of dimension 6}, J. Pure Appl. Algebra {\bf 229} (2025), 107991.

\bibitem{bn}
D. Bulacu, E. Nauwelaerts, {\sl Radford's biproduct for quasi-Hopf algebras and bosonization}, J. Pure Appl. Algebra {\bf 174} (2002), 1--42.
 
\bibitem{bn3}
D. Bulacu, E. Nauwelaerts, {\sl Quasitriangular and ribbon quasi-Hopf algebras}. 
Comm. Algebra {\bf 31} (2003), 657--672.

\bibitem{bpv}
D. Bulacu, F. Panaite and F. Van Oystaeyen, {\sl Quasi-Hopf algebra actions and smash products}, Comm. Alg. {\bf 28} (2000), 631--651.





\bibitem{HCohen}
H. Cohen, "Number Theory Volume I: Tools and Diophantine Equations", {\em Graduate Texts in Mathematics} {\bf 239}, {\em New York: Springer-Verlag}, 2007. 

\bibitem{DNRha}
S. D\u asc\u alescu, C. N\u ast\u asescu, \c S. Raianu, "Hopf Algebras. An Introduction", {\em CRC Press}, Boca Raton, 2000.
		

\bibitem{dri87} 
V.G. Drinfeld, {\em Quantum groups}. In "Proc. Int. Cong. Math." (Berkeley, 1986), 798-820.
Amer. Math. Soc., Providence, RI, 1987.

\bibitem{dri}
V. G. Drinfeld, {\sl Quasi-Hopf algebras}. Leningrad Math. J. {\bf 1} (1990), 1419--1457.



\bibitem{eg3}
P. Etingof, S. Gelaki, {\sl Finite dimensional quasi-Hopf algebras with radical of codimension 2},  
Math. Res. Lett. {\bf 11} (2004), 685--696. 

\bibitem{eg5}
P. Etingof, S. Gelaki, {\sl On radically graded finite dimensional quasi-Hopf algebras}, Mosc. Math. J. {\bf 5} (2005), 371-–378.

\bibitem{eg4}
P. Etingof, S. Gelaki,
{\sl Liftings of graded quasi-Hopf algebras with radical of prime codimension},	J. Pure Appl. Algebra {\bf 205} (2006), 310--322.

\bibitem{egno}
P. Etingof, S. Gelaki, D. Nikshych, V. Ostrikr, "Tensor categories", Mathematical Surveys and Monographs {\bf 205}, American Mathematical Society 2016.

\bibitem{GrSo}
E.L. Green, $\emptyset.$ Solberg, {\sl Basic Hopf algebras and quantum groups}, Math. Z. {\bf 229}, 45-–76, 1998.

\bibitem{ham}
S. Hamada, {\sl On a free resolution of a dihedral group}, Tohoku Mathematical Journal {\bf 15}(3) (1963),  212--219.

\bibitem{han}
D. Handel, {\sl On products in the cohomology of the dihedral groups}, Tohoku Mathematical Journal {\bf 45}(1) (1993),  13--42. 

\bibitem{hnRMP}
F. Hausser and F. Nill, Diagonal crossed products by duals of quasi-quantum
groups. Rev. Math. Phys. 11, 553--629, 1999.


\bibitem{hlyy} 
HL. Huang, G. Liu, Y. Yang and Y. Ye,  {\sl On the Classification of Finite Quasi-Quantum Groups over abelian Groups}, arXiv:2403.04455, 2024. doi:10.48550/arXiv.2403.04455.

\bibitem{hlyy2}
HL. Huang, G. Liu, Y. Yang, Y. Ye, {\sl Finite quasi-quntum groups of diagonal type}. J. Reine. Angew. Math. 759 (2020), 201-243








\bibitem{m1}
S. Majid, {\sl Quantum Double for quasi-Hopf algebras}, Letters in Mathematical Physics {\bf 45}(1) (1998), 1--9.



\bibitem{maj}
S. Majid, "Foundations of quantum group theory", {\em Cambridge University Press}, 1995.

\bibitem{kmd}
K. Matthews, {\sl The Diophantine equation $x^2-Dy^2 = N$, $D>0$}, Expositiones Mathematicae {\bf 18} (2000), 323--331.

\bibitem{radproj}
D. E. Radford, {\sl The structure of Hopf algebras with a projection}, J. Algebra {\bf 92}(1985), 322-347. 

\bibitem{pan}
F. Panaite, {\sl A Maschke-type theorem for quasi-Hopf Algebras}, in "Rings, Hopf
Algebras and Brauer Groups", S. Caenepeel and A. Verschoren (eds.), Lecture Notes in Pure and Appl. Math., Marcel Dekker, 1998.

\bibitem {panfred}
F. Panaite, F. Van Oystaeyen,{\sl A structure theorem for quasi-Hopf comodule algebras},  
Proc. Amer. Math. Soc. {\bf 135}, 1669--1677, 2007.

\bibitem{rad}
D.E. Radford, {\sl The structure of Hopf algebras with a projection}, J. Algebra 
{\bf 92}, 322--347, 1985.

\bibitem{Rbke}
D. E. Radford, {\sl Biproducts and Kashina's Examples}, Comm. Alg. {\bf 44} (2015), 174--204.

\bibitem{rt}
D. E. Radford, {\sl Yetter-Drinfel’d categories associated to an arbitrary bialgebra}, J. Pure Appl. Algebra {\bf 87} (1993), 259--279. 




\bibitem{wp}
W. Propitius, {\sl Topological Interactions in Broken Gauge Theories}, PhD Thesis, University of Amsterdam; arXiv:hep-th/9511195v1 27 Nov 1995. 

\bibitem{y}
D. N. Yetter, {\sl Quantum groups and representations of monoidal categories}, Math. Proc. Cambridge Philos. Soc. {\bf 108} (1990), 261-–290.

\end{thebibliography}
\end{document}